\documentclass[12pt,leqno]{article}
\usepackage{amsmath}
\usepackage{amssymb}
\usepackage{theorem}
\usepackage{curves}
\usepackage{epic}
\usepackage{rotating}
\usepackage{epsf}
\usepackage[scanall]{psfrag}
\usepackage{graphicx}
\usepackage{esint} 

\usepackage{hyperref} 
\hypersetup{backref,  pdfpagemode=FullScreen,  colorlinks=true} 

\numberwithin{equation}{section}

\newtheorem{theorem}{Theorem}[section]
\newtheorem{proposition}[theorem]{Proposition}
\newtheorem{corollary}{Corollary}[section]
\newtheorem{lemma}{Lemma}[section]

\theorembodyfont{\rmfamily}

\newenvironment{proof}{\medskip\noindent{\bf Proof.}}{\medskip}

\newcommand{\R}{\mathbb{R}}

\renewcommand{\SS}{\mathbb {S}}

\renewcommand{\H}{\mathcal H}

\renewcommand{\O}{\Omega}
\newcommand{\p}{\partial}
\newcommand{\loc}{\mathrm{loc}}

\newcommand{\dist}{\,\mathrm{dist}\,}

\renewcommand{\a}{\alpha}

\renewcommand{\o}{\omega}

\newcommand{\sm}{\setminus}
\newcommand\qed{\hfill\vrule height8pt width6pt depth0pt}

\renewcommand{\i}{\subset}
\newcommand{\wt}{\widetilde}

\begin{document}

\title{Regularity of almost minimizers with free boundary}
\author{G.\ David\footnote{The first author was partially supported by NSF grant DMS 08-56687
and Institut Universitaire de France} 
\ and \ T.\ Toro\footnote{The second author was partially supported by NSF grant DMS 08-56687 and 
by a Simons Foundation fellowship.}}
\date{}
\maketitle

\begin{abstract}
In this paper we study the local regularity of almost minimizers of the functional 
\begin{equation*}
J(u)=\int_\O |\nabla u(x)|^2
+q^2_+(x)\chi_{\{u>0\}}(x) +q^2_-(x)\chi_{\{u<0\}}(x)
\end{equation*}
where $q_\pm \in L^\infty(\O)$.
Almost minimizers do not satisfy a PDE or a monotonicity formula 
like minimizers do (see \cite{AC}, \cite{ACF}, \cite{CJK}, \cite{W}). 
Nevertheless we succeed in proving that they are locally Lipschitz, which is the optimal regularity for minimizers.

\end{abstract}

\section{Introduction}

In this paper we consider a bounded domain $\O\subset\R^n$, 
$n \geq 2$, 
and study the local regularity of almost minimizers of the functional 
\begin{equation}\label{eqn1.1}
J(u)=\int_\O |\nabla u(x)|^2
+q^2_+(x)\chi_{\{u>0\}}(x) +q^2_-(x)\chi_{\{u<0\}}(x)
\end{equation}
where $q_\pm \in L^\infty(\O)$ are bounded real valued functions.

In \cite{AC}, Alt and Caffarelli proved existence and regularity results
for minimizers in the following context. 
Let $\O \i \R^n$ be a bounded  Lipschitz domain and $q_+ \in L^\infty(\O)$ 
be given, set 
\begin{equation} \label{eqn1.2}
K_+(\O) =\left\{u\in L^1_{\loc}(\O) \, ; \, u(x) \geq 0 \mbox{ almost everywhere on } \O
\mbox{ and } \nabla u\in L^2(\O) \right\}
\end{equation} 
and 
\begin{equation} \label{eqn1.3}
J^+(u)=\int_\O|\nabla u|^2+q^2_+(x)\chi_{\{u>0\}}
\end{equation}
for $u\in K_+(\O)$, and let $u_{0} \in K_{+}(\O)$ be given, 
with $J^+(u_0)<\infty$.
They prove the existence of a function
$u\in K_+(\O)$ that minimizes $J^+$ among functions of $K_+(\O)$ such that 
\begin{equation} \label{eqn1.4}
u = u_{0} \mbox{ on } \p \O.
\end{equation}
Notice that when $\O$ is Lipschitz and $u\in L^1_{\loc}(\O)$ is such that
$\nabla u\in L^2(\O)$,
we can define the trace of $u$ almost everywhere on $\p \O$, which lies in a slightly
better space than $L^2(\p\O)$, so (\ref{eqn1.4}) makes sense (\cite{D}).

They also showed that the minimizers are Lipschitz-continuous up 
to the free boundary $\p\{u>0\}$, and that if $q_+$ is 
H\"older-continuous and bounded away from zero, then
\begin{equation} \label{eqn1.5}
\p\{u>0\}=\p_\ast\{u>0\}\cup E, 
\end{equation}
where $\mathcal{H}^{n-1}(E)=0$ and $\p_\ast\{u>0\}$ 
is the reduced boundary of $\{x\in \Omega \, ; \, u(x)>0\}$ in $\Omega$;
in addition, $\p_\ast\{u>0\}$ locally coincides with a $C^{1,\alpha}$ submanifold of dimension $n-1$.

Alt, Caffarelli and Friedman \cite{ACF} later showed that if $\O$ is a bounded 
Lipschitz domain, $q_\pm\in L^\infty(\O)$, 
\begin{equation} \label{eqn1.6}
K(\O) = \left\{ u\in L^1_\loc(\O) \, ; \, \nabla u\in L^2(\O) \right\}
\end{equation}
and $u_0\in K(\O)$, then there exists $u\in K(\O)$
that minimizes $J(u)$ under the constraint (\ref{eqn1.4}).
See the proof of Theorem 1.1 in \cite{ACF}. 
In fact in \cite{ACF} they considered a slightly different functional,
for which they prove that the minimizers are Lipschitz. They also prove optimal regularity of the free boundary when $n=2$ 
and make important strides toward understanding the regularity of the free boundary in higher dimensions.  Later papers by \cite{CJK}, \cite{DeJ} and \cite{W}
present a more complete picture of the structure of the free boundary in higher dimensions.

\medskip
In this paper we study the regularity properties of the
almost minimizers for $J^+$ and $J$. We consider a 
domain $\O \i \R^n$ and two functions 
$q_\pm \in L^\infty(\O)$. 
We shall restrict to $n \geq 2$ to simplify the discussion, but
$n=1$ would be simpler.
We do not need any boundedness or regularity assumption on $\O$,
because our results will be local and so 
we do not need to define a trace on $\p \O$. 
Also, $q_-$ is not needed when we consider $J^+$, and then we may 
assume that it is identically zero. Set
\begin{equation} \label{eqn1.7}
K_{\loc}(\O) = \left\{u\in L^1_{\loc}(\O) \, ; \nabla u\in L^2(B(x,r))
\mbox{ for every open ball } B(x,r) \i \O \right\},
\end{equation}
\begin{equation} \label{eqn1.8}
K_{\loc}^+(\O) = \left\{u\in K_{\loc}(\O) \, ; u(x) \geq 0 
\mbox{ almost everywhere on } \O\right\}, 
\end{equation}
and let constants $\kappa \in (0,+\infty)$ and $\alpha \in (0,1]$
be given.

We say that $u$ is an almost minimizer for $J^+$ in $\O$ 
(with constant $\kappa$ and exponent $\alpha$) if
$u\in K_{\loc}^+(\O)$ and
\begin{equation}\label{eqn1.9}
J^+_{x,r}(u)\le (1+\kappa r^\alpha)J^+_{x,r}(v)
\end{equation}
for every ball $B(x,r)$ such that $\overline B(x,r) \i \O$ and every  
$v\in L^1(B(x,r))$ such that  $\nabla v\in L^2(B(x,r))$ 
and $v=u$ on $\p B(x,r)$, where
\begin{equation}\label{eqn1.10}
J^+_{x,r}(v)=\int_{B(x,r)}|\nabla v|^2+q^2_+ \, \chi_{\{v>0\}}.
\end{equation}
Here, when we say that $v=u$ on $\p B(x,r)$, we mean
that they have the same trace on $\p B(x,r)$, or equivalently
that their radial limits on $\p B(x,r)$, which happen to exist 
almost everywhere on $\p B(x,r)$, are equal almost everywhere 
on $\p B(x,r)$. See Section~13 of \cite{D}.
Note that we could easily restrict to nonnegative competitors $v$
(just because $v_{+} = \mathrm{max}(0,v)$ is at least as good as $v$).
Thus we can assume that $v\in K^+_{\loc}(\O)$.
Thus we would get an equivalent definition by considering competitors
$v\in K^+_{\loc}(\O)$ such that $v=u$ on $\Omega \sm B(x,r)$.

Similarly, we say that $u$ is an almost minimizer for $J$ in $\O$ 
if $u\in K_{\loc}(\O)$ and
\begin{equation}\label{eqn1.11}
J_{x,r}(u)\le (1+\kappa r^\alpha)J_{x,r}(v)
\end{equation}
for every ball $B(x,r) \i \O$ and every  $v\in L^1(B(x,r))$ such that  
$\nabla v\in L^2(B(x,r))$ and $v=u$ on $\p B(x,r)$, where
\begin{equation}\label{eqn1.12}
J_{x,r}(v)=\int_{B(x,r)}|\nabla v|^2
+q^2_+ \, \chi_{\{v>0\}}
+ q^2_- \, \chi_{\{v<0\}}.
\end{equation}
In effect, we would obtain essentially the same results
if we only required (\ref{eqn1.9}) or (\ref{eqn1.11}) when 
$r \leq r_{0}$. Potentially more interestingly, we could try to replace
$\kappa r^\alpha$ with larger gauge functions, for instance satisfying some Dini condition of some order. We will
not address this question here.

We try to follow the lead of \cite{AC} and \cite{ACF}, and
study the local regularity in $\O$ of almost minimizers $u$ for
$J$ and $J^+$. We shall show that they are locally Lipschitz in $\O$. 
Furthermore under the assumption that $q_{+}$ is bounded below way from zero, 
we shall prove 
a non-degeneracy condition for $u$ which translates into some 
regularity properties 
for the free boundary $\Gamma=\p(\{ x\in \Omega \,;\, u(x)>0\}$.
Although our results are similar to those 
in \cite{AC} and \cite{ACF}, we lack one of their major ingredients, that
is almost minimizers to not satisfy a differential equation. Our proofs are mostly 
based on choosing appropriate competitors for the almost minimizers.

In Section 2
we show that almost minimizers are locally continuous. The main tool in this section and in several other through the paper is to compare, in balls, the almost minimizer with its harmonic 
extension. In Section 3 we prove higher regularity for almost minimizers inside the open sets
$\big\{ u > 0 \big\}$ and $\big\{ u < 0 \big\}$. 
Once again this amounts to making careful local comparisons with the harmonic extension of the almost 
minimizer. In Section 4 we start the proof that almost minimizers are locally Lipschitz. 
In Sections 2, 3 and 4 we do not need to distinguish between almost minimizers for $J$ and $J^+$. 
In Section 5 we finish the proof of the fact that almost minimizers for $J^+$ 
are locally Lipschitz. 
In this section the fact that the almost minimizer does not change signs 
plays a key role. The proof of the fact that almost minimizers for $J$ are also locally Lipschitz 
requires, as in the minimizing case, an additional tool. In our case we need to prove 
an almost monotonicity formula. The proof of the almost monotonicity 
formula appears in Sections 6 and 7. In Section 8 we use this almost monotonicity to prove that almost minimizers for $J$ are locally Lipschitz. In Section 9, we study limits of sequences of almost 
minimizers, paying special attention to blow-up sequences. 
These sequences play an important role when trying to understand the fine properties and the structure of the free boundary. In Section
10 we prove non-degeneracy results for almost minimizers under mild assumptions for the functions $q_\pm$. These results, in particular, ensure that the blow-up limits 
are non-trivial and that therefore 
they yield information on the free boundary. In Section 11 we briefly summarize what we know thus far about the free boundary. Understanding the structure and the regularity of the free boundary for 
almost minimizers is the subject of our current research.

\section{Almost minimizers are continuous in $\O$.}

In this section we do not distinguish between $J^+$ and $J$. 
We will write $J$ in the proofs with the understanding that 
$q_-$ might be identically zero and that we may be working 
with nonnegative functions.

\begin{theorem}\label{thm2.1} 
Almost minimizers of $J$ are continuous in $\O$. 
Moreover if $u$ is an almost minimizer for $J$ there
exists a constant $C>0$ such that  if $B(x_0,2r_0)\subset \O$ then for $x,y\in B(x_0,r_0)$
\begin{equation}\label{eqn2.1}
|u(x)-u(y)|\le C|x-y| \, \big(1+ \log \frac{2r_0}{|x-y|} \big).
\end{equation}
\end{theorem}

\begin{proof}
Let $u$ be an almost minimizer for $J$, and let $x\in \O$ and $r>0$
be such that $B(x,r)\subset\O$. For the moment, and up to (\ref{eqn2.13}) (included), 
we need no other assumption on $(x,r)$.

For $s\le r$ let $u_s^\ast$ denote the function in $L^1(B(x,s))$
with $\nabla u^\ast_s \in L^2(B(x,s))$ and trace $u$ on
$\p B(x,s)$ which minimizes the Dirichlet energy on
$B(x,s)$.  The existence and uniqueness of $u^\ast_s$ are easy to
obtain, by convexity, and we shall often refer to $u^\ast_s$ 
as the harmonic extension of the restrtiction of $u$ to $\p B(x,s)$. 
Note that this is the case if $u$ is smooth enough on $\p B(x,s)$. 
In the present context, the minimality property is just easier to
work with.
Many of the estimates in this paper will come from a 
comparison with functions like $u^{\ast}_{s}$.
By definition, 
for any $t\in\R$
\begin{equation}\label{eqn2.2}
\int_{B(x,s)} |\nabla u^\ast_s|^2 \leq \int_{B(x,s)} |\nabla (u^\ast_s
+ t (u-u^\ast_s))|^2.
\end{equation}

Expanding near $t=0$ yields
$\int_{B(x,s)}\langle\nabla u- \nabla u^\ast_s,\nabla u^\ast_s\rangle = 0$
and hence
\begin{equation}\label{eqn2.3}
\int_{B(x,s)}|\nabla u^\ast_s|^2
= \int_{B(x,s)}\langle\nabla u,\nabla u^\ast_s\rangle.
\end{equation}

Since $u$ is an almost minimizer and $q_\pm\in L^\infty$,
(\ref{eqn2.3}) yields
\begin{eqnarray}\label{eqn2.4}
\int_{B(x,s)}|\nabla u-\nabla u^\ast_s|^2 & = & \int_{B(x,s)}|\nabla
u|^2-\int_{B(x,s)}|\nabla u^\ast_s|^2 
\nonumber\\
& \le & (1+\kappa s^\alpha)\int_{B(x,s)}|\nabla u^\ast_s|^2 -
\int_{B(x,s)}|\nabla u^\ast_s|^2 + C s^n 
\\
& \le & \kappa s^\alpha\int_{B(x,s)}|\nabla u^\ast_s|^2 + C s^n 
 \le  \kappa s^\alpha\int_{B(x,s)}|\nabla u|^2+ C s^n,
\nonumber
\end{eqnarray}
where in the last inequality we used again the fact that $u^\ast_s$ 
is an energy minimizer. Here $C\ge 0$ depends on 
$\|q_\pm\|_{L^\infty}$. 
Thus (\ref{eqn2.4}) applied to $B(x,r)$ yields
\begin{equation}
\label{eqn2.5}
\int_{B(x,r)}|\nabla u-\nabla u^\ast_r|^2 
 \le  \kappa r^\alpha \int_{B(x,r)}|\nabla u^\ast_r|^2+Cr^n  \\
\leq \kappa r^\alpha \int_{B(x,r)}|\nabla u|^2+Cr^n.
\end{equation}
For $s> 0$ such that $B(x,s) \i \O$, we set
\begin{equation}\label{eqn2.6}
\o(x,s) = \left(\fint_{B(x,s)} |\nabla u|^2\right)^{\frac{1}{2}}
= \left(\frac{1}{|B(x,s)|} \int_{B(x,s)} |\nabla u|^2\right)^{\frac{1}{2}}.
\end{equation}
Since $u^\ast_r$ is an energy minimizer, it is harmonic in $B(x,r)$.
This is easy to see. Then
$|\nabla u^\ast_r|^2$ is subharmonic, and therefore for $s\le r$
\begin{equation}\label{eqn2.7}
\left(\fint_{B(x,s)}|\nabla u^\ast_r|^2\right)^{\frac{1}{2}} \le
\left(\fint_{B(x,r)}|\nabla u^\ast_r|^2\right)^{\frac{1}{2}}.
\end{equation}
Combining the triangle inequality in $L^2$, 
(\ref{eqn2.5}), (\ref{eqn2.6}), 
and (\ref{eqn2.7}) we obtain
\begin{eqnarray}
\label{eqn2.8}
\o(x,s) & \le &
\left(\fint_{B(x,s)}|\nabla u-\nabla u^\ast_r|^2\right)^{\frac{1}{2}} 
+ \left(\fint_{B(x,s)}|\nabla u^\ast_r|^2\right)^{\frac{1}{2}} 
\nonumber \\ 
& = & Cs^{-n/2}
\left(\int_{B(x,s)}|\nabla u-\nabla u^\ast_r|^2\right)^{\frac{1}{2}} 
 + \left(\fint_{B(x,s)}|\nabla u^\ast_r|^2\right)^{\frac{1}{2}} 
\nonumber \\
& \le & Cs^{-n/2}
\left(\kappa r^\alpha \int_{B(x,r)}|\nabla u|^2+ r^n\right)^{\frac{1}{2}} 
 + \left(\fint_{B(x,s)}|\nabla u^\ast_r|^2\right)^{\frac{1}{2}}
 \\
& \le & C\left(\frac{r}{s}\right)^{n/2}r^{\a/2}\o(x,r) +
C\left(\frac{r}{s}\right)^{n/2} + \left(\fint_{B(x,s)}|\nabla
u^\ast_r|^2\right)^{\frac{1}{2}} 
\nonumber \\
& \le & C\left(\frac{r}{s}\right)^{n/2}r^{\a/2}\o(x,r) +
C\left(\frac{r}{s}\right)^{n/2} + \left(\fint_{B(x,r)}|\nabla
u^\ast_r|^2\right)^{\frac{1}{2}} 
\nonumber \\
& \le & \left(1+C\left(\frac{r}{s}\right)^{n/2}r^{\a/2}\right)\o(x,r)
+ C\left(\frac{r}{s}\right)^{n/2}. \nonumber
\end{eqnarray}
Here $C$ denotes a constant that only depends on 
$\kappa$, $\|q_\pm\|_{L^\infty}$, and $n$. 

Set  $r_j=2^{-j} r$ for $j \geq 0$. By (\ref{eqn2.8}),
\begin{equation}
\label{eqn2.9}
\o(x,r_{j+1}) \le
\left(1+C2^{n/2}r_j^{\alpha/2}\right)\o(x,r_j)+C2^{n/2} 
\end{equation}
and an iteration yields
\begin{eqnarray}
\label{eqn2.10}
\o(x,r_{j+1}) & \le &
\o(x,r)
\prod^j_{l=0}\left(1+C2^{n/2}r_l^{\alpha/2}\right)  
\nonumber \\
&& \qquad + C\sum^{j+1}_{l=1}
\left(\prod^j_{k=l}\left(1+C2^{n/2}r_k^{\alpha/2}\right)\right)2^{n/2}
\\
& \le & 
\o(x,r) P + C P 2^{n/2}  j 
\leq C\o(x,r) + C j, \nonumber
\end{eqnarray}
where we set 
$P =  \prod^{\infty}_{j=0}\left(1+C2^{n/2}r_{j}^{\alpha/2}\right)
=\prod^{\infty}_{j=0}\left(1+C2^{n/2}(2^{-j}r)^{\alpha/2}\right)$,
and use the fact that $P$ can be bounded, depending on an upper bound for $r$. 

At this point we have proved that if $B(x,r) \i \O$,
then for $0 < s \leq r$,
\begin{equation}
\label{eqn2.11}
\o(x,s) \le C\o(x,r) + C \log(r/s),
\end{equation}
with a constant $C$ that depends only on 
$\kappa$, $\|q_+\|_{L^\infty}$, $\|q_-\|_{L^\infty}$, 
$\alpha$, $n$, and an upper bound for $r$. 
Indeed, if $s \geq r/4$, just observe that
$\o(x,s) \le 2^n \o(x,r)$ by (\ref{eqn2.6}), and otherwise
choose $j$ such that $r_{j+2} \leq s \leq r_{j+1}$,
observe that $\o(x,s) \le 2^{n/2} \o(x,r_{j+1})$,
and use (\ref{eqn2.10}).

We now return to (\ref{eqn2.1}). Set $u_j=\fint_{B(x,r_j)}u$; 
Poincar\'e's inequality and (\ref{eqn2.10}) yield
\begin{eqnarray}
\label{eqn2.12}
\left(\fint_{B(x,r_j)}|u-u_j|^2\right)^{\frac{1}{2}} & \le &
Cr_j\o(x,r_j) \le  Cr_j\o(x,r)+C j r_j.
\end{eqnarray}

Suppose in addition that $x$ is a Lebesgue point for $u$; then
$u(x) = \lim_{l \to \infty} u_l$, and 
\begin{eqnarray}
\label{eqn2.13}
|u(x)-u_j| 
& \le & \sum^{\infty}_{l=j} |u_{l+1}-u_l| 
 \le  \sum^\infty_{l=j}\fint_{B(x,r_{l+1})}  |u-u_l|  
\nonumber \\
&\le & 2^n\sum^\infty_{l=j} \fint_{B(x,r_l)} |u-u_l| 
 \le  2^n \sum^\infty_{l=j}\left(\fint_{B(x,r_l)}
|u-u_l|^2\right)^{\frac{1}{2}} 
\\
& \le & C\sum^\infty_{l=j}r_l\left(\o(x,r)+l\right) 
\le C r_{j} (\o(x,r) + j + 1)  \nonumber
\end{eqnarray}
because $\sum^\infty_{l=j} 2^{j-l} \frac{l}{j+1}  \leq C$.

We are now ready to prove Theorem \ref{thm2.1}. 
Let $x_0$, $r_{0}$, $x$, and $y$ be as in the statement.
It is enough to prove (\ref{eqn2.1}) 
when $x$ and $y$ are Lebesgue points for $u$
(for the other points, we would use the estimates that we have on the Lebesgue set 
to modify $u$ on a set of measure zero and get an equivalent continuous function).

We may even assume that $|x-y| \leq r_{0}/8$; otherwise, use a few 
intermediate points. From (\ref{eqn2.13}) with $r = |x-y|$ and $j=0$ we deduce that
\begin{equation}
\label{eqn2.14}
\Big| u(x)-\fint_{B(x,r)}u \Big| \leq C r (\o(x,r) + 1).
\end{equation}
Similarly,
\begin{equation}
\label{eqn2.15}
\Big| u(y)-\fint_{B(y,r)}u \Big| \leq C r (\o(y,r) + 1).
\end{equation}
Then set $z = \frac{x+y}{2}$ and $m = \fint_{B(z,2r)} u$; 
by Poincar\'e, Cauchy-Schwarz, and (\ref{eqn2.6}),
\begin{eqnarray}
\label{eqn2.16}
\Big |m- \fint_{B(x,r)}u \Big| 
&\leq & 
\fint_{B(x,r)} |u - m| \leq 2^n \fint_{B(z,2r)} |u - m| 
\leq C r \fint_{B(z,2r)} |\nabla u| \\
&\leq & C r \left( \fint_{B(z,2r)} |\nabla u|^2 \right)^{\frac{1}{2}}
= C r \o(z,2r).
\nonumber
\end{eqnarray}
Similarly, 
\begin{equation}
\label{eqn2.17}
\Big | m -\fint_{B(y,r)}u \Big |
\leq C r \left( \fint_{B(z,2r)} |\nabla u|^2 \right)^{\frac{1}{2}}
= C r \o(z,2r).
\end{equation}
Altogether,
\begin{equation}
\label{eqn2.18}
|u(x)-u(y)| \leq C r (\o(x,r) + \o(y,r) + 1 + \o(z,2r))
\leq C r (1 + \o(z,2r))
\end{equation}
by (\ref{eqn2.14})-(\ref{eqn2.17}) and because 
$B(x,r) \cup B(y,r) \i B(z,2r)$. 

Let $j$ be such that $2^{-j-3} r_{0} \leq r \leq 2^{-j-2} r_{0}$,
and apply (\ref{eqn2.10}) to the pair $(z,r_{0}/2)$. Notice that
$B(z,r_{0}/2) \i B(x_{0},3r_{0}/2) \i \O$ by assumption, so 
(\ref{eqn2.10}) holds. We get that
\begin{equation}
\label{eqn2.19}
\o(z,2r) \leq 2^{n/2} \o(z,2^{-j-1} r_{0})
\leq C \o(z,r_{0}/2) + C j
\leq C \o(x_{0},3r_{0}/2) + C \log\frac{r_{0}}{r}
\end{equation}
(recall that $r = |x-y| \leq r_{0}/8$). Now
(\ref{eqn2.18}) shows that 
\begin{equation}
\label{eqn2.20}
|u(x)-u(y)| 
\leq C |x-y| \left(\o(x_{0},3r_{0}/2) + \log\frac{r_{0}}{|x-y|}\right).
\end{equation}
This holds under the assumptions of Theorem \ref{thm2.1},
for $x,y \in B(x_{0},r_{0})$ such that $|x-y| \leq r_{0}/8$,
 with a constant $C$ that depends only on 
$\kappa$, $\|q_+\|_{L^\infty}$, $\|q_-\|_{L^\infty}$, 
$\alpha$, $n$, and an upper bound for $r$. 
Obviously (\ref{eqn2.1}) and Theorem \ref{thm2.1} follow.
\qed
\end{proof}
\bigskip

Theorem \ref{thm2.1} has the following immediate consequence.

\begin{corollary}\label{cor2.1}
If $u$ is an almost minimizer for $J$, then for each compact set
$K\subset\O$ there is a constant $C_K>0$ such that for $x, y\in K$
\begin{equation}\label{eqn2.21}
|u(x)-u(y)|\le C_K|x-y|\left(1+\left|\log
\frac{1}{|x-y|}\right|\right).
\end{equation}
\end{corollary}

\section{Almost minimizers are $C^{1,\beta}$ in $\{u>0\}$ and
in $\{u<0\}$}

In this section again, we do not distinguish between $J^+$ and $J$; we 
write $J$ in the proofs with the understanding that $q_-$ might be
identically zero and the functions nonnegative.
We address the regularity of an almost minimizer
$u$, far from the zero set $\{ u=0 \}$. 
Our estimates will depend on the distance to the zero set. 
We start with Lipschitz bounds; see Theorem \ref{thm3.2}
for $C^{1,\beta}$ bounds.  

\begin{theorem}\label{thm3.1}
Let $u$ be an almost minimizer for $J$ in $\O$. Then $u$ is locally
Lipschitz in $\{u>0\}$ and in $\{u<0\}$.
\end{theorem}

\begin{proof}
We start with $\{u>0\}$, and assume that $B(x_0,3r_0)\i \{u>0\} \i\O$. 
For $x\in B(x_0,r_0)$ and $r \leq r_0$, denote by $u^\ast_{r}$ 
the function with the same trace (or radial limit 
almost everywhere) as $u$ on $\p B(x,r)$ and 
which minimizes the Dirichlet energy under this constraint. 
First we address a minor technical issue.

\medskip\noindent{\bf Remark 3.1}  
Let us check that in the present case, since $u$
is continuous on $\p B(x,r)$ because of Theorem \ref{thm2.1},
$u^\ast_{r}$ is also the harmonic extension of the 
restrition of $u$ to $\p B(x,r)$, obtained by convolution
with the Poisson kernel. Indeed, first observe that $u$
is bounded on $\p B(x,r)$, and hence $u^\ast_{r}$ is bounded
in $B(x,r)$ (otherwise truncate and get a better  competitor).
It is also harmonic in $B(x,r)$, by a standard variational argument.
Next let $y\in B(x,r)$ be given, and for $|y-x| < t < r$, write
$u^\ast_{r}(y)$ as the Poisson integral of $u^\ast_{r}$ on
$\p B(x,t)$. That is, $u^\ast_{r}(y) = \int_{ \p B(x,t)} P_{t}(y,z) 
u^\ast_{r}(z) dz$. Then let $t$ tend to $r$; with $y$ fixed,
the Poisson kernel stays bounded, and so does $u^\ast_{r}(z)$;
then we can use the fact that $u^\ast_{r}$ has radial limits
equal to those of $u$ almost everywhere, the dominated 
convergence theorem, and the fact that $P_{t}(y,\cdot)$
tends to $P_{r}(y,\cdot)$ radially, to conclude that
$u^\ast_{r}(y) = \int_{ \p B(x,r)} P_{r}(y,z) u^\ast_{r}(z) dz$,
as needed. 

\medskip
We return to the proof of Theorem \ref{thm3.1}.
Since $u$ is an almost minimizer we have
\begin{equation}
\label{eqn3.1}
J_{x,r}(u)\le (1+\kappa r^\alpha)J_{x,r}(u_r^\ast) 
\end{equation}
as in (\ref{eqn1.9}) or (\ref{eqn1.11}).
Here $u > 0$ on $\overline B(x,r)$, and by the maximum principle 
$u^\ast_r > 0$ on $\overline B(x,r)$ also; by (\ref{eqn1.10}) or
(\ref{eqn1.12}),  (\ref{eqn3.1}) becomes
\begin{equation}
\label{eqn3.2a}
\int_{B(x,r)}\left(|\nabla u|^2 + q^2_+ (x)\right)dx 
\le (1+\kappa r^\alpha) \int_{B(x,r)} 
\left(|\nabla u^\ast_{r}|^2 + q^2_+(x)\right) dx.
\end{equation}
Thus
\begin{eqnarray}
\label{eqn3.2}
\int_{B(x,r)}|\nabla u|^2 
& \le & (1+\kappa r^\alpha)
\int_{B(x,r)}|\nabla u^\ast_r|^2 + \kappa r^\alpha
\int_{B(x,r)} q^2_+(x)dx \\
& \le & (1+\kappa r^\alpha)\int_{B(x,r)} |\nabla u^\ast_r|^2 +
Cr^{\a+n}, \nonumber
\end{eqnarray}
where $C=\kappa\|q_+\|^2_{L^\infty}\o_n$ and $\o_n$ denotes the measure of the unit ball in $\R^n$.
By (\ref{eqn2.3}),
$\int_{B(x,r)}|\nabla u^\ast_r|^2
=\int_{B(x,r)}\langle\nabla u,\nabla u^\ast_r\rangle$, 
so (\ref{eqn3.2}) yields
\begin{eqnarray}
\label{eqn3.3}
\int_{B(x,r)}|\nabla u-\nabla u^\ast_r|^2 
& = & \int_{B(x,r)}|\nabla u|^2 +\int_{B(x,r)} |\nabla u^\ast_r|^2 
- 2\int_{B(x,r)}\langle \nabla u, \nabla u^\ast_r\rangle 
\nonumber \\
& = & 
\int_{B(x,r)}|\nabla u|^2-\int_{B(x,r)}|\nabla u^{\ast}_r|^2 
\\ 
&\leq  &
\kappa r^\a \int_{B(x,r)}|\nabla u^\ast_r|^2 +Cr^{\a+n}
\leq \kappa r^\a \int_{B(x,r)}|\nabla u|^2 +Cr^{\a+n},
\nonumber 
\end{eqnarray}
by the minimizing property of $u^\ast_r$.

Define $\o(x,s)$, for $0< s\le r$, as in (\ref{eqn2.6}).
The triangle inequality, (\ref{eqn2.7}), and (\ref{eqn3.3}) yield, 
as for (\ref{eqn2.8}), 
\begin{eqnarray}
\label{eqn3.5}
\o(x,s) & \le & 
\left(\fint_{B(x,s)}|\nabla u-\nabla u^\ast_r|^2\right)^{\frac{1}{2}} 
+ \left(\fint_{B(x,s)}|\nabla u^\ast_r|^2\right)^{\frac{1}{2}} 
\nonumber \\
& \le & 
\left(\frac{r}{s}\right)^{\frac{n}{2}} 
\left(\fint_{B(x,r)}|\nabla u-\nabla u^\ast_r|^2\right)^{\frac{1}{2}} 
+ \left(\fint_{B(x,r)}|\nabla u^\ast_r|^2\right)^{\frac{1}{2}} 
\\
& \le & 
\left(\frac{r}{s}\right)^{\frac{n}{2}} \kappa^{1/2} r^{\alpha/2} 
\left(\fint_{B(x,r)}|\nabla u|^2\right)^{\frac{1}{2}} 
+C  \left(\frac{r}{s}\right)^{\frac{n}{2}} r^{\alpha/2} 
+  \left(\fint_{B(x,r)}|\nabla u|^2\right)^{\frac{1}{2}} 
\nonumber \\
& \le & 
\left(1+C\left(\frac{r}{s}\right)^{\frac{n}{2}}r^{\a/2}\right)\o(x,r)
+ C\left(\frac{r}{s}\right)^{\frac{n}{2}} r^{\a/2},
\nonumber
\end{eqnarray}
with $C=C(\kappa,\|q_+\|_\infty)$. 
Then set $r_{j} = 2^{-j} r$
for $j \geq 0$, and apply (\ref{eqn3.5}) repeatedly;
we obtain (as in as in (\ref{eqn2.9}) and as in (\ref{eqn2.10}))
that 
\begin{eqnarray}
\label{eqn3.6}
\o(x,r_{j+1}) 
& \le &
\left( 1+C 2^{n/2} r_{j}^{\alpha/2}\right) \o(x,r_j) +
C2^{n/2} r_{j}^{\alpha/2}  
\nonumber\\
& \le &
\o(x,r)\prod^j_{l=0}\left(1+C2^{n/2} r^{\alpha/2}_l\right)
+ C \sum^{j+1}_{l=1}\left(\prod^j_{k=l} 
\left(1+C2^{n/2} r^{\alpha/2}_k\right)\right)
2^{n/2} r^{\alpha/2}_{l-1}  .
\end{eqnarray}
Again $\prod^{\infty}_{l=0}\left(1+C2^{n/2} r^{\alpha/2}_l\right) \leq C$,
where $C$  depends on an upper bound for $r$, and (\ref{eqn3.6}) 
yields
\begin{equation}
\label{eqn3.7}
\o(x,r_{j+1}) \leq C\o(x,r) +
C 2^{n/2} \sum^{j+1}_{l=1} r^{\alpha/2}_{l-1}
\le C\o(x,r)+Cr^{\alpha/2}.
\end{equation}
From this it follows, using (\ref{eqn3.7}) with $j$ such that $r_{j+1} < s \leq r_{j}$, that
\begin{equation}
\label{eqn3.8}
\o(x,s) \leq  C\o(x,r)+Cr^{\alpha/2}
\ \hbox{ for } 0 < s \leq r.
\end{equation}

Let us continue as in (\ref{eqn2.12}) and (\ref{eqn2.13}),
but with (\ref{eqn2.10}) replaced by the better estimate 
(\ref{eqn3.7}).
Set $u_j=\fint_{B(x,r_j)}u$; by Poincar\'e and (\ref{eqn3.7}),
\begin{eqnarray}
\label{eqn3.9}
\left(\fint_{B(x,r_j)}|u-u_j|^2\right)^{\frac{1}{2}} & \le &
Cr_j \o(x,r_j) \le  Cr_j \left(\o(x,r)+r^{\alpha/2} \right).
\end{eqnarray}
Now we know that $x$ is a Lebesgue point for $u$,
because Theorem \ref{thm2.1} says that $u$ is continuous,
and
\begin{eqnarray}
\label{eqn3.10}
|u(x)-u_j| 
& \le & 
2^n \sum^\infty_{l=j}\left(\fint_{B(x,r_l)}
|u-u_l|^2\right)^{\frac{1}{2}} 
\leq C\sum^\infty_{l=j} r_l \left(\o(x,r)+ r^{\alpha/2}\right) 
\\
& \le & 
C r_{j} \left(\o(x,r)+r^{\alpha/2} \right)  \nonumber
\end{eqnarray}
as in (\ref{eqn2.13}) and by (\ref{eqn3.7}).

All this holds for $x\in B(x_{0},r_{0})$ and $0 < r \leq r_{0}$.
Now let $y\in B(x_{0},r_{0})$ be such that $|x-y| \leq r_{0}$.
Set $r = |x-y|$; by (\ref{eqn3.10}) with $j=0$,
\begin{equation}
\label{eqn3.11}
\left| u(x)-\fint_{B(x,r)}u \right| \leq C r (\o(x,r) + r^{\alpha/2})
\end{equation}
and similarly
\begin{equation}
\label{eqn3.12}
\left| u(y)-\fint_{B(y,r)}u \right| \leq C r (\o(y,r) + r^{\alpha/2}).
\end{equation}
Again set $z = \frac{x+y}{2}$ and $m = \fint_{B(z,2r)} u$, and
notice that $B(z,2r) \i B(x_{0}, 3r_{0})$.
By Poincar\'e, Cauchy-Schwarz, and (\ref{eqn2.6}),
and exactly as in (\ref{eqn2.16}), we have
\begin{eqnarray}
\label{eqn3.13}
\left |m- \fint_{B(x,r)}u \right| 
&\leq & 
\fint_{B(x,r)} |u - m| \leq 2^n \fint_{B(z,2r)} |u - m| 
\leq C r \fint_{B(z,2r)} |\nabla u| \\
&\leq & C r \left( \fint_{B(z,2r)} |\nabla u|^2 \right)^{\frac{1}{2}}
= C r \o(z,2r).
\nonumber
\end{eqnarray}
Similarly, 
\begin{equation}
\label{eqn3.14}
\left| m -\fint_{B(y,r)}u \right|
\leq C r \o(z,2r).
\end{equation}
Thus (\ref{eqn3.11}) to (\ref{eqn3.14}) yield
\begin{equation}
\label{eqn3.15}
|u(x)-u(y)| \leq C r (\o(x,r) + \o(y,r) + r^{\alpha/2} + \o(z,2r))
\leq C r (\o(z,2r) + r^{\alpha/2})
\end{equation}
because $B(x,r) \cup B(y,r) \i B(z,2r)$. Let us check that
\begin{equation}
\label{eqn3.16}
\o(z,2r) \leq C (\o(x_{0},3r_{0}) + r_{0}^{\alpha/2}).
\end{equation}
If $r \leq r_{0}/2$, this follows from
(\ref{eqn3.8}), applied to the pair $(z,r_{0})$ and $s=2r$ 
(recall that $z\in B(x_{0},r_{0})$); otherwise,  for $r\ge r_0/2$, 
$\o(z,2r) \leq C \o(x_{0},3r_{0})$ since 
$B(z,2r) \i B(x_{0},3r_{0})$ because $r = |x-y| \leq r_{0}$.
Thus we proved that 
\begin{equation}
\label{eqn3.17}
|u(x)-u(y)| \leq  C  |x-y| (\o(x_{0},3r_{0}) + r_{0}^{\alpha/2})
\end{equation}
whenever $B(x_{0},3r_{0})  \i \{ u > 0 \}$ and 
$x, y \in B(x_{0},r_{0})$ are such that $|x-y| \leq r_{0}$,
and where the constant $C$ depends on $\kappa$, 
$\|q_+\|_\infty$ and
an upper bound for $r_{0}$ (because of the transition from
(\ref{eqn3.6}) to (\ref{eqn3.7})). For the record, notice that
this implies that when $B(x_{0},3r_{0})  \i \{ u > 0 \}$,
\begin{equation}
\label{eqn3.18}
|\nabla u(x)| \leq  C (\o(x_{0},3r_{0}) + r_{0}^{\alpha/2})
\ \hbox{ for almost every } x \in B(x_{0},r_{0}).
\end{equation}

All this gives local Lipschitz bounds for $u$ in $\{ u > 0 \}$.
The local Lipschitz bounds in $\{ u < 0 \}$ are handled
exactly the same way; we prove that (\ref{eqn3.17}) holds when
$B(x_{0},3r_{0})  \i \{ u < 0 \} \i \O$ and 
$x, y \in B(x_{0},r_{0})$ are such that $|x-y| \leq r_{0}$,
where now the constant $C$ depends on $\kappa$, 
an upper bound for $r_{0}$, $\kappa$, 
and $\|q_-\|_\infty$ rather than $\|q_+\|_\infty$. This 
second case is only relevant 
when we work with $J$ rather than $J^+$. Theorem \ref{thm3.1} follows.
\qed
\end{proof}

Let us now improve Theorem \ref{thm3.1} slighty, and prove
local $C^{1,\beta}$ estimates for $u$ far from its zero set.

\begin{theorem}\label{thm3.2} 
Let $u$ be an almost minimizer for $J$ in $\O$, and set
$\beta = {\alpha \over n+2+\alpha}$.
Then  $u$ is $C^{1,\beta}$ locally in $\{u>0\}$ and in $\{u<0\}$.
\end{theorem}

\begin{proof}
Of course we do not claim that this value of $\beta$ is optimal.
First assume that $B(x_0,4r_0)\subset \{u>0\} \i \O$, and
for $x\in B(x_0, r_{0})$ and $0 < r \leq r_{0}$, let us compare again 
with $u^\ast_r$, the harmonic extension of the restriction of 
$u$ to $\p B(x,r)$

Let $\tau\in \left(0,\frac{1}{2}\right)$ be small,
to be chosen later. Also set $v(x,r) = \fint_{B(x,r)} \nabla u^\ast_r$;
then by (\ref{eqn3.3})
\begin{eqnarray}
\label{eqn3.19}
\int_{B(x,\tau r)} |\nabla u-v(x,r)|^2 
& \le &
2\int_{B(x,\tau r)} |\nabla u-\nabla u^\ast_r|^2 +
2\int_{B(x,\tau r)} \left|\nabla u^\ast_r- v(x,r) \right|^2 \nonumber \\
& \le & Cr^\a \int_{B(x,r)} |\nabla u|^2  +Cr^{\alpha+n} +
2\int_{B(x,\tau r)} |\nabla u^\ast_r - v(x,r)|^2 .
\end{eqnarray}
By the mean value theorem and because $u^\ast_r$ is
harmonic in $B(x,r)$, $v(x,r) = \nabla u^\ast_r (x)$. In addition,
standard estimates for harmonic functions (use Remark 3.1 and
write $\nabla u^\ast_r$ as the integral of its Poisson kernel) 
yield for $y\in B(x,\tau r)$
\begin{eqnarray}
\label{eqn3.20}
|\nabla u_r^\ast(y) - v(x,r)| 
& = & |\nabla u_r^\ast(y) - \nabla u^\ast_r(x)|  
\leq \tau r \sup_{B(x,\tau r)}|\nabla^2 u^\ast_r|  
\nonumber \\
&\leq &  
C \tau \left(\fint_{B(x,r)} |\nabla u^\ast_r|\right)   
\leq C\tau \left(\fint_{B(x,r)} |\nabla u^\ast_r|^2\right)^{\frac{1}{2}}
\\
& \le & C\tau \left(\fint_{B(x,r)} |\nabla u|^2\right)^{\frac{1}{2}}
= C\tau \o(x,r)
\nonumber
\end{eqnarray}
because $u^\ast_r$ is harmonic and energy minimizing.
Combining (\ref{eqn3.19}) and (\ref{eqn3.20}) we obtain
\begin{eqnarray}
\label{eqn3.21}
\left(\fint_{B(x,\tau r)} |\nabla u-v(x,r)|^2 \right)^\frac{1}{2}
&\le & C \tau^{-n/2} r^{\alpha/2} \o(x,r)  +C \tau^{-n/2} r^{\alpha/2} +
\sqrt{2}\left( \fint_{B(x,\tau r)} |\nabla u^\ast_r - v(x,r)|^2 \right)^\frac{1}{2}
\nonumber \\
&\le &
C \tau^{-n/2} r^{\alpha/2} \o(x,r)  +C \tau^{-n/2} r^{\alpha/2} + C\tau \o(x,r).
\end{eqnarray}

By Theorem \ref{thm3.1}, $u$ is locally Lipschitz. More precisely,
notice that for $x\in B(x_{0},r_{0})$, $B(x,3r_{0}) \i \{u>0\} \i \O$,
so (\ref{eqn3.18}) holds for $B(x,3r_{0})$ and yields
\begin{equation}
\label{eqn3.22}
|\nabla u(y)| \leq  C (\o(x,3r_{0}) + r_{0}^{\alpha/2})
\leq C (\o(x_{0},4r_{0}) + r_{0}^{\alpha/2})
\end{equation}
for almost every $y \in B(x, r_{0})$.
Hence $\o(x,r) \leq C (\o(x_{0},4r_{0}) + r_{0}^{\alpha/2})$
for $0 < r \leq r_{0}$, and (\ref{eqn3.21}) yields
\begin{equation}
\label{eqn3.23}
\fint_{B(x,\tau r)}|\nabla u-v(x,r)|^2 
\leq C_0^2(r^\a \tau^{-n}+\tau^2) 
\end{equation}
where $C_0 = C (\o(x_{0},4r_{0}) + r_{0}^{\alpha/2}+1)$ 
depends on $\kappa$, 
$\|q_+\|_\infty$, $\o(x_0, 4r_0)$, and 
an upper bound for $r_0$. 

We want to apply this with $\tau=r^{\frac{\a}{n+2}}$;
since we want $\tau < 1/2$, we only do this for $r$ small
(precisely, so small that $r^{\frac{\a}{n+2}}<\frac{1}{2}$).
Then set $\rho=\tau r=r^{1+\frac{\a}{n+2}}= r^{\frac{n+2+\a}{n+2}}$,
and observe that 
$r^\a\tau^{-n} =  \tau^2 = r^{\frac{2\a}{n+2}} = \rho^{\frac{2\a}{n+2+\a}}$.
Also set $\beta=\frac{\a}{n+2+\a}$ and
$m(x,\rho) =  \fint_{B(x,\rho)} \nabla u$; then (\ref{eqn3.23}) yields
\begin{equation}
\label{eqn3.24}
\fint_{B(x,\rho)} \left|\nabla u-m(x,\rho)\right|^2 
\le \fint_{B(x,\rho)} \big|\nabla u-v(x,r) \big|^2 
\le 2C_0\rho^{2\beta}.
\end{equation}
Note that (\ref{eqn3.24}) holds for all $x\in B(x_{0},r_{0})$ (as above)
and $0 < \rho \leq \rho_{0}$, where $\rho_{0}$ is chosen so that
if $\rho \leq \rho_{0}$ and if we set $r = \rho^{\frac{n+2}{n+2+\a}}$,
then $r \leq r_{0}$ and $r^{\frac{\a}{n+2}} < \frac{1}{2}$.
In other words, we pick $\rho_{0}$ such that
$\rho_{0}^{\frac{n+2}{n+2+\a}} \leq r_{0}$ and
$\rho_{0}^{\frac{n+2}{n+2+\a}\frac{\a}{n+2}} < \frac{1}{2}$.

It follows from the triangle inequality, Cauchy-Schwarz, and (\ref{eqn3.24}), 
that for $0 < \rho \leq \rho_{0}$,
\begin{eqnarray}
\label{eqn3.25}
|m(x,\rho/2)-m(x,\rho)|
& = & \left|\fint_{B(x,\rho/2)} \nabla u-m(x,\rho)\right|
\leq 2^n \fint_{B(x,\rho)} \left|\nabla u-m(x,\rho)\right|
\nonumber \\
&\leq & 2^n 
\left(\fint_{B(x,\rho)} \left|\nabla u-m(x,\rho)\right|^2 \right)^{\frac{1}{2}}
\le C C_0 
\rho^{\beta};
\end{eqnarray}
then using a standard argument we get that if $x$
is a Lebesgue point for $\nabla u$,
\begin{equation}\label{eqn3.25A}
|\nabla u(x) - m(x,\rho)| \leq C C_{0} \rho^{\beta}.
\end{equation}

Now let $x$, $y$ be Lebesgue points for $\nabla u$ such that
$x, y \in B(x_{0},r_{0})$ and $|x-y| \leq \rho_{0}/2$.
Set $\rho = 2 |x-y|$, and observe that by (\ref{eqn3.24}),
\begin{eqnarray}
\label{eqn3.26}
| m(y,\rho/2)- m(x,\rho)|
&\leq & \fint_{B(y,\rho/2)} \left|\nabla u-m(x,\rho)\right|
\\
&\leq & 2^n \fint_{B(x,\rho)} \left|\nabla u-m(x,\rho)\right|
\le CC_0^{\frac{1}{2}}\rho^{\beta}
\nonumber 
\end{eqnarray}
because $B(y,\rho/2) \i B(x,\rho)$.
Since $| \nabla u(y) - m(y,\rho/2)| + | \nabla u(x) - m(y,\rho/2)| 
\leq C C_{0} \rho^{\beta}$ by (\ref{eqn3.25A}), we get that
\begin{equation}
\label{eqn3.27}
| \nabla u(y) - \nabla u(x) | \leq C C_{0} \rho^{\beta} \leq  C C_{0} |x-y| ^{\beta}.
\end{equation}
Thus (\ref{eqn3.27}) holds for all Lebesgue points 
$x \in B(x_{0},r_{0})$ and $y \in B(x_{0},r_{0})$ 
such that $|x-y| \leq \rho_{0}/2$, and $B(x_{0},4r_{0}) \i \{ u > 0 \}$.
Redefining $\nabla u$ on the remaining set of measure zero we have that
(\ref{eqn3.27}) holds everywhere.

Now if $x,y \in B(x_{0},r_{0})$ are such that $|x-y| > \rho_{0}/2$,
we can connect $x$ to $y$ through less than $4\frac{r_{0}}{\rho_{0} }$ 
intermediate points, and we get that
\begin{equation}
\label{eqn3.28}
| \nabla u(y) - \nabla u(x) | 
\leq C C_{0} \, \frac{r_{0}} {\rho_{0}} \, |x-y| ^{\beta}.
\end{equation}
This gives the local $C^{1,\beta}$-regularity of $u$ on $\{ u > 0 \}$.
The argument for $\{ u < 0 \}$ is the same, and Theorem \ref{thm3.2}
follows.
\qed
\end{proof}

\section{First estimates for the local Lipschitz regularity.}  

We now focus on the local Lipschitz estimates for $u$
when $u$ is an almost minimizer for $J$ or for $J^+$.
The difference with the previous section is that now we also 
consider balls that meet $\{ u > 0 \}$, $\{ u < 0 \}$, and the zero set
of $u$. Our general strategy is to show that 
the quantity $\o(x,r)$ (defined in (\ref{eqn2.6})) 
is bounded on compact sets, and 
does not become too large when $r$ gets small.

In this section we shall prove a few lemmas that work equally 
well for almost minimizers for $J$ or $J^+$, so we continue 
with our convention that we do not distinguish between $J$ 
and $J^+$. For the moment we shall only prove 
local Lipschitz bounds under an additional condition 
(namely, that $(x,r)\in {\cal G}(\tau, C_{0}, C_{1}, r_0)$; see below), 
which we get rid of in later sections.

In addition to 
$\o(x,r) = \left(\fint_{B(x,r)} |\nabla u|^2\right)^{\frac{1}{2}}$
(see (\ref{eqn2.6})), we shall often use the quantities
\begin{equation}
\label{eqn4.1}
b(x,r) = \fint_{\p B(x,r)} u
\ \hbox{ and } \  
b^+(x,r) = \fint_{\p B(x,r)} |u|
\end{equation}
which are well defined when $\overline B(x,r) \i \O$, because 
Theorem \ref{thm2.1} says that $u$ is continuous on $\p B(x,r)$.

For each choice of constants $\tau \in (0,10^{-2})$, $C_{0} \geq 1$, 
and $C_{1} \geq 3$, and $r_{0} > 0$ we introduce the class
of pairs $(x,r) \in \O \times (0,r_{0}]$ such that $B(x,2r) \i \O$,
\begin{equation}
\label{eqn4.2}
r^{-1} | b(x,r)|  
\geq  C_0\tau^{-n}\left(1+r^\a\o(x,r)^2\right)^{\frac{1}{2}} ,
\end{equation} 
and 
\begin{equation}
\label{eqn4.3}
b^+(x,r) \leq C_1 |b(x,r)|.
\end{equation} 
We denote by ${\cal G}(\tau, C_{0}, C_{1}, r_{0})$ this class of pairs,
i.e., set
\begin{equation}
\label{eqn4.4}
{\cal G}(\tau, C_{0}, C_{1}, r_{0}) = \big\{ (x,r) \in \O \times (0,r_{0}] \, ; \,
B(x,2r) \i \O  \hbox{ and (\ref{eqn4.2}) and (\ref{eqn4.3}) hold} \big\}.
\end{equation} 
Notice that ${\cal G}(\tau, C_{0}, C_{1}, r_{0})$ really 
depends on $C_0\tau^{-n}$, rather than $C_{0}$ and $\tau$, 
but it will be more convenient to use the definition this way.
We see (\ref{eqn4.2}) and  as (\ref{eqn4.3}) nice properties, 
because they will help us control places where $u$ and  
the harmonic extension of $u_{\big| \p B(x,r)}$ have a given sign.
In the mean time we start with a self-improvement lemma. 

\begin{lemma}
\label{lem4.1}
Assume that  $u$ is an almost minimizer for $J$ in $\O$.
For each choice of constants $C_{1} \geq 3$ and $r_{0}$,  
there is a constant $\tau_{1} \in (0, 10^{-2})$, that depends 
only on $n$, $\kappa$, $\alpha$, $r_{0} > 0$, and $C_{1}$,
such that if $(x,r)\in {\cal G}(\tau, C_{0}, C_{1}, r_{0})$ 
for some choice of $\tau\in(0,\tau_1)$ and $C_{0} \geq 1$, 
then  for each $z\in B\left(x,\frac{\tau r}{2}\right)$ 
we can find $\rho_z\in \left(\frac{\tau r}{4}, \frac{\tau r}{2}\right)$
such that $(z,\rho_{z})\in {\cal G}(\tau, 10C_{0}, 3, r_{0})$. 
\end{lemma}

\begin{proof}
Let $(x,r)\in {\cal G}(\tau, C_{0}, C_{1}, r_{0})$ be as in the statement, 
and as usual denote by $u_r^\ast$ the harmonic extension of
(or initially the smallest energy $W^{1,2}$-function with the same trace as)
the restriction of $u$ to $\p B(x,r)$. Since $0 < \tau \leq 10^{-2}$, 
standard estimates on harmonic functions ensure that
\begin{equation}
\label{eqn4.5}
\sup_{B(x,\tau r)} |\nabla u^\ast_r|
\le \frac{C}{r}\sup_{\p B \left(x,r/2 \right)} |u^\ast_r| 
\le \frac{C}{r} \fint_{\p B(x, r)}|u^\ast_r| = \frac{C}{r} \, b^+(x,r)
\end{equation}
because $u^\ast_r=u$ on $\p B(x,r)$ and by (\ref{eqn4.1}).
Hence, for $y\in B(x,\tau r)$, we deduce from (\ref{eqn4.1})
and our assumption (\ref{eqn4.3}) that
\begin{eqnarray}
\label{eqn4.6}
|u^\ast_r(y)-b(x,r)| & = & |u^\ast_r(y)-\fint_{\p B(x, r)} u^\ast_r| 
= |u^\ast_r(y)-u^\ast_r(x)|
\\
& \leq & \tau r \sup_{B(x,\tau r)} |\nabla u^\ast_r|
\leq C \tau \, b^+(x,r)
\leq C C_{1} \tau |b(x,r)|.
\nonumber
\end{eqnarray}
We choose $\tau_{1}$ so small that $C C_{1} \tau_{1} \leq \frac{1}{4}$; 
since here $\tau \leq \tau_1$, (\ref{eqn4.6}) yields
\begin{equation}
\label{eqn4.7}
|u^\ast_r(y)-b(x,r)| \leq C C_{1} \tau |b(x,r)|
\leq \frac{1}{4} |b(x,r)|
\ \hbox{ for } y\in B(x,\tau r).
\end{equation}
Recall that (\ref{eqn2.4}) holds as soon as $B(x,r) \i \O$
and $0 < s \leq r$; we take $s =  r$ and get that
\begin{equation}
\label{eqn4.8}
\fint_{B(x,r)}|\nabla u-\nabla u^\ast_r|^2 
\leq \kappa r^\a \fint_{B(x,r)}|\nabla u|^2  + C 
= \kappa r^\a\o(x,r)^2 + C.
\end{equation}
Then, by Poincar\'e's inequality,
\begin{equation}
\label{eqn4.9}
\fint_{B(x,r)}|u-u^\ast_r|^2 
\leq Cr^2\fint_{B(x,r)}|\nabla u-\nabla u^\ast_r|^2
\leq Cr^2 \left(r^\a\o(x,r)^2+1\right)
\end{equation}
and by Cauchy-Schwarz
\begin{equation}
\label{eqn4.10}
\fint_{B(x, \tau r)} |u-u^\ast_r| 
\leq \left( \fint_{B(x,\tau r)}|u-u^\ast_r|^2 \right)^{\frac{1}{2}}
\le C\tau^{-n/2}r\left(r^\a\o(x,r)^2+1\right)^{\frac{1}{2}}.
\end{equation}
Let $z\in B\left(x,\frac{\tau r}{2}\right)$ be given; since
\begin{equation}
\label{eqn4.11}
\int_{B(z,\frac{\tau r}{2})}|u-u^\ast_r| 
= \int^{\frac{\tau r}{2}}_0 \int_{\p B(z,s)}|u-u^\ast_r|
\end{equation}
there exists $\rho_{z} \in \left(\frac{\tau r}{4}, \frac{\tau r}{2}\right)$ 
such that
\begin{eqnarray}
\label{eqn4.12}
\int_{\p B(z,\rho_z)}|u-u^\ast_r|
&\leq &\frac{4}{\tau r} \int_{B\left(z,\frac{\tau r}{2}\right)} |u-u^\ast_r| 
\leq \frac{4}{\tau r} \int_{B(x,\tau r)} |u-u^\ast_r| 
\nonumber\\
& = & C (\tau r)^{n-1} \fint_{B(x,\tau r)} |u-u^\ast_r|
\leq C (\tau r)^{n-1} \tau^{-n/2} r\left(r^\a\o(x,r)^2+1\right)^{\frac{1}{2}}.
\end{eqnarray} 
because $B\left(z,\frac{\tau r}{2}\right) \i B(x,\tau r)$
and by (\ref{eqn4.10}).

Set $b^\ast = \fint_{\p B(z,\rho_z)} u^\ast_r$. Then 
 by (\ref{eqn4.1}), (\ref{eqn4.12}), and our assumption (\ref{eqn4.2})
\begin{eqnarray}
\label{eqn4.13}
\left|  b(z,\rho_{z}) - b^\ast \right| 
& = &
\left|  \fint_{\p B(z,\rho_z)} u - \fint_{\p B(z,\rho_z)} u^\ast_r  \right| 
\leq \fint_{\p B(z,\rho_z)}|u-u^\ast_r|
 \\
& \le & C\tau^{-\frac{n}{2}}r\left(1+r^\a\o(x,r)^2\right)^{\frac{1}{2}}
\leq \frac{C\tau^{-\frac{n}{2}}} {C_{0}\tau^{-n}} \, |b(x,r)|.
\nonumber 
\end{eqnarray}
Recall that $\tau \leq \tau_1$ and $C_{0} \geq 1$;
if $\tau_1$ is chosen small enough, then
$\frac{C\tau^{-\frac{n}{2}}} {C_{0}\tau^{-n}} 
\leq C \tau^{\frac{n}{2}} \leq \frac{1}{4}$,
and (\ref{eqn4.13}) says that
\begin{equation}
\label{eqn4.14}
\left|  b(z,\rho_{z}) - b^\ast \right| 
\leq \fint_{\p B(z,\rho_z)}|u-u^\ast_r|
\leq \frac{1}{4} \, |b(x,r)|.
\end{equation}
Since $\p B(z,\rho_z) \i B(x,\tau r)$ by (\ref{eqn4.7}) we have
\begin{equation}
\label{eqn4.15}
\left|  b^\ast - b(x,r)\right| 
= \left|\fint_{\p B(z,\rho_z)} (u^\ast_r - b(x,r))\right|
\leq \fint_{\p B(z,\rho_z)}|u^\ast_r - b(x,r)|
\leq \frac{1}{4} \, |b(x,r)|.
\end{equation}
Combining (\ref{eqn4.14}) and (\ref{eqn4.15}) we get
\begin{equation}
\label{eqn4.16}
\left|  b(z,\rho_{z}) - b(x,r) \right| 
\leq \left|  b(z,\rho_{z}) - b^\ast \right| + \left|  b^\ast - b(x,r)\right| 
\leq \frac{1}{2} \, |b(x,r)|
\end{equation}
and hence
\begin{equation}
\label{eqn4.17}
|b(z,\rho_{z})| \geq |b(x,r)| - \left|  b(z,\rho_{z}) - b(x,r) \right|
\geq \frac{1}{2}  |b(x,r)|.
\end{equation}

We can also control $b^+(z,\rho_{z})$. Recall that $b(x,r) \neq 0$
by (\ref{eqn4.2}); by (\ref{eqn4.7}), $u^\ast_r$ keeps the same
sign as $b(x,r)$ on $\p B(z,\rho_z) \i B(x,\tau r)$. Then
$ \fint_{\p B(z,\rho_z)} |u^\ast_r| = 
\big| \fint_{\p B(z,\rho_z)} u^\ast_r \big| = \left| b^\ast \right|$.
Hence by (\ref{eqn4.1}), (\ref{eqn4.14}), and (\ref{eqn4.15})
\begin{eqnarray}
\label{eqn4.18}
b^+(z,\rho_z) &=& \fint_{\p B(z,\rho_z)} |u|
\leq \fint_{\p B(z,\rho_z)} |u^\ast_r| + \fint_{\p B(z,\rho_z)} |u-u^\ast_r|
\\
&\leq & \fint_{\p B(z,\rho_z)} |u^\ast_r| + \frac{1}{4}  |b(x,r)|
= \left| b^\ast \right|+ \frac{1}{4}  |b(x,r)|
\leq \frac{3}{2}  |b(x,r)|.
\nonumber
\end{eqnarray}
This and (\ref{eqn4.17}) imply that $b^+(z,\rho_z) \leq 3 |b(z,\rho_z)|$
which is (\ref{eqn4.3}) with $C_{1} = 3$.

\medskip
We still need to check that $(z,\rho_{z})$ satisfies (\ref{eqn4.2}).
By (\ref{eqn4.17}) and (\ref{eqn4.2}),
\begin{equation}
\label{eqn4.19}
|b(z,\rho_{z})| \geq \frac{1}{2}  |b(x,r)|
\geq \frac{C_0\tau^{-n}}{2} \, r \left(1+r^\a\o(x,r)^2\right)^{\frac{1}{2}}.
\end{equation}

We want a similar estimate, but in terms
of $\o(z,\rho_{z})$, so we need a reasonable upper
bound for $\o(z,\rho_{z})$. Let $j \geq 0$ be such that 
$2^{-j-1} \leq \tau \leq 2^{-j}$, and set $r_{j } = 2^{-j} r$ 
as before. Observe that (\ref{eqn2.10}) applies, 
just because it is valid as soon as $B(x,r) \i \O$. It yields
\begin{equation}
\label{eqn4.20}
\o(x,r_{j}) \leq C \o(x,r) + C j \le C\o(x,r) + C(1+|\log\tau|).
\end{equation}
Here $C$ depends on $n$, $\kappa$, $\alpha$, and our
upper bound $r_{0}$ for $r$.
In addition, $B(z,\rho_{z}) \i B(x, \tau r) \i B(x, r_{j})$
by definitions, so
\begin{equation}
\label{eqn4.21}
\o(z,\rho_{z}) \leq \left(\frac{r_{j}}{\rho_{z}}\right)^{\frac{n}{2}} \o(x,r_{j})
\leq 8^{\frac{n}{2}} \o(x,r_{j}) \leq C\o(x,r) + C(1+|\log\tau|).
\end{equation}
Now 
\begin{eqnarray}
\label{eqn4.22}
1 + \rho_{z}^\a\o(z,\rho_{z})^2
&\leq & 1 + C \rho_{z}^\a\o(x,r)^2 + C\rho_{z}^\a(1+|\log\tau|)^2
\nonumber \\
&\leq & 1 + C r^\a\o(x,r)^2 + C(\tau r)^\a (1+|\log\tau|)^2
\nonumber \\
&\leq & 1 + C r^\a\o(x,r)^2 + Cr^\alpha \big[\tau^\alpha(1+|\log\tau|)^2 \big]
 \\
&\leq & C (1+r^\a\o(x,r)^2).
\nonumber
\end{eqnarray}
By (\ref{eqn4.19}), (\ref{eqn4.22}), and because $r \ge 2 \tau^{-1}\rho_{z}$,
\begin{eqnarray}
\label{eqn4.23}
|b(z,\rho_{z})| 
&\geq & \frac{C_0\tau^{-n}}{2} \, r \left(1+r^\a\o(x,r)^2\right)^{\frac{1}{2}}
\nonumber\\
&\geq & C^{-1} C_0 \tau^{-n}  \, (2\tau^{-1}\rho_{z}) 
 \left(1+\rho_{z}^\a\o(z,\rho_{z})^2\right)^{\frac{1}{2}}
 \\
& = & (C\tau)^{-1}  \,  C_0 \tau^{-n}  \, \rho_{z}
 \left(1+\rho_{z}^\a\o(z,\rho_{z})^2\right)^{\frac{1}{2}}.
 \nonumber
\end{eqnarray}
We choose $\tau_{1}$ so large that $(C\tau)^{-1}  \geq 10$
in (\ref{eqn4.23}), and this gives (\ref{eqn4.2}) with the constant $10C_{0}$.
This completes our verification that 
$(z,\rho_{z})\in {\cal G}(\tau, 10C_{0}, 3, r_{0})$.
 Lemma \ref{lem4.1} follows.
\qed
\end{proof}

\begin{lemma}
\label{lem4.2}
Let $u$, $x$, and $r$ satisfy the assumptions of Lemma \ref{lem4.1}, 
so that in particular $(x,r)\in {\cal G}(\tau, C_{0}, C_{1}, r_{0})$
for some $C_{0} \geq 1$ and $\tau \leq \tau_{1}$.
Recall that $b(x,r) \neq 0$ by (\ref{eqn4.2}).
If $b(x,r) > 0$, then
\begin{equation}
\label{eqn4.24}
u \geq 0 \hbox{ on $B(x,\tau r/2)$ and 
$u > 0$ almost everywhere on $B(x,\tau r/2)$.}
\end{equation}
If instead $b(x,r) < 0$, then
\begin{equation}
\label{eqn4.25}
u \leq 0 \hbox{ on $B(x,\tau r/2)$ and 
$u < 0$ almost everywhere on $B(x,\tau r/2)$.}
\end{equation} 
\end{lemma}

\begin{proof}
In this proof we apply Lemma \ref{lem4.1} repeatedly.
Let $u$, $x$, and $r$ be as in the statement, and
let $z\in B(x,\tau r/2)$ be given. Apply Lemma \ref{lem4.1} 
a first time. This gives a radius $\rho_{z} \in( \tau r/4,\tau t/2)$ 
such that $(z,\rho_{z})\in {\cal G}(\tau, 10C_{0}, 3, r_{0})$.
Set $\rho_{0} = \rho_{z}$ to unify the notation below.

Then iterate, but this time systematically apply 
Lemma \ref{lem4.1} with $x$ and $z$ both equal to the recently given $z$.
This gives a sequence of
radii $\rho_{j}$, $j \geq 0$, such that for $j \geq 0$
\begin{equation}
\label{eqn4.26}
(z,\rho_{j})\in {\cal G}(\tau, 10^j C_{0}, 3, r_{0})
\end{equation} 
which implies by (\ref{eqn4.2}) that
\begin{equation}
\label{eqn4.28}
\rho_{j}^{-1} | b(z,\rho_{j})|  \geq  10^j C_0\tau^{-n}
\left(1+\rho_{j}^\a\o(z,\rho_{j})^2\right)^{\frac{1}{2}}.
\end{equation}
Moreover by construction, we also get that for $j \geq 0$
\begin{equation}
\label{eqn4.27}
\frac{\tau \rho_{j}}{4} \leq \rho_{j+1} \leq \frac{\tau \rho_{j}}{2}.
\end{equation}

Let $u^\ast_j$ denote 
the energy-minimizing 
extension of the restriction of $u$ to $\p B(z,\rho_{j})$.  
This is the function that we call $u^\ast_r$ 
when we prove Lemma \ref{lem4.1} with $x = z$ and $r=\rho_{j}$.
By (\ref{eqn4.7}) in this context,
\begin{equation}
\label{eqn4.29}
|u^\ast_j(y)-b(z,\rho_{j})| \leq \frac{1}{4} |b(z,\rho_{j})|
\ \hbox{ for } y\in B(z,\tau \rho_{j}).
\end{equation}

Note that (\ref{eqn4.16}) ensures that $b(z,\rho_{z})$ is not zero, 
and that it has the same sign as $b(x,r)$. 
In addition, $b(z,\rho_j)\not = 0$ by (\ref{eqn4.28}).
Thus by (\ref{eqn4.16}) and 
an easy induction argument we conclude that $b(z,\rho_{j})$ has the same sign 
for all $j$. Set
\begin{equation}
\label{eqn4.30}
Z_{j} = \big\{ y\in B(z,\tau\rho_{j}) \, ; \, u(y)b(x,r) \leq 0 \big\}
= \big\{ y\in B(z,\tau\rho_{j}) \, ; \, u(y)b(z,\rho_{j}) \leq 0 \big\};
\end{equation}
this is just the subset of $B(z,\tau\rho_{j})$ where $u$ does not
have the right sign. Notice that for $y\in Z_{j}$,
\begin{equation}
\label{eqn4.31}
|u(y)-u^\ast_j(y)| \geq |u(y)-b(z,\rho_{j})|-|b(z,\rho_{j})-u^\ast_j(y)|
\geq \frac{3}{4} |b(z,\rho_{j})|
\end{equation}
because $|u(y)-b(z,\rho_{j})| \geq |b(z,\rho_{j})|$ and by (\ref{eqn4.29}).
Then by Chebyshev, the Lebesgue measure of $Z_{j}$ is
\begin{equation}
\label{eqn4.32}
m(Z_{j}) \leq \frac{4}{3 |b(z,\rho_{j})|} \int_{B(z,\tau\rho_j)}|u-u^\ast_j|.
\end{equation}
But by (\ref{eqn4.10}) 
\begin{equation}
\label{eqn4.33}
\int_{B(z, \tau \rho_j)} |u-u^\ast_j|  
\le C(\tau \rho_{j})^n \tau^{-n/2}\rho_{j}
\left(1 + \rho_{j}^\a\o(z,\rho_{j})^2\right)^{\frac{1}{2}}.
\end{equation}
We now combine (\ref{eqn4.32}), (\ref{eqn4.33}), and (\ref{eqn4.28})
to get that
\begin{equation}
\label{eqn4.34}
m(Z_{j}) \leq C \big[10^j C_0\tau^{-n} \rho_{j}\big]^{-1}
(\tau \rho_{j})^n \tau^{-n/2}\rho_{j}
= C \big[10^j C_0\big]^{-1}  \tau^{3n/2} \rho_{j}^{n}.
\end{equation}

Let us assume that $b(x,r) > 0$, and thus $b(z,\rho_{j}) > 0$
for $j\geq 0$, to simplify the discussion.
We just proved that for every $z\in B(x,\tau r/2)$,
\begin{equation}
\label{eqn4.35}
\lim_{j\to\infty} \frac{m\left(\{u \leq 0\}\cap B(z, \tau \rho_j)\right)}
{m\left(B(z,\tau\rho_j)\right)}
= \lim_{j\to\infty} \frac{m\left(Z_{j}\right)}
{m\left(B(z,\tau\rho_j)\right)} =0.
\end{equation}
Now (\ref{eqn4.35}) fails when $z$ is a Lebesgue point of 
$\{u \leq 0\}$, which is the case for almost every $z \in \{u \leq 0\}$.
Hence $m(\{u \leq 0\} \cap B(x,\tau r/2)) = 0$, that is
$u(x) > 0$ almost everywhere on $B(x,\tau r/2)$, and
(\ref{eqn4.24}) holds (recall that $u$ is continuous).

The case when $b(x,r) < 0$ is dealt with in a
similar fashion;  Lemma \ref{lem4.2} follows.
\qed
\end{proof}

\begin{lemma}\label{lem4.3}
Let $u$ be an almost minimizer for $J$ in $\O$, and 
let $x$ and $r$ satisfy the assumptions of Lemma \ref{lem4.1}, 
so that in particular $(x,r)\in {\cal G}(\tau, C_{0}, C_{1}, r_{0})$
for some $C_{0} \geq 1$ and $\tau \leq \tau_{1}$.
Then for $z\in B\left(x, \frac{\tau r}{4}\right)$ 
and $\rho\in\left(0,\frac{\tau r}{8}\right)$ 
\begin{equation}
\label{eqn4.36}
\o(z,\rho) \le C \left(\tau^{-\frac{n}{2}} \o(x,r)+r^{\frac{\a}{2}}\right).
\end{equation}
Moreover for $y,z\in B\left(x, \frac{\tau r}{4}\right)$ 
\begin{equation}
\label{eqn4.37}
|u(y)-u(z)| \leq C\left(\tau^{-\frac{n}{2}}\o(x,r) 
+ r^{\frac{\a}{2}}\right) \, |y-z|.
\end{equation} 
Here $C$ depends on $n$, $\kappa$, $\alpha$, and $r_{0}$.
Finally, there is a constant $C(\tau,r)$, that depends on 
$n$, $\kappa$, $\alpha$, $r_{0}$, $\tau$, and $r$, such that
\begin{equation}
\label{eqn4.38}
|\nabla u(y)-\nabla u(z)|\le C(\tau,r) 
\left(\o(x,r) + 1\right) \, |y-z|^\beta ,
\end{equation} 
for $y, z \in B\left(x, \frac{\tau r}{8}\right)$
and where $\beta = \frac{\alpha}{n+2+\alpha}$ as in
Theorem \ref{thm3.2}.
\end{lemma}

\begin{proof}
Let $u$, $x$, and $r$ be as in the statement.
Also let $z\in B\left(x, \frac{\tau r}{4}\right)$ be given.
By the proof of Lemma~\ref{lem4.1}, we can find 
$\rho\in \left(\frac{\tau r}{8}, \frac{\tau r}{4}\right)$
such that $(z,\rho)\in {\cal G}(\tau, 10C_{0}, 3, r_{0})$.
Notice that we made $\rho$ smaller here, but the proof of 
Lemma~\ref{lem4.1} still gives this, maybe at the expense
of making $\tau_{1}$ a little smaller; we could also get
a pair $(z,\rho)$ with a smaller $\rho$, by applying
Lemma~\ref{lem4.1} a second time to the pair $(z,\rho_{z})$
given by Lemma~\ref{lem4.1}, but this would be a little clumsier
and would give slightly worse estimates.

Let $u^\ast_\rho$ denote the usual energy-minimizing 
extension of the restriction of $u$ to $\p B(z,\rho)$. 
Since $u$ is an almost minimizer,
\begin{equation}
\label{eqn4.39}
J_{z,\rho}(u) \leq (1+\kappa \rho^\alpha) J_{z,\rho}(u^\ast_\rho)
\end{equation}
as in (\ref{eqn1.11}), and where $J_{z,\rho}$ is as in (\ref{eqn1.12}). 
Recall that here we shall just take $q_- = 0$ if we work with $J^+$;
see (\ref{eqn1.9}) and (\ref{eqn1.10}).

Let us assume that $b(x,r) > 0$; the other case would be similar.
By Lemma~\ref{lem4.2}, $u \geq 0$ everywhere and
$u > 0$ almost everywhere on $B(x,\tau r /2)$,
as in (\ref{eqn4.24}). We just made sure to chose $\rho$ so that 
$B(z,\rho) \i B(x,\tau r /2)$, so this happens on $B(z,\rho)$ too.
Then (\ref{eqn1.12}) yields
\begin{equation}
\label{eqn4.40}
J_{z,\rho}(u) = \int_{B(z,\rho)} |\nabla u|^2  + q^2_+ \, \chi_{\{u>0\}}  + q^2_- \, \chi_{\{u<0\}} 
= \int_{B(z,\rho)} |\nabla u|^2 + q^2_+ 
\end{equation}
(the term with $\chi_{\{u<0\}}$ disappears as $u\ge 0$ in $B(z,\rho)$). 
Since $u^\ast_\rho$ is harmonic in $B(z,\rho)$ and $u^\ast_\rho = u$
on $\p B(z,\rho)$ the
maximum principle ensures that 
 $u^\ast_\rho\ge 0$ in $B(z,\rho)$.
 Moreover
\begin{equation}
\label{eqn4.41}
J_{z,\rho}(u^\ast_\rho) =
\int_{B(z,\rho)} |\nabla u^\ast_\rho|^2 
+ q^2_+ \, \chi_{\{u^\ast_\rho>0\}} 
\leq \int_{B(z,\rho)} |\nabla u^\ast_\rho|^2 + q^2_+ .
\end{equation}
So $\int_{B(z,\rho)}  q^2_+$ in (\ref{eqn4.39}) cancels partially,
and we get that
\begin{eqnarray}
\label{eqn4.42}
\int_{B(z,\rho)} |\nabla u|^2 
&\leq& (1+\kappa \rho^\alpha) \int_{B(z,\rho)} |\nabla u^\ast_\rho|^2
+ \kappa \rho^\alpha \int_{B(z,\rho)} q^2_+
\\
&\leq& (1+\kappa \rho^\alpha) \int_{B(z,\rho)} |\nabla u^\ast_\rho|^2
+ C \rho^{n+\alpha}.
\nonumber
\end{eqnarray}
This is the same as (\ref{eqn3.2}), 
with $(x,r)$ replaced with $(z,\rho)$,
even though we obtained the cancellation for different reasons.
We continue the argument as in 
(\ref{eqn3.3})-(\ref{eqn3.8}). Notice that we sill
have that $u > 0$ almost everywhere in smaller balls,
so we can iterate the argument as we did in (\ref{eqn3.6})
We finally obtain, as in (\ref{eqn3.8}), that
\begin{equation}
\label{eqn4.43}
\o(z,s) \leq C \o(z,\rho) + C \rho^{\alpha/2}
\ \hbox{ for } 0 < s \leq \rho.
\end{equation}
Since in addition $\o(z,\rho) \leq C \tau^{-\frac{n}{2}}\o(x,r)$ by the definition (2.6), we deduce from (\ref{eqn4.43})
that for $0 < s \leq \tau r /8$, 
\begin{equation}
\label{eqn4.44}
\o(z,s) \leq C (\tau^{-\frac{n}{2}}\o(x,r) + \rho^{\alpha/2}),
\end{equation}
which is slightly better than (\ref{eqn4.36}).

We continue the argument as in the proof of 
Theorem \ref{thm3.1}. Set $\rho_{j} = 2^{-j} \rho$
(the analogue of $r_{j} = 2^{-j} r$)
and $u_{j} = \fint_{B(z,\rho_{j})} u$, we still have 
the analogue of (\ref{eqn3.10}), i.e., that
\begin{equation}
\label{eqn4.45}
|u(z)-u_{j}| \leq C \rho_{j} (\o(z,\rho) + \rho^{\alpha/2})
\leq C 2^{-j}\tau r\left(\tau^{-\frac{n}{2}}\o(x,r) + \rho^{\alpha/2}\right).
\end{equation}
Next let $z'$ be another point of $B(x, \tau r/4)$, denote by
$\rho' \in \left(\frac{\tau r}{8}, \frac{\tau r}{4}\right)$ 
the analogue of $\rho$ for $z'$, and also define 
$\rho'_{j}=2^{-j} \rho'$ and $u'_{j} = \fint_{B(z',\rho'_{j})} u$.
Then
\begin{equation}
\label{eqn4.46}
|u(z')-u'_{j}| \leq C \rho'_{j} (\o(z',\rho') + (\rho')^{\alpha/2})
\leq C 2^{-j}\tau r \left(\tau^{-\frac{n}{2}}\o(x,r) + \rho^{\alpha/2}\right).
\end{equation}
Set $\delta = |z'-z|$.
Choose $j$ such that $\rho_{j+1} \leq \delta < \rho_{j}$ and
$j'$ such that $\rho'_{j'+1} \leq \delta < \rho'_{j'}$, and set
$B = B(z,\rho_j)$, $B' = B(z',\rho'_{j'})$, and $B'' = B(z,(2^{-j}\rho+2^{-j'}\rho'))$.
Notice that $B \cup B' \i B''$ because $|z'-z| =\delta \leq 2^{-j}\rho$, and that
$u_{j} = \fint_{B} u$ and $u'_{j'} = \fint_{B'} u$. Also set
$m = \fint_{B''} u$. Then by Poincar\'e
\begin{eqnarray}
\label{eqn4.47}
\left| u'_{j'}-u_{j} \right| &\leq& 
\left| m - \fint_{B'} u \right| + \left| m -\fint_{B} u \right| 
\leq C \fint_{B'} |u-m| + C \fint_{B} |u-m|
\nonumber 
\\
&\leq& C \left\{\left(\frac{2^{-j}\rho +2^{-j'}\rho'}{2^{-j}\rho} \right)^n + \left(\frac{2^{-j}\rho +2^{-j'}\rho'}{2^{-j'}\rho'} \right)^n
\right\}\fint_{B''} |u-m|\\
&\leq& C \fint_{B''} |u-m|
\leq C \delta  \fint_{B''} |\nabla u|
\nonumber
\end{eqnarray}
because the radii $2^{-j}\rho$, $2^{-j'}\rho'$, and $2^{-j}\rho+2^{-j'}\rho'$
are all comparable to $\delta $.

Suppose in addition that $\delta \leq \tau r/32$.
Then $2^{-j}\rho+2^{-j'}\rho' \leq 4 \delta \leq \tau r/8 \leq \rho$. 
By Cauchy-Schwarz and (\ref{eqn4.44}),
\begin{equation}
\label{eqn4.48}
\fint_{B''} |\nabla u| \leq \omega(z,2^{-j}\rho+2^{-j'}\rho')
\leq C (\tau^{-\frac{n}{2}}\o(x,r) + \rho^{\alpha/2}).
\end{equation}
Altogether, by (\ref{eqn4.45}), (\ref{eqn4.46}), (\ref{eqn4.47}), 
and (\ref{eqn4.48}),
\begin{eqnarray}
\label{eqn4.49}
\left| u(z')-u(z) \right| 
&\leq& \left| u(z')- u'_{j'} \right| + \left| u'_{j'}-u_{j} \right|
+ \left| u_{j} - u(z)\right|
\nonumber\\
&\leq& 
C (2^{-j}\rho+2^{-j'}\rho') \left(\tau^{-\frac{n}{2}}\o(x,r) + \rho^{\alpha/2}\right)
+ C \delta  \fint_{B''} |\nabla u|
\\
&\leq& C \delta \left(\tau^{-\frac{n}{2}}\o(x,r) + \rho^{\alpha/2}\right)
= C |z'-z| \left(\tau^{-\frac{n}{2}}\o(x,r) + \rho^{\alpha/2}\right).
\nonumber
\end{eqnarray}
When $z, z' \in B(x, \tau r/4)$ are such that
$\delta = |z'-z| > \tau r/32$, we still get (\ref{eqn4.49}),
by going through a few intermediate points; (\ref{eqn4.37})
follows.

\medskip
For (\ref{eqn4.38}) we follow the proof of Theorem \ref{thm3.2}.
Let us nonetheless sketch the argument, since there are a few small
differences.
First we fix $z \in B(x, \tau r/8)$, define $\rho$ and $u^\ast_{\rho}$
as above, and proceed as in Theorem \ref{thm3.2}, but
with $(x,r)$ replaced by $(z,\rho)$. Up to (\ref{eqn3.21}), we change
nothing; then we observe that since (\ref{eqn4.49}) holds for all
$z \in B(x, \tau r/4)$,
\begin{equation}
\label{eqn4.50}
|\nabla u(y)| \leq  C \left(\tau^{-\frac{n}{2}}\o(x,r) + \rho^{\alpha/2}\right)
\end{equation} 
almost everywhere on $B(x, \tau r/4)$, and hence 
\begin{equation}
\label{eqn4.51}
\o(z,\rho) \leq C \left(\tau^{-\frac{n}{2}}\o(x,r) + \rho^{\alpha/2}\right)
\end{equation} 
(because $B(z,\rho) \i B(x, \tau r/4)$ since $z \in B(x, \tau r/8)$).
We compare with the analogue of (\ref{eqn3.21}) and get that
\begin{eqnarray}
\label{eqn4.52}
\fint_{B(x,\tau \rho)} \left| \nabla u- \fint_{B(z,\rho)} \nabla u^\ast_{\rho} \right|^2 
&\leq&  C (\tau^{-n} \rho^\alpha + \tau^2) \omega(z,\rho) + C \tau^{-n} \rho^\alpha
\\
&\leq& C (\tau^{-n} \rho^\alpha + \tau^2)
\left(\tau^{-\frac{n}{2}}\o(x,r) + \rho^{\alpha/2}\right)
+ C \tau^{-n} \rho^\alpha
\nonumber \\
&\leq&  C_{0} (\tau^{-n} \rho^\a +\tau^2),
\nonumber
\end{eqnarray}
with $C_0 = C \left(\tau^{-\frac{n}{2}}\o(x,r) + \rho^{\alpha/2}+1\right)$,
and as in (\ref{eqn3.23}). We continue the argument, with (\ref{eqn3.23}) 
replaced by (\ref{eqn4.52}). We find a small $\rho_{0}$, that depends on 
$\beta = \frac{\alpha}{n+2+\alpha}$ (see below (\ref{eqn3.24})) $r$, and $\tau$,
such that, as in (\ref{eqn3.25}), 
\begin{eqnarray}
\label{eqn4.53}
\left| \fint_{B(z,t/2)} \nabla u - \fint_{B(z,t)} \nabla u \right|
\le C C_0^{\frac{1}{2}}\rho^{\beta}
\ \hbox{ for } 0 < t \leq \rho_{0}.
\end{eqnarray}
This also yields that $\left| \nabla u(z) - \fint_{B(z,t)} \nabla u \right| 
\leq C C_0^{\frac{1}{2}}\rho^{\beta}$ if $z$ is a Lebesgue point
for $\nabla u$, by summing a geometric series. 

Then, if $y\in B(x, \tau r/8)$ is another Lebesgue point for $\nabla u$,
and if $|y-z| \leq \rho_{0}/2$, we deduce from (\ref{eqn4.53}), 
a similar estimate for $y$, and a comparison as in (\ref{eqn3.26}) that
 \begin{equation}
\label{eqn4.54}
| \nabla u(y) - \nabla u(z) | \leq  C C_{0} |x-y| ^{\beta};
\end{equation}
our last estimate (\ref{eqn4.38}) follow easily from this.
This completes our proof of Lemma \ref{lem4.3}.
\qed
\end{proof}

We end this section with a way to obtain pairs $(x,\rho)$
that satisfy the conditions of Lemmas~\ref{lem4.1}-\ref{lem4.3},
under somewhat weaker assumptions. We shall now assume that
$b(x,r) = \fint_{\p B(x,r)} u$ is not too small, and 
use homogeneity to find a smaller radius $\rho$ where
(accounting for scale invariance) it looks big.

\begin{lemma}\label{lem4.4} 
Let $u$ be an almost minimizer for $J$ in $\O$. 
For each choice of $\gamma \in (0,1)$, $\tau > 0$,
and $C_{0} \geq 1$,  we can find  $r_{0}$, 
$\eta < 10^{-1}$, and $K \geq 1$  with the following property. 
Let $x \in \O$ and $r > 0$ be such that
$0 < r \leq r_{0}$, $B(x,2r) \i \O$,
 \begin{equation}
\label{eqn4.55}
| b(x,r) | \geq \gamma \, r \, (1+\omega(x,r)).
\end{equation}
and  
\begin{equation}
\label{eqn4.56}
\omega(x,r) \geq K.
\end{equation}
Then there exists $\rho \in (\frac{\eta r}{2}, \eta r)$ such that
$(x,\rho) \in {\cal G}(\tau, C_{0}, 3, r_{0})$.
\end{lemma}

\begin{proof}
Let $\eta \in (0,10^{-1})$ be small, to be chosen later, and 
let $(x,r)$ be as in the statement. As usual, denote by $u^\ast_r$
the energy-minimizing function that coincides with $u$ on $\p B(x,r)$.
Notice that $|\nabla u^\ast_r|^2$ is subharmonic on $B(x,r)$ 
(because $u^\ast_r$ is harmonic), and 
$\int_{B(x,r)} |\nabla u^\ast_r|^2 \leq \int_{B(x,r)} |\nabla u|^2$.
Hence, for $y\in B(x,\eta r)$,
\begin{equation}
\label{eqn4.57} 
|\nabla u^\ast_r(y)|^2 \leq 
\fint_{B\left(y,\frac{r}{2}\right)}|\nabla u^\ast_r|^2
\le 2^n\fint_{B(x,r)}|\nabla u^\ast_r|^2 
\le 2^n\o(x,r)^2.
\end{equation}
Then, for $y\in B(x,\eta r)$ and because 
$u^\ast_r(x) = \fint_{\p B(x,r)} u = b(x,r)$ by harmonicity
(and either Remark 3.1 
or a small limiting argument to go to the traces
on the boundary $\p B(x,r)$, which is very easy because 
our functions are bounded and we could use the dominated
convergence theorem),
\begin{equation}
\label{eqn4.58}
|u^\ast_r(y)-b(x,r)| = |u^\ast_r(y)-u^\ast_r(x)| 
\leq \eta r \sup_{B(x,\eta r)}|\nabla u^\ast_r|
\leq 2^{n/2} \eta r \o(x,r).
\end{equation}

We shall choose $\eta$ so small that 
$\eta 2^{n/2}<\gamma/4$; then (\ref{eqn4.55}) yields
\begin{equation}
\label{eqn4.59}
|u^\ast_r(y)-b(x,r)| \leq 2^{n/2} \eta r \o(x,r)
\leq \frac{1}{4}\,\gamma  r \o(x,r)
\leq \frac{1}{4} |b(x,r)|.
\end{equation}
In particular, $u^\ast_r(y)$ keeps the same sign as 
$b(x,r)$ on $B(x,\eta r)$, and 
\begin{equation}
\label{eqn4.60}
\frac{5}{4}|b(x,r)|\ge |u^\ast_r(y)| \geq \frac{3}{4} |b(x,r)|
\ \hbox{ for $y\in B(x,\eta r)$.}
\end{equation}

Choose $\rho\in \left(\frac{\eta r}{2},\eta r\right)$, as we did near
(\ref{eqn4.11}), such that by Poincar\'e's inequality and Cauchy-Schwarz,
\begin{eqnarray}
\label{eqn4.61}
\int_{\p B(x,\rho)}|u-u^\ast_r| 
&\leq& 
\frac{2}{\eta r} \int_{\eta r/2}^{\eta r} \int_{\p B(x,\rho)}|u-u^\ast_r| d\rho
\leq \frac{2}{\eta r} \int_{B(x,\eta r)}|u-u^\ast_r| 
\nonumber\\
&\leq&  C \int_{B(x,\eta r)}|\nabla u- \nabla u^\ast_r| 
\leq C (\eta r)^{\frac{n}{2}} 
\left(\int_{B(x,\eta r)}|\nabla u- \nabla u^\ast_r|^2 \right)^{\frac{1}{2}}
\\ 
&\leq& C (\eta r)^{\frac{n}{2}} 
\left(\int_{B(x,r)}|\nabla u- \nabla u^\ast_r|^2 \right)^{\frac{1}{2}}.
\nonumber
\end{eqnarray}
By (\ref{eqn2.4}) and as in (\ref{eqn4.8})
\begin{equation}
\label{eqn4.62}
\fint_{B(x,r)}|\nabla u-\nabla u^\ast_r|^2 
\leq \kappa r^\a \fint_{B(x,r)}|\nabla u|^2  + C 
= \kappa r^\a\o(x,r)^2 + C, 
\end{equation}
so
\begin{equation}
\label{eqn4.63}
\int_{\p B(x,\rho)}|u-u^\ast_r| 
\leq C \eta^{\frac{n}{2}} r^n \left(1+ r^{\a} \o(x,r)^2\right)^{\frac{1}{2}}.
\end{equation}
Recall that $r \leq r_{0}$; then $r^{\a} \leq r_{0}^{\a}$, 
and by (\ref{eqn4.63}) and (\ref{eqn4.56})
\begin{eqnarray}
\label{eqn4.64}
\fint_{\p B(x,\rho)}|u-u^\ast_r|
&\leq& C (\eta r)^{1-n} \int_{\p B(x,\rho)}|u-u^\ast_r| 
\leq C \eta^{1-\frac{n}{2}} r \left(1+ r_{0}^{\a}\o(x,r)^2\right)^{\frac{1}{2}}
\\
&\leq& 
C \eta^{1-\frac{n}{2}} r \o(x,r) \left(K^{-2}+ r_{0}^{\a}\right)^{\frac{1}{2}}.
\nonumber
\end{eqnarray}
We shall choose $K$ large enough, and $r_{0}$ small enough, both 
depending on $\gamma$ and $\eta$, so that in (\ref{eqn4.64})
\begin{equation}
\label{eqn4.65}
C \eta^{1-\frac{n}{2}} \left(K^{-2}+ r_{0}^{\a}\right)^{\frac{1}{2}}
\leq \frac{\gamma}{4}.
\end{equation}
Then by (\ref{eqn4.64}) and (\ref{eqn4.55}) 
\begin{equation}
\label{eqn4.66}
\fint_{\p B(x,\rho)}|u-u^\ast_r| \leq \frac{\gamma}{4}\, r \o(x,r)
\leq \frac{|b(x,r)|}{4}.
\end{equation}

Recall that $u^\ast_r$ does not change signs on $\p B(x,\rho)$;
then by (\ref{eqn4.60}) and (\ref{eqn4.66})
\begin{eqnarray}
\label{eqn4.67}
|b(x,\rho)| &=& \left| \fint_{\p B(x,\rho)} u \right|
\geq \left| \fint_{\p B(x,\rho)} u^\ast_{r} \right|
- \fint_{\p B(x,\rho)} |u-u^\ast_r|
\nonumber\\
&=&  \fint_{\p B(x,\rho)} |u^\ast_r| 
- \fint_{\p B(x,\rho)} |u-u^\ast_r|
\\ 
&\geq& \frac{3}{4} |b(x,r)| - \frac{1}{4} |b(x,r)| = \frac{1}{2} |b(x,r)|.
\nonumber
\end{eqnarray}
The same computations yield
\begin{eqnarray}
\label{eqn4.68}
|b^+(x,\rho)| &=& \fint_{\p B(x,\rho)} | u |
\leq \fint_{\p B(x,\rho)} |u^\ast_r|  
+ \fint_{\p B(x,\rho)} |u-u^\ast_r|
\\
&\leq&  \frac{5}{4} |b(x,r)| + \frac{1}{4} |b(x,r)|
\leq \frac{3}{2} |b(x,r)|
\nonumber
\end{eqnarray}
by the definition (\ref{eqn4.1}), (\ref{eqn4.60}), and
(\ref{eqn4.66}). It follows from
(\ref{eqn4.67}) and (\ref{eqn4.68}) that $(x,\rho)$ satisfies (\ref{eqn4.3}) 
with $C_{1} = 3$.

We still need to check (\ref{eqn4.2}). But
by (\ref{eqn4.67}) and (\ref{eqn4.55}),
\begin{equation}
\label{eqn4.69} 
\frac{|b(x,\rho)|}{\rho} \geq  \frac{1}{2\rho} \, |b(x,r)|
\geq \frac{\gamma r}{2\rho} (1+\o(x,r))
\geq \frac{\gamma}{2 \eta} (1+\o(x,r))
\end{equation}
and now we want a lower bound for $\o(x,r)$ in terms of $\o(x,\rho)$.

Recall from (\ref{eqn2.10}) 
that whenever $B(x,r) \i \O$, 
we have that for $j \geq 0$,
\begin{equation}
\label{eqn4.70}
\o(x,2^{-j-1}r) \leq C \o(x,r) + Cj.
\end{equation}
We apply this to the integer $j$ such that
$2^{-j-2} r \leq \rho < 2^{-j-1} r$ and get that
\begin{equation}
\label{eqn4.71}
\omega(x,\rho) \leq 2^{n/2}\o(x,2^{-j-1}r) \leq C \o(x,r) + Cj
\leq C \o(x,r) + C |\log \eta|
\end{equation}
and now (\ref{eqn4.69}) yields
\begin{eqnarray}
\label{eqn4.72}
\left(1+ \rho^\alpha \omega(x,\rho)^2\right)^{\frac{1}{2}}
&\leq & 1 + \rho^{\alpha/2} \omega(x,\rho)
\leq 1 + Cr_{0}^{\alpha/2}  \o(x,r) + C r_{0}^{\alpha/2}  |\log \eta|
\nonumber\\
&\leq& \left(1+Cr_{0}^{\alpha/2} + C r_{0}^{\alpha/2}  |\log \eta|\right)
\, (1+\omega(x,r))
\\
&\leq& \left(1+Cr_{0}^{\alpha/2} + C r_{0}^{\alpha/2}  |\log \eta|\right)
\, \frac{2 \eta}{\gamma} \, \frac{|b(x,\rho)|}{\rho} 
\nonumber
\end{eqnarray}
and (multiplying with $C_0\tau^{-n}$)
\begin{equation}
\label{eqn4.73}
C_0\tau^{-n} \left(1+ \rho^\alpha \omega(x,\rho)^2\right)^{\frac{1}{2}}
\leq A \, \frac{|b(x,\rho)|}{\rho},
\end{equation} 
with
\begin{equation}
\label{eqn4.74}
A =  C_0\tau^{-n} \left(1+Cr_{0}^{\alpha/2} + C r_{0}^{\alpha/2}  |\log \eta|\right)
\, \frac{2 \eta}{\gamma} 
\end{equation} 
 Thus we see that (\ref{eqn4.2}) holds for the pair $(x,\rho)$ as soon as $A \leq 1$.
 We choose $\eta$ so small, depending on $C_{0}$, $\tau$, and $\gamma$,
 that $C_0\tau^{-n} \, \frac{2 \eta}{\gamma} \leq \frac{1}{2}$, and then
 $r_{0}$ so small, depending on $\eta$, that
 $1+Cr_{0}^{\alpha/2} + C r_{0}^{\alpha/2}  |\log \eta| \leq 2$.
 This is compatible with our previous constraints (see below (\ref{eqn4.58})
 and below (\ref{eqn4.64})). 
 It is true that $C$ in (\ref{eqn4.74}) depends on $r_0$, but only through an upper
 bound on $r_0$, so there is no vicious circle here.

We just finished our verification that $(x,\rho) \in {\cal G}(\tau, C_{0}, 3, r_{0})$;
this completes our proof of Lemma \ref{lem4.4}.
\qed
\end{proof}

\section{Almost minimizers for $J^+$ are locally Lipschitz}

We now distinguish between $J^+$ and $J$.
We start with the somewhat easier case of $J^+$, and prove in this section
that the almost minimizers of $J^+$ are locally Lipschitz in $\O$
(see Theorem \ref{thm5.1} below).

\begin{lemma} \label{lem5.1}
Let $u$ be an almost minimizer for $J^+$ in $\O$. Pick $\theta \in (0,1/2)$.
There exist $\gamma>0$, $K_1>1$, $\beta\in (0,1)$, and $r_1>0$ 
such that if $x\in \O$ and $0 < r \leq r_{1}$ are such that 
$B(x,r) \i \O$,
\begin{equation}
\label{eqn5.1}
b(x, r) \leq \gamma r \left(1+\o(x,r)\right),
\end{equation}
and
\begin{equation}
\label{eqn5.2}
\o(x, r) \geq K_1,
\end{equation}
then
\begin{equation}
\label{eqn5.3}
\o(x,\theta r) \leq \beta\o(x,r).
\end{equation}
\end{lemma}

\begin{proof}
Recall from the definition (\ref{eqn1.8}) that 
almost minimizers for $J^+$ are non-negative almost everywhere,
hence everywhere on $\O$ because Theorem \ref{thm2.1}
says that they are continuous (after 
modification on a set of measure zero). 

Let $x \in \O$ and $r \leq r_{1}$ be such that $B(x,r)  \i \O$.
Let $u^\ast_r$ denote the energy-minimizing extension
of the restriction of $u$ to $\p B(x,r)$. Note 
that $u_r^\ast\ge 0$ in $\overline B(x,r)$.
For $y\in B(x,r)$, set $a(y) = u^\ast_r(x) + 
 \langle\nabla u^\ast_r(x), y-x\rangle$, and
\begin{equation}
\label{eqn5.4}
v^\ast_r(y) = u^\ast_r(y)- a(y) =
u^\ast_r(y)-u^\ast_r(x)
-\langle\nabla u^\ast_r(x), y-x\rangle; 
\end{equation}
then $v^\ast$ is harmonic in $B(x,r)$, $v^\ast_r(x)=0$, 
and $\nabla v^\ast_r(x)=0$. 

Recall from the fourth line of (\ref{eqn2.8}) that for $0 < s \leq r$,
\begin{equation}
\label{eqn5.5}
\o(x,s) \leq C\left(\frac{r}{s}\right)^{n/2}r^{\a/2}\o(x,r) +
C\left(\frac{r}{s}\right)^{n/2} + 
\left(\fint_{B(x,s)}|\nabla u^\ast_r|^2\right)^{\frac{1}{2}}.
\end{equation} 
Next we evaluate $\fint_{B(x,s)}|\nabla u^\ast_r|^2$.
By (\ref{eqn5.4}) and because $\nabla a = \nabla u^\ast_r(x)$,
\begin{eqnarray}
\label{eqn5.6}
\fint_{B(x,s)}|\nabla u^\ast_r|^2 
&=& 
\fint_{B(x,s)}|\nabla (a + v^\ast_r)|^2
= \fint_{B(x,s)}|\nabla v^\ast_r|^2 + \fint_{B(x,s)}|\nabla a|^2
+ 2\fint_{B(x,s)}\langle \nabla a ,\nabla v^\ast_r\rangle
\nonumber\\
&=& 
\fint_{B(x,s)}|\nabla v^\ast_r|^2 
+ |\nabla u^\ast_r(x)|^2 
+ 2\langle\nabla u^\ast_r(x) , \fint_{B(x,s)} \nabla v^\ast_r\rangle.
\end{eqnarray}
Since $v^\ast_r$ is harmonic on $B(x,r)$, so is $\nabla v^\ast_r$, and
$\fint_{B(x,s)} \nabla v^\ast_r = \nabla v^\ast_r(x) = 0$
by definition of $v^\ast_r$. So 
\begin{equation}
\label{eqn5.7}
\fint_{B(x,s)}|\nabla u^\ast_r|^2 = |\nabla u^\ast_r(x)|^2
+ \fint_{B(x,s)}|\nabla v^\ast_r|^2.  
\end{equation}
The same proof, with $B(x,s)$ replaced by $B(x,r)$ shows that
\begin{equation}
\label{eqn5.8}
\fint_{B(x,r)}|\nabla u^\ast_r|^2 = |\nabla u^\ast_r(x)|^2
+ \fint_{B(x,r)}|\nabla v^\ast_r|^2.  
\end{equation}
We return to $\fint_{B(x,s)}|\nabla u^\ast_r|^2$. 
By (\ref{eqn5.7}), because $\fint_{B(x,s)} \nabla v^\ast_r
= \nabla v^\ast_r(x) = 0$, by Poincar\'e's inequality , and because 
$\nabla^2 a = 0$,
\begin{eqnarray}
\label{eqn5.9}
\fint_{B(x,s)} |\nabla u^\ast_r|^2 & = & 
|\nabla u^\ast_r(x)|^2 + \fint_{B(x,s)}|\nabla v^\ast_r|^2 
\nonumber\\
& = & |\nabla u^\ast_r(x)|^2 
+ \fint_{B(x,s)} \left| \nabla v^\ast_r 
- \fint_{B(x,s)}\nabla v^\ast_r \right|^2 
\nonumber \\
& \le & |\nabla u^\ast_r(x)|^2 
+ C s^2\fint_{B(x,s)}|\nabla^2 v^\ast_r|^2 
\\
& \le & |\nabla u^\ast_r(x)|^2 
+ Cs^2\fint_{B(x,s)}\left|\nabla^2u^\ast_r\right|^2.
\nonumber 
\end{eqnarray}
Now suppose that $s< r/2$;
by standard estimates on harmonic functions,
\begin{eqnarray}
\label{eqn5.10}
\fint_{B(x,s)}\left|\nabla^2u^\ast_r\right|^2
&\leq& \sup_{B(x,s)} |\nabla^2u^\ast_r|^2
\leq C  \Big( r^{-2} \fint_{\p B(x,r)} |u^\ast_r| \Big)^2
\\
&=& C \Big( r^{-2} \fint_{\p B(x,r)} u \Big)^2
= C r^{-4} b(x,r)^2 
\nonumber
\end{eqnarray}
because $u^\ast_r = u \geq 0$ on $\p B(x,r)$
and by the definition (\ref{eqn4.1}). Now (\ref{eqn5.9})
and (\ref{eqn5.10}) yield
\begin{equation}
\label{eqn5.11}
\fint_{B(x,s)} |\nabla u^\ast_r|^2 
\leq |\nabla u^\ast_r(x)|^2 
+ Cs^2\fint_{B(x,s)}\left|\nabla^2u^\ast_r\right|^2
\leq |\nabla u^\ast_r(x)|^2 +  C r^{-4} s^2 b(x,r)^2.
\end{equation}

By (\ref{eqn5.5}), (\ref{eqn5.11}) and since $b(x,r) \geq 0$
(and because $\sqrt{a^2+b^2} \leq a+b$ for $a, b \geq 0$)
\begin{eqnarray}
\label{eqn5.12}
\o(x,s) 
&\leq& 
C\left(\frac{r}{s}\right)^{n/2}r^{\a/2}\o(x,r) +
C\left(\frac{r}{s}\right)^{n/2} + 
\left(\fint_{B(x,s)}|\nabla u^\ast_r|^2\right)^{\frac{1}{2}}
\\
&\leq&
C\left(\frac{r}{s}\right)^{n/2}r^{\a/2}\o(x,r) +
C\left(\frac{r}{s}\right)^{n/2} + |\nabla u^\ast_r(x)| +  C r^{-2} s b(x,r).
\nonumber
\end{eqnarray}
Let $\theta \in (0,1/2)$ be as in the statement, and  
take $s=\theta r < r/2$.
With this notation, (\ref{eqn5.12}) yields
\begin{eqnarray}
\label{eqn5.13}
\o(x,\theta r) 
&\leq &
 |\nabla u^\ast_r(x)| + C \theta^{-n/2} r^{\a/2}\o(x,r)
 + C \theta^{-n/2}  + C \theta r^{-1} b(x,r)
\nonumber \\
 &\leq &
 |\nabla u^\ast_r(x)| 
 + C \theta^{-n/2} \left(r^{\a/2} + K_{1}^{-1}\right) \o(x,r)
 + C \theta \gamma (1+\o(x,r))
 \\
 &\leq &
 |\nabla u^\ast_r(x)| 
 + C \Big[ \theta^{-n/2} \left(r^{\a/2} + K_{1}^{-1}\right) 
 +\theta \gamma \left( K_{1}^{-1} + 1\right) \Big] \, \o(x,r)
 \nonumber
\end{eqnarray}
by (\ref{eqn5.2}) and (\ref{eqn5.1}).

We shall now control $|\nabla u^\ast_r(x)|$ in terms of $\o(x,r)$.
We consider two cases. Let $\eta > 0$ be small, to be chosen soon. If 
\begin{equation}
\label{eqn5.14}
\fint_{B(x,r)}|\nabla v^\ast_r|^2 \geq \eta^2 \fint_{B(x,r)}|\nabla u|^2
= \eta^2 \o(x,r)^2
\end{equation}
then we use (\ref{eqn5.8}) to prove that
\begin{eqnarray}
\label{eqn5.15}
\o(x,r)^2 &=&  \fint_{B(x,r)}|\nabla u|^2 
\geq \fint_{B(x,r)}|\nabla u^\ast_r|^2
 = |\nabla u^\ast_r(x)|^2 + \fint_{B(x,r)}|\nabla v^\ast_r|^2
 \\
&\geq&
|\nabla u^\ast_r(x)|^2 + \eta^2 \o(x,r)^2 
 \nonumber
\end{eqnarray}
and hence by (\ref{eqn5.13})
\begin{eqnarray}
\label{eqn5.16}
\o(x,\theta r) 
&\leq &
 |\nabla u^\ast_r(x)| 
 + C \Big[ \theta^{-n/2} \left(r^{\a/2} + K_{1}^{-1}\right) 
 +\theta \gamma \left( K_{1}^{-1} + 1\right) \Big] \, \o(x,r)
 \nonumber \\
 &\leq &
\sqrt{1-\eta^2} \, \o(x,r)
 + C \Big[ \theta^{-n/2} \left(r^{\a/2} + K_{1}^{-1}\right) 
 +\theta \gamma \left( K_{1}^{-1} + 1\right) \Big] \, \o(x,r).
\end{eqnarray}
We shall deal with the quantifiers soon, but let us get rid of the 
case when (\ref{eqn5.14}) fails. 
Then by (\ref{eqn5.8})
\begin{equation}
\label{eqn5.17}
\fint_{B(x,r)}|\nabla u^\ast_r|^2 =
 |\nabla u^\ast_r(x)|^2 + \fint_{B(x,r)}|\nabla v^\ast_r|^2
 \leq |\nabla u^\ast_r(x)|^2 + \eta^2 \o(x,r)^2.
 \end{equation}
 Since standard estimates on harmonic functions yield
 \begin{equation}
\label{eqn5.18}
 |\nabla u^\ast_r(x)| \leq C r^{-1} \fint_{\p B(x,r)} |u^\ast_r| 
 = C r^{-1} \fint_{\p B(x,r)} |u|  
 = C r^{-1} \fint_{\p B(x,r)} u 
 = C r^{-1} b(x,r)
 \end{equation}
as for (\ref{eqn5.10}), and because 
$u^\ast_r = u \geq 0$ on $\p B(x,r)$.
Returning to (\ref{eqn5.17}),
\begin{eqnarray}
\label{eqn5.19}
\fint_{B(x,r)}|\nabla u^\ast_r|^2 
 &\leq & 
 |\nabla u^\ast_r(x)|^2 + \eta^2 \o(x,r)^2
 \leq C r^{-2} b(x,r)^2 + \eta^2 \o(x,r)^2
 \\
 &\leq&
  C \gamma^2 (1+\o(x,r))^2 + \eta^2 \o(x,r)^2
  \nonumber
 \end{eqnarray}
by (\ref{eqn5.1}). At the same time, (\ref{eqn5.5})
with $s=r$ and then (\ref{eqn5.19})  
and  (\ref{eqn5.2}) yield
\begin{eqnarray}
\label{eqn5.20}
\o(x,r) &\leq& Cr^{\a/2}\o(x,r) + C
+\left(\fint_{B(x,r)}|\nabla u^\ast_r|^2\right)^{\frac{1}{2}}
\\
&\leq& Cr^{\a/2}\o(x,r) + C
+ C \gamma (1+\o(x,r)) + \eta \o(x,r)
\nonumber\\
&\leq&  
C \big[r_{1}^{\a/2} + K_{1}^{-1} + \gamma K_{1}^{-1}+\gamma + \eta \big]
\, \o(x,r)
 \nonumber
\end{eqnarray} 
(recall our assumption that $r < r_{1}$). 
If $K_{1}$ is large enough, and $r_{1}$, $\gamma$, and $\eta$
are small enough, we get a contradiction because $\o(x,r) \geq K_1 > 0$.
That is, we choose $\eta$ so that $C\eta<1/4$ 
and only consider $K_1$ large enough and $r_1$ small enough so
\begin{equation}\label{eqn5.20A}
C\left( r_1^{\alpha/2} + K_1^{-1} +\gamma K_1^{-1} +\gamma\right)<\frac{1}{4}
\end{equation}
Under these conditions the second case is impossible and (\ref{eqn5.16}) holds. To deduce (\ref{eqn5.3})
choose $K_1$, $r_1$ and $\gamma$ satisfying both (\ref{eqn5.20A}) and
\begin{equation}
\label{eqn5.21A}
 C \Big[ \theta^{-n/2} \left(r_{1}^{\a/2} + K_{1}^{-1}\right) 
 +\theta \gamma \left( K_{1}^{-1} + 1\right) \Big]  \leq \frac{1-\sqrt{1-\eta^2}}{2},
\end{equation}
where $\eta$ is as above. Then letting $\beta\in (\frac{1+\sqrt{1-\eta^2}}{2},1)$
we have
\begin{equation}
\label{eqn5.21}
\sqrt{1-\eta^2}+ C \Big[ \theta^{-n/2} \left(r_{1}^{\a/2} + K_{1}^{-1}\right) 
 +\theta \gamma \left( K_{1}^{-1} + 1\right) \Big]  \leq \beta,
\end{equation}
which ensures that (\ref{eqn5.3}) holds ; 
Lemma \ref{lem5.1} follows.
\qed
\end{proof}

\begin{theorem}\label{thm5.1}
Let $u$ be an almost minimizer for $J^+$ in $\O$. Then $u$ is locally
Lipschitz in $\O$.
\end{theorem}

Let us make two observations before we start the proof.
The theorem comes with uniform estimates. That is,
there exists $r_{2} > 0$ and $C_{2} \geq 1$
(that depend on $n$, $\kappa$, and $\alpha$) such that
for each choice of $x_{0} \in \O$ and $r_{0}>0$ such that
$r_{0} \leq r_{2}$ and $B(x_{0},2r_{0}) \i \O$,
\begin{equation}
\label{eqn5.22}
|u(x)-u(y)| \leq C_{2} (\omega(x_{0},2r_{0}) + 1) |x-y|
\ \mbox{ for } x, y \in B(x_{0},r_{0}).
\end{equation}

 Theorem \ref{thm3.2} ensures that $u$ is more regular away from the free boundary
 $\p(\{ u >0 \})$. 
In the good cases, we expect 
$u$ to behave, near a point of $\p(\{ u >0 \})$, like 
$a(x)_{+} = \mathrm{max}(0,a(x))$ for some non constant affine
function that vanishes at the given point. 
Precise results in this direction are beyond the scope of this paper.

\begin{proof}
Let $(x,r)$ be such that $B(x,2r) \i \O$; we want to see whether our
different lemmas can bring us to a pair $(x,\rho)$ 
where we control $u$, and for this we will distinguish between a few cases.

Pick $\theta = \frac{1}{3}$ for definiteness (but smaller
values would work as well), and let $\beta$, $\gamma$, $K_{1}$, 
and $r_{1}$ be as in Lemma \ref{lem5.1}. 
Then pick $\tau = \tau_{1}/2$,
where $\tau_{1}\in (0,10^{-2})$ is the constant that we get in 
Lemma \ref{lem4.1}, 
applied with $C_{1}=3$ and $r_{0} = r_{1}$.

Next let $r_{0}$, $\eta$, and $K$ be as in Lemma \ref{lem4.4}, 
applied with $C_{0} = 10$ and the small $\gamma$ that we just found. 
Set $r_{\gamma} = r_{0}$ to avoid confusion. Set
\begin{equation}
\label{eqn5.23}
K_{2} \geq \mathrm{max}(K_{1},K)
\ \mbox{ and } r_{2} \leq \mathrm{min}(r_{1}, r_{\gamma}),
\end{equation}
and assume that $r \leq r_{2}$.  We consider three cases.
\medskip

\noindent{\bf Case 1}: 
\begin{equation}
\label{eqn5.24}
\left\{\begin{array}{rcl}
\o(x,r) & \ge & K_2 \\
b(x,r) & \ge & \gamma r \left(1+\o(x,r)\right)
\end{array}\right.
\end{equation}
\noindent{\bf Case 2}:
\begin{equation}
\label{eqn5.25}
\left\{\begin{array}{rcl}
\o(x,r) & \ge & K_2 \\
b(x,r) & < & \gamma r \left(1+\o(x,r)\right)
\end{array}\right.
\end{equation}
\noindent{\bf Case 3}:
\begin{equation}
\label{eqn5.26}
\o(x,r) < K_2.
\end{equation}

We start with Case 1.  
By (\ref{eqn5.24}) we can apply Lemma \ref{lem4.4} 
(recall that $r \leq r_{2}$, which by (\ref{eqn5.23}) 
is not more than the $r_{\gamma}$ of Lemma \ref{lem4.4})
and we find $\rho \in \left(\frac{\eta r}{2}, \eta r\right)$ such that
$(x,\rho) \in {\cal G}(\tau, 10, 3, r_{\gamma})$. Notice that the pair
$(x,\rho)$ satisfies the assumptions of Lemmas \ref{lem4.1}-\ref{lem4.3},
which we applied with $r_{0} = r_{1}$. This is the way
we defined $\tau_{1}$ and $\tau$. 
Lemmas \ref{lem4.1}-\ref{lem4.3} still apply even if $r << r_{1}$.
Now  Lemma \ref{lem4.3} says that $u$ is $C_{x}$-Lipschitz in
$B(x,\frac{\tau \rho}{4})$, and hence also on
$B(x,\frac{\tau \eta r}{8})$. By (\ref{eqn4.37})
we can take
\begin{equation}
\label{eqn5.27}
C_{x} = C \left(\tau^{-\frac{n}{2}}\o(x,\rho) + \rho^{\frac{\a}{2}}\right)
\leq C \left(\tau^{-\frac{n}{2}} \eta^{-\frac{n}{2}} \o(x,r) + r^{\frac{\a}{2}}\right).
\end{equation}
By Lemma \ref{lem4.3}, we even know that $u$ 
is $C^{1,\beta}$ in a neighborhood of $x$, thus Case 1 yields additional regularity.

In the two remaining cases, we set $r_{k} = \theta^k r = 3^{-k} r$
for $k \geq 0$, and our main task will be to control $\o(x,r_{k})$.
If the pair $(x,r_{k})$ ever satisfies (\ref{eqn5.24})
(the definition of Case 1), we denote by $k_{stop}$ the 
smallest integer $k$ such that $(x,r_{k})$ satisfies (\ref{eqn5.24})
(notice that $k \geq 1$ because we are not in Case 1); otherwise
set $k_{stop} = +\infty$.

Let $k < k_{stop}$ be given. If $(x,r_{k})$ satisfies (\ref{eqn5.25}),
we can apply Lemma \ref{lem5.1} to it 
(this is how we chose $\gamma$, $K_{1}$, and $r_{1}$ above), 
and we get that
\begin{equation}
\label{eqn5.28}
\o(x,r_{k+1}) \leq \beta \o(x,r_{k}),
\end{equation}
as in (\ref{eqn5.3}). Otherwise, $(x,r_{k})$ satisfies (\ref{eqn5.26})
(because (\ref{eqn5.25}) is false when $k < k_{stop}$), and we observe that
by definitions
\begin{equation}
\label{eqn5.29}
\o(x,r_{k+1}) 
= \left( \fint_{B(x,r_{k+1})} |\nabla u|^2 \right)^{\frac{1}{2}}
\leq 3^{\frac{n}{2}} \o(x,r_{k})
\leq 3^{\frac{n}{2}} K_{2}.
\end{equation}
By (\ref{eqn5.28}), (\ref{eqn5.29}), and an easy induction, we get that
for $0 \leq k \leq k_{stop}$,
\begin{equation}
\label{eqn5.30}
\o(x,r_{k}) \leq \mathrm{max}\left(\beta^k \o(x,r), 3^{\frac{n}{2}} K_{2}\right).
\end{equation}
If $k_{stop} = +\infty$, this implies that
\begin{equation}
\label{eqn5.31}
\limsup_{k\to \infty }\o(x,r_{k}) \leq  3^{\frac{n}{2}} K_{2}.
\end{equation}
In particular, if  $x$ is a Lebesgue point for $\nabla u$
\begin{equation}
\label{eqn5.32}
 |\nabla u(x)| \leq  3^{\frac{n}{2}} K_{2}.
\end{equation}
 
If $k_{stop} < \infty$, we apply our argument for Case 1 to the pair
$(x,r_{k_{stop}})$, and get that $u$ is $C^{1,\beta}$ in a neighborhood of $x$,
and by (\ref{eqn5.27}) and (\ref{eqn5.30})
\begin{eqnarray}
\label{eqn5.33}
 |\nabla u(x)| &\leq & 
 C \left(\tau^{-\frac{n}{2}} \eta^{-\frac{n}{2}} \o(x,r_{k_{stop}}) 
 + r_{k_{stop}}^{\frac{\a}{2}}\right)
 \nonumber\\
& \leq & 
C \tau^{-\frac{n}{2}} \eta^{-\frac{n}{2}}
 \mathrm{max}\left(\beta^{k_{stop}}\o(x,r), 3^{\frac{n}{2}} K_{2}\right)
 + C r^{\frac{\a}{2}}
 \\
 &\leq& C' \o(x,r) + C',
 \nonumber
\end{eqnarray}
where $C'$ depends on $n$, $\kappa$, $\alpha$ through 
the various constants above.
We still have (\ref{eqn5.33}) in Case 1 (directly by (\ref{eqn5.27})),
and since (\ref{eqn5.32}) is better than (\ref{eqn5.33}), we proved
that if $r \leq r_{2}$, (\ref{eqn5.33}) holds for almost every $x\in \O$
such that $B(x,2r) \i \O$.

Now let $x_{0} \in \O$ and $r_{0} < r_{2}$ be such that $B(x_{0},2r_{0}) \i \O$.
Then for almost every $x\in B(x_{0},r_{0})$, (\ref{eqn5.33}) holds for with 
$r = r_{0}/2$ (so that $B(x,2r) \i B(x_{0},2r_{0}) \i \O$), and so
\begin{equation}
\label{eqn5.34}
 |\nabla u(x)| 
 \leq  C' \o(x,r) + C'
\leq 2^{n/2} C' \o(x_{0},2r_{0}) + C'.
\end{equation}
We already know that $u$ is in the Sobolev space 
$W^{1,2}_{\loc}(B(x_{0},r_{0}))$, so we deduce 
from (\ref{eqn5.34}) that $u$ is Lipschitz in $B(x_{0},r_{0})$, 
with the estimate (\ref{eqn5.22}).
Theorem \ref{thm5.1} follows.
\qed
\end{proof}

\section{Almost monotonicity; statement and first estimates}

To prove that almost minimizers for $J$ are locally Lipschitz, we need a variant of the
monotonicity result of Alt, Caffarelli, and Friedman  \cite{ACF}. 
We use the functional $\Phi$ they introduced, but we shall only be 
able to prove that it is almost nondecreasing and has a limit at the origin; 
see Theorem \ref{thm6.1} and \eqref{eqn6.5}.

First we need some notation. As usual, $u$ is an almost
minimizer in the domain $\Omega$, we fix $x \in \Omega$,
and for $r > 0$ such that $B(x,r) \i \Omega$, set
\begin{equation}
\label{eqn6.1}
A_{+}(r) = \int_{B(x,r)} 
{\frac {|\nabla u^+(y)|^2}{|x-y|^{n-2}}} dy
\ \hbox{ and } \ 
A_{-}(r) = \int_{B(x,r)} 
{\frac {|\nabla u^{-}(y)|^2}{|x-y|^{n-2}}} dy
\end{equation}
and 
\begin{equation}
\label{eqn6.2}
\Phi(r) = {\frac {1}{r^4}} A_{+}(r) A_{-}(r) 
= {\frac {1}{r^4}} \left(\int_{B(x,r)} 
{\frac {|\nabla u^+(y)|^2}{|x-y|^{n-2}}} dy\right)
\left(\int_{B(x,r)} 
{\frac {|\nabla u^{-}(y)|^2}{|x-y|^{n-2}}} dy\right).
\end{equation}
We keep the same formula when $n=2$ (and some
proofs will be simpler), and recall that we do not consider
$n=1$ here.
We want to study the monotonicity properties of $\Phi$ as a function of $r$.
To some extent we follow the argument of  \cite{ACF},
but since we cannot expect $u$ to be as smooth, we 
need to avoid integration by parts. We can only expect
estimates that hold in average and weaker monotonicity properties.
Here is  the main result that we prepare for in this section
and prove in the next one.

\begin{theorem} \label{thm6.1}
Let $u$ be an almost minimizer for $J$ in $\O$,
and let $\delta$ be such that $0 < \delta  < \alpha/4(n+1)$.
Then there is a constant $C > 0$, which depends only
on $n$, $\alpha$, $\delta$, $\kappa$,
and $L^\infty$ bounds for $q_{+}$ and $q_{-}$,
such that the following holds.
Let $x\in \O$ and $r_{0} > 0$ be such that
$B(x,2r_{0}) \i \O$. Suppose that $u(x) = 0$.
Then for $0 < s < r < \frac{1}{2}\min(1,r_{0})$
\begin{equation}
\label{eqn6.4}
\Phi(s) \leq \Phi(r) + C(x,r_{0}) r^{\delta},
\end{equation}
where
\begin{equation}
\label{eqn6.3}
C(x,r_{0}) = C 
+ C \Big( \fint_{B(x,3r_{0}/2)} |\nabla u(y)|^2 \Big)^2
+ C((\log r_{0})_+)^4.
\end{equation}
 \end{theorem}

Of course none of the exponents above are expected
to be sharp.  Notice that 
Theorem~\ref{thm6.1} implies that if $u$ is an 
almost minimizer and $x\in \O$ is such that $u(x)=0$,
then
\begin{equation}
\label{eqn6.5}
\lim_{r \to 0} \Phi(r) \in [0,+\infty)
\ \hbox{ exists.}
\end{equation}
Indeed, first observe that by (\ref{eqn6.4}),
$\limsup_{s \to 0} \Phi(s) \leq \Phi(r) + C(x,r_{0}) r^{\delta}$
for $r$ small. So $0 \leq \liminf_{s \to 0} \Phi(s) 
\leq \limsup_{s \to 0} \Phi(s) < \infty$.
Set $l = \liminf_{s \to 0} \Phi(s)$. For each $\varepsilon > 0$, 
we can find $r > 0$, arbitrarily small, such that 
$\Phi(r) \leq l + \varepsilon$. Then by (\ref{eqn6.4}) again
$\limsup_{s \to 0} \Phi(s) \leq \sup_{0 < s < r} \Phi(s)
\leq l + \varepsilon + C(x,r_{0}) r^{\delta}$, which is
arbitrarily close to $l$ as it holds for all $r<r_0$;  (\ref{eqn6.5}) follows.

\medskip
We start the proof with a few computations using competitors.
In the case of minimizers, these calculations would often be obtained using
integrations by parts.
The next lemma, which will be obtained by replacing $u^\pm$ 
with something like $u^\pm+\lambda \varphi u^\pm$, will 
be used a few times later, with different choices for $\varphi$.

\begin{lemma} \label{lem6.1}
Let $u$ be an almost minimizer for $J$ in $\O$,
and assume that $B(x,2r) \i \O$. Let 
$\varphi \in W^{1,2}(\O) \cap C(\O)$
be such that $\varphi(y) \geq 0$ everywhere,
$\varphi(y) = 0$ on $\O \sm B(x,r)$,
and let $\lambda \in \R$ be such that
\begin{equation}
\label{eqn6.6}
|\lambda \varphi(y)| < 1 \ \hbox{ on } \O.
\end{equation}
Then for each choice of sign $\pm$,
\begin{eqnarray}
\label{eqn6.7}
0 \leq \kappa r^\alpha J_{x,r}(u)
&+& 2\lambda \left[
\int_{B(x,r)} \varphi \, |\nabla u^\pm|^2
+  \int_{B(x,r)} 
u^\pm \langle \nabla u^\pm,\nabla \varphi \rangle
\right]
 \nonumber
\\
&+& \lambda^2 \left[
\int_{B(x,r)} \varphi^2 \, |\nabla u^\pm|^2
+  (u^\pm)^2 |\nabla \varphi|^2 
+ 2\varphi u^\pm \langle \nabla u^\pm,\nabla \varphi \rangle
\right],
\end{eqnarray}
where $J_{x,r}(u) = \int_{B(x,r)}|\nabla u|^2
+q^2_+ \, \chi_{\{u>0\}} + q^2_- \, \chi_{\{u<0\}}$
is as in (\ref{eqn1.12}), and $\nabla u^\pm$
denotes the gradient of $u^\pm$.
\end{lemma}

\begin{proof}
Recall that $u^+=\max\{u,0\}$ and $u^-=\max\{-u,0\}$. Let us do the verification for $u^+$, the proof
for $u^-$ would be similar. Define $v$ on $\O$ by
\begin{equation}
\label{eqn6.8}
v(y) = u(y) + \lambda \varphi(y)u(y)
= (1+\lambda\varphi(y)) u^+(y)
\ \hbox{ if $y\in B(x,r)$ and $u(y) > 0$}
\end{equation}
and $v(x) = u(x)$ otherwise.
Recall that $u$ is continuous, and then $v$ is also continuous, because $\varphi(y) = 0$ on
$\O \sm B(x,r)$, and the expression in (\ref{eqn6.8}) yields
$v(x) = 0$ when $u(x)=0$. Next, $v(y)$ is obtained
from $u(y)$ by multiplying it by either $1$ or $1+\lambda\varphi(y)$,
which is also positive by (\ref{eqn6.6}); so
\begin{equation}
\label{eqn6.9}
\big\{ y\in \O \,;\, v(y) > 0 \big\}
= \big\{ y\in \O \,;\, u(y) > 0 \big\} 
\ \hbox{ and } \ 
\big\{ y\in \O \,;\, v(y) < 0 \big\}
= \big\{ y\in \O \,;\, u(y) < 0 \big\}.
\end{equation}
In addition, $v^- = u^-$ and $v^+ = (1+\lambda\varphi) u^+$
everywhere on $\O$  (recall that $\varphi(y) = 0$ on
$\O \sm B(x,r)$). We know, for instance from 
Corollary 2.1.8 in \cite{Z}, 
that $u^-$ and $u^+$ lie in $W^{1,2}(\O)$,
and that $\nabla u^\pm =  \chi_{\{\pm u>0\}} \nabla u$.
Then $v^\pm \in W^{1,2}(\O)$, with 
\begin{equation}
\label{eqn6.10}
\nabla v^+ = (1+\lambda\varphi) \nabla u^+
+ \lambda u^+ \nabla \varphi.
\end{equation}
Thus $v\in W^{1,2}(\O)$, and we can apply the definition
(\ref{eqn1.11}) of almost minimality. This yields
\begin{equation}
\label{eqn6.11}
J_{x,r}(u) \leq (1+\kappa r^\alpha) J_{x,r}(v)
= (1+\kappa r^\alpha) \int_{B(x,r)}|\nabla v|^2
+q^2_+ \, \chi_{\{v>0\}} + q^2_- \, \chi_{\{v<0\}}
\end{equation}
by \eqref{eqn1.12}. 
Now $q^2_+ \, \chi_{\{v>0\}} + q^2_- \, \chi_{\{v<0\}}
= q^2_+ \, \chi_{\{u>0\}} + q^2_- \, \chi_{\{u<0\}}$
by (\ref{eqn6.9}). Also, 
\begin{eqnarray}
\label{eqn6.12}
\int_{B(x,r)}|\nabla v|^2 
&=&  \int_{B(x,r)}|\nabla v^+|^2 + \int_{B(x,r)}|\nabla v^-|^2 
= \int_{B(x,r)}|\nabla v^+|^2 + \int_{B(x,r)}|\nabla u^-|^2
\nonumber 
\\
&=&  \int_{B(x,r)}|\nabla u|^2 
- \int_{B(x,r)}|\nabla u^+|^2 + \int_{B(x,r)}|\nabla v^+|^2
\end{eqnarray}
because $\nabla v^\pm =  \chi_{\{\pm v>0\}} \nabla v$,
similarly for $u$, and by (\ref{eqn6.8}) again. So
\begin{equation}
\label{eqn6.13}
J_{x,r}(v) = J_{x,r}(u) + \int_{B(x,r)}|\nabla v^+|^2
- \int_{B(x,r)}|\nabla u^+|^2 
\end{equation}
and (\ref{eqn6.11}) yields
\begin{eqnarray}
\label{eqn6.14}
0 &\leq& (1+\kappa r^\alpha) J_{x,r}(v) - J_{x,r}(u)
= \kappa r^\alpha J_{x,r}(v) + J_{x,r}(v) - J_{x,r}(u)
\nonumber
\\
&=& \kappa r^\alpha J_{x,r}(v) + \int_{B(x,r)}|\nabla v^+|^2
- \int_{B(x,r)}|\nabla u^+|^2 
\\
&=& \kappa r^\alpha J_{x,r}(u) + 
(1+ \kappa r^\alpha) \int_{B(x,r)} \big[|\nabla v^+|^2 -|\nabla u^+|^2 \big].
\nonumber
\end{eqnarray}
Next by (\ref{eqn6.10})
\begin{eqnarray}
\label{eqn6.15}
|\nabla v^+|^2 
&=& (1+\lambda\varphi)^2  |\nabla u^+|^2
+ 2 \lambda (1+\lambda\varphi)  u^+ \langle \nabla u^+,\nabla \varphi\rangle
+ \lambda^2 (u^+)^2 |\nabla \varphi|^2
\nonumber
\\
&=& |\nabla u^+|^2 + 2\lambda \big[\varphi |\nabla u^+|^2 + 
 u^+ \langle \nabla u^+,\nabla \varphi\rangle\big]
\nonumber
\\
&\,& \hskip 2cm
+ \lambda^2 \big[\varphi^2 |\nabla u^+|^2 
+ 2 \varphi  u^+ \langle \nabla u^+,\nabla \varphi\rangle
+ (u^+)^2 |\nabla \varphi|^2 \big]
\end{eqnarray}
by (\ref{eqn6.10}). We integrate this, replace in (\ref{eqn6.14}),
and get that
\begin{eqnarray}
\label{eqn6.16}
0 \leq \kappa r^\alpha J_{x,r}(u)
&+& 2\lambda (1+\kappa r^\alpha) \left[
\int_{B(x,r)} \varphi \, |\nabla u^+|^2
+  \int_{B(x,r)} 
u^+ \langle \nabla u^+,\nabla \varphi \rangle \right]
 \nonumber
\\
&+& \lambda^2 (1+\kappa r^\alpha) \left[
\int_{B(x,r)} \varphi^2 \, |\nabla u^+|^2
+  (u^+)^2 |\nabla \varphi|^2 
+ 2\varphi u^+ \langle \nabla u^+,\nabla \varphi \rangle
\right].
\end{eqnarray}
We divide by $1+\kappa r^\alpha$, add
$ \kappa r^\alpha J_{x,r}(u) 
- {\frac{\kappa r^\alpha}{1+\kappa r^\alpha}} J_{x,r}(u) \geq 0$,
and get (\ref{eqn6.7}) for $u^+$. The proof for $u^-$ is the same.
\qed
\end{proof}

The next lemma will only be needed in ambient dimensions $n \geq 3$;
when $n=2$, the algebra leading to the monotonicity formula will just be simpler.

\begin{lemma} \label{lem6.2}
Suppose $n \geq 3$.
Let $u$ be an almost minimizer for $J$ in $\O$, assume that $B(x,2r_{0}) \i \O$ and that $u(x) = 0$. 
Then for $0 < r < \min(1,r_{0})$,  
\begin{eqnarray}
\label{eqn6.17}
{\frac{c_{n}}{r^2}} \, A_{+}(r) 
&-& {\frac{1}{n(n-2)}} \fint_{B(x,r)} |\nabla u^+|^2
- {\frac{1}{2}} \fint_{\partial B(x,r)} \Big({\frac{u^+}{r}}\Big)^2
\nonumber
\\
&\geq& - C r^{\frac{\alpha}{n+1}} \,
\Big(1+ \fint_{B(x,3r_{0}/2)} |\nabla u|^2 +  \log^2(r_{0}/r) 
+ \log^2(1/r) \Big) , 
\end{eqnarray}
where we set
\begin{equation}
\label{eqn6.18}
c_{n} = {\frac{1}{n(n-2) \omega_{n}}}
\ \hbox{ and } \ 
\omega_{n} = |B(0,1)|,
\end{equation}
and $C$ depends only on $\kappa$, the $||q_{\pm}||_{\infty}$,
$\alpha$, and $n$.
\end{lemma}

This will be generalized later, with a similar proof; see
Lemma \ref{lem6.3}.

\begin{proof}
We use the Green function for $B(x,r)$ with pole $x$ defined by 
\begin{equation}
\label{eqn6.19}
G_{r}(y) = {\frac{c_{n}}{|y-x|^{n-2}}} -  {\frac{c_{n}}{r^{n-2}}}.
\end{equation}
Fix $s < r$, and set

\begin{equation}\label{eqn6.20}
\varphi(y)=\varphi_{r,s}(y)=\left\{\begin{array}{rcll} 
& 0 
&\ \hbox{ for } y \in \Omega \sm B(x,r)
\\
& G_{r}(y) 
&\ \hbox{ for } y \in B(x,r)\sm B(x,s)
\\
& 
c_{n} s^{2-n} - c_{n} r^{2-n}
&\ \hbox{ for } y \in B(x,s).
\end{array}
\right.
\end{equation}

We want to apply Lemma \ref{lem6.1} to $\varphi$, with
a $\lambda > 0$  to be chosen soon. First observe
that $\varphi \in W^{1,2}(\O) \cap C(\O)$ and 
$\varphi(y) = 0$ on $\O \sm B(x,r)$, as needed. Also,
\begin{equation}
\label{eqn6.21}
||\varphi||_{\infty} \leq {\frac{c_{n}}{s^{n-2}}} 
\ \hbox{ and } \ 
||\nabla \varphi||_{\infty} \leq {\frac{c_{n} (n-2)}{s^{n-1}}} 
\end{equation}
so that we just need to take
\begin{equation}
\label{eqn6.22}
\lambda < {\frac{s^{n-2}}{c_{n}}}
\end{equation}
for (\ref{eqn6.6}) to hold. Then the lemma applies
and (\ref{eqn6.7}) yields 
\begin{eqnarray}
\label{eqn6.23}
0 \leq \kappa r^\alpha J_{x,r}(u)
&+& 2\lambda \left[
\int_{B(x,r)} \varphi \, |\nabla u^+|^2
+  \int_{B(x,r)} 
u^+ \langle \nabla u^+,\nabla \varphi \rangle
\right]
 \nonumber
\\
&+& 2\lambda^2 \left[
\int_{B(x,r)} \varphi^2 \, |\nabla u^+|^2
+  (u^+)^2 |\nabla \varphi |^2 \right],
\end{eqnarray}
where we just applied Cauchy-Schwarz to take care
of $ 2\varphi u^+ \langle \nabla u^+,\nabla \varphi 
\rangle$ in the last bracket.
We now estimate the various terms. First
\begin{eqnarray}
\label{eqn6.24}
\int_{B(x,r)} \varphi \, |\nabla u^+|^2
&=&
c_{n} \int_{B(x,s)} 
\left({\frac{1}{s^{n-2}}} -  {\frac{1}{r^{n-2}}}\right)
\, |\nabla u^+|^2
\nonumber
\\
&\,& \hskip 2cm
+c_{n} \int_{B(x,r) \sm B(x,s)} 
\left({\frac{1}{|x-y|^{n-2}}} -  {\frac{1}{r^{n-2}}}\right)
\, |\nabla u^+|^2
\\
&\leq& c_{n} A_{+}(r) - {\frac{c_{n}}{r^{n-2}}}
\int_{B(x,r)}  |\nabla u^+|^2
\nonumber
\end{eqnarray}
by (\ref{eqn6.1}). Let us check now that
\begin{equation}
\label{eqn6.25}
2 \int_{B(x,r)} 
u^+ \langle \nabla u^+,\nabla \varphi \rangle
= 2 \int_{B(x,r) \sm B(x,s)} 
u^+ \langle \nabla u^+,\nabla G_{r} \rangle
= \fint_{\partial B(x,s)} (u^+)^2 - \fint_{\partial B(x,r)} (u^+)^2
\end{equation}
The first part folllows from (\ref{eqn6.20}).
The second one will come from the fact that $G_{r}$ is a Green function.
In fact for $0 < \rho < 2 r$,
Set 
\begin{equation}
\label{eqn6.26}
F(\rho) = \fint_{\partial B(x,\rho)} (u^+)^2
= \fint_{\partial B(0,1)} u^+(x+\rho\theta)^2.
\end{equation}
We know that $u\in W^{1,2}(B(x,2r)) \cap C(B(x,2r))$,
and hence $(u^+)^2 \in W^{1,1}(B(x,3r/2))$, which
implies that $F \in W^{1,1}((0,3r/2))$, with a derivative
\begin{equation}
\label{eqn6.27}
F'(\rho) = 2 \fint_{\partial B(0,1)} u^+(x+\rho\theta)
\,\frac{ \partial u^+}{\partial \rho}(x+\rho\theta)
= \frac{ 2}{|\partial B(0,1)| \rho^{n-1}}
\int_{\partial B(x,\rho)} u^+(y) \,\frac{ \partial u^+}{\partial \rho}(y).
\end{equation}
At the same time,
\begin{equation}
\label{eqn6.28}
2\int_{\partial B(x,\rho)} u^+ 
\langle \nabla u^+,\nabla G_{r} \rangle
= - 2 \int_{\partial B(x,\rho)} u^+(y) 
\,\frac{ \partial u^+}{\partial \rho}(y)
\, {\frac{c_{n} (n-2)}{|y-x|^{n-1}}} = -F'(\rho),
\end{equation}
because
$\frac{ 1}{|\partial B(0,1)|}=c_{n} (n-2)$
(see definition (\ref{eqn6.18})). Now (\ref{eqn6.25})
follows from (\ref{eqn6.26}) and (\ref{eqn6.28}) because $F \in W^{1,1}((0,3r/2))$.

We estimate the $\lambda^2$ term in (\ref{eqn6.23}) using (\ref{eqn6.21}):
\begin{equation}
\label{eqn6.29}
\int_{B(x,r)} \varphi^2 \, |\nabla u^+|^2 
+  (u^+)^2 |\nabla \varphi |^2
\leq\frac{c_{n}^2}{s^{2n-4}} \int_{B(x,r)}  |\nabla u^+|^2
+  \frac{c_{n}^2 (n-2)^2}{s^{2n-2}} \int_{B(x,r)} (u^+)^2.
\end{equation}
Combining (\ref{eqn6.23}), (\ref{eqn6.24}), (\ref{eqn6.25}) and (\ref{eqn6.29}) we get that
\begin{eqnarray}
\label{eqn6.30}
0 \leq \kappa r^\alpha J_{x,r}(u)
&+& 2\lambda \left[
c_{n} A_{+}(r) - {\frac{c_{n}}{r^{n-2}}}
\int_{B(x,r)}  |\nabla u^+|^2 \right]
+ \lambda \left[
\fint_{\partial B(x,s)} (u^+)^2 - \fint_{\partial B(x,r)} (u^+)^2
\right]
 \nonumber
\\
&+& 2\lambda^2 \left[
\frac{c_{n}^2}{s^{2n-4}} \int_{B(x,r)}  |\nabla u^+|^2
+  \frac{c_{n}^2 (n-2)^2}{s^{2n-2}} \int_{B(x,r)\backslash B(x,s)} (u^+)^2 
\right].
\end{eqnarray}
Set
\begin{equation}
\label{eqn6.31}
A = {\frac{c_{n}}{r^2}} \, A_{+}(r) 
- {\frac{1}{n(n-2)}} \fint_{B(x,r)} |\nabla u^+|^2
- {\frac{1}{2}} \fint_{\partial B(x,r)} \Big({\frac{u^+}{r}}\Big)^2;
\end{equation}
this is the quantity that we want to estimate
in (\ref{eqn6.17}). Recall from (\ref{eqn6.18})
that $c_{n} = {\frac{1}{n(n-2) \omega_n}}$; thus
\begin{equation}
\label{eqn6.32}
A = {\frac{c_{n}}{r^2}} \, A_{+}(r) 
- c_{n} r^{-n} \int_{B(x,r)} |\nabla u^+|^2
- {\frac{1}{2 r^2}} \fint_{\partial B(x,r)} \big(u^+\big)^2;
\end{equation}
we move the corresponding pieces of (\ref{eqn6.30})
to the right, divide by $2\lambda r^2$, and get that
\begin{eqnarray}
\label{eqn6.33}
-A \leq \frac{\kappa r^\alpha J_{x,r}(u)}{2\lambda r^2}
&+& \frac{1}{2r^2}
\fint_{\partial B(x,s)} (u^+)^2  \nonumber
\\
&+& \lambda r^{-2}\left[
\frac{c_{n}^2}{s^{2n-4}} \int_{B(x,r)}  |\nabla u^+|^2
+  \frac{c_{n}^2 (n-2)^2}{s^{2n-2}} \int_{B(x,r)\backslash B(x,s)} (u^+)^2 
\right].
\end{eqnarray}
Recall from (\ref{eqn2.20}) that
\begin{equation}
\label{eqn6.34}
|u(y)-u(z)| \leq C |y-z| \Big( \o(x,3r_{0}/2) + \log\frac{r_{0}}{|y-z|}\Big)
\end{equation}
when $y,z\in B(x,r_0)$ are such that $|y-z| \leq r_{0}/8$.
We apply this with $z=x$ (and possibly a few intermediate
points if $|y-x| > r_{0}/8$) and get that
\begin{equation}
\label{eqn6.35}
|u(y)| = |u(y)-u(x)|
\leq C |y-x| \Big( \o(x,3r_{0}/2) + \log\frac{2r_{0}}{|y-x|}\Big)
\end{equation}
for $y\in B(x,r)$, because we assumed that $u(x)=0$. Hence
\begin{equation}
\label{eqn6.36}
\fint_{\partial B(x,s)} (u^+)^2 \leq C s^2 \Big( \o(x,3r_{0}/2) + \log\frac{r_{0}}{s}\Big)^2
\end{equation}
and
\begin{equation}
\label{eqn6.37}
\int_{B(x,r)\backslash B(x,s)} (u^+)^2 
\leq C r ^{n+2}\Big( \o(x,3r_{0}/2) + \log\frac{r_{0}}{s}\Big)^2.
\end{equation}
By (\ref{eqn2.6}) and (\ref{eqn2.11})
\begin{equation}
\label{eqn6.38}
\int_{B(x,r)}  |\nabla u^+|^2
\leq \int_{B(x,r)}  |\nabla u|^2
= |B(x,r)| \o(x,r)^2
\leq C r^n \big(\o(x,r_{0}) +  \log(r_{0}/r)\big)^2.
\end{equation}
In all these estimates (starting with (\ref{eqn2.11})
and (\ref{eqn2.20})), $C$ depends only on 
$\kappa$, $\|q_+\|_{L^\infty}$, $\|q_-\|_{L^\infty}$, 
$\alpha$, and $n$.

It is now time to choose $\lambda$ and $s$.
We take
\begin{equation}
\label{eqn6.39}
\lambda = c_{n}^{-1} r^{n-2+\beta}, \   \beta = \frac{n \alpha}{n+1},
\ \hbox{ and } \  s = r^{1+\frac{\beta}{2n}}.
\end{equation}
Since $r < 1$ and $n-2+\beta > (n-2)(1+\frac{\beta}{2n})$ then (\ref{eqn6.22}) holds.
Thus Lemma \ref{lem6.1} applies, and we get (\ref{eqn6.33}).
Let us now estimate each term in (\ref{eqn6.33}). First,
\begin{eqnarray}
\label{eqn6.40}
\frac{\kappa r^\alpha J_{x,r}(u)}{2\lambda r^2}
&=& \frac{\kappa r^\alpha}{2\lambda r^2} \,
\int_{B(x,r)}  |\nabla u|^2 + q^2_+ \, \chi_{\{u>0\}} 
+ q^2_- \, \chi_{\{u<0\}}
\nonumber
\\
&\leq&
 \frac{\kappa r^\alpha}{2\lambda r^2} \,
\Big[
C r^n \big(\o(x,r_{0}) +  \log(r_{0}/r)\big)^2
+ C r^n (||q_{+}||^2_{\infty} + ||q_{-}||^2_{\infty})
\Big]
\\
&\leq& \frac{C r^n \kappa r^\alpha}{\lambda r^2} \,
\big(1+ \o(x,r_{0}) +  \log(r_{0}/r)\big)^2
\nonumber
\end{eqnarray}
by (\ref{eqn6.38}). Given the choice of $\lambda$ in (\ref{eqn6.39}), the exponent of $r$ is 
$n+\alpha - (n-2-\beta) - 2 = \alpha-\beta = \frac{\alpha}{n+1}$,
so this first term is in accordance with (\ref{eqn6.17}).
Next 
\begin{eqnarray}
\label{eqn6.41}
r^{-2} \fint_{\partial B(x,s)} (u^+)^2
&\leq& C r^{-2} s^2 \Big( \o(x,3r_{0}/2) + \log\frac{r_{0}}{s}\Big)^2
\nonumber
\\
&\leq& C r^{-2} r^{2+\frac{\beta}{n}}
\Big( \o(x,3r_{0}/2) +\log\frac{r_{0}}{r} +\log\frac{r}{s}\Big)^2
\\
\nonumber
\end{eqnarray}
by (\ref{eqn6.36}); the exponent is again $\frac{\beta}{n}
= \frac{\alpha}{n+1}$, and $\log\frac{r}{s} = 
\log(r^{-\frac{\beta}{2n}}) = \frac{\beta}{2n}\log(1/r)$,
so this term fits with (\ref{eqn6.17}) too.
By (\ref{eqn6.38}) the third term in (\ref{eqn6.33}) is
\begin{equation}
\label{eqn6.42}
\lambda r^{-2}
\frac{c_{n}^2}{s^{2n-4}} \int_{B(x,r)}  |\nabla u^+|^2
\leq C r^{n-2+\beta} r^{-2} r^{-(2n-4)(1+\frac{\beta}{2n})}
r^n \big(\o(x,r_{0}) +  \log(r_{0}/r)\big)^2.
\end{equation}
The power is 
$2n-4 +\beta -(2n-4)(1+\frac{\beta}{2n})
= \beta -(2n-4)\frac{\beta}{2n} = \frac{2\beta}{n}
=\frac{2\alpha}{n+2} >  \frac{\alpha}{n+1}$,
which is again all right. By (\ref{eqn6.37}) the last term in (\ref{eqn6.33}) is
\begin{equation}
\label{eqn6.43}
\lambda r^{-2}  \frac{c_{n}^2 (n-2)^2}{s^{2n-2}} \int_{B(x,r)\backslash B(x,s)} (u^+)^2 
\leq C r^{n-2+\beta} r^{-2} r^{-(2n-2)(1+\frac{\beta}{2n})}
r ^{n+2}\Big( \o(x,3r_{0}/2) + \log\frac{r_{0}}{s}\Big)^2.
\end{equation}
The power is 
$2n-2+\beta-(2n-2)(1+\frac{\beta}{2n}) = \beta(1-\frac{2n-2}{2n})
=\frac{2\beta}{n}$, like the previous one, and this last term
fits with (\ref{eqn6.17}). 
The terms $\log{r_0/s}$ is handled as in (\ref{eqn6.41}). Thus
this proves (\ref{eqn6.17}) 
and Lemma \ref{lem6.2}.
\qed
\end{proof}

The following generalization of Lemma \ref{lem6.2}
contains three other cases, that will be treated similarly.

\begin{lemma} \label{lem6.3}
Still assume that $n \geq 3$.
Let $u$ be an almost minimizer for $J$ in $\O$,
and assume that $B(x,2r_{0}) \i \O$ and that $u(x) = 0$. 
Then for $0 < r < \min(1,r_{0})$ and each choice of sign
$\pm$,
\begin{eqnarray}
\label{eqn6.44}
\Big| {\frac{c_{n}}{r^2}} \, A_{\pm}(r) 
&-& {\frac{1}{n(n-2)}} \fint_{B(x,r)} |\nabla u^\pm|^2
- {\frac{1}{2}} \fint_{\partial B(x,r)} \Big({\frac{u^\pm}{r}}\Big)^2
\Big|
\nonumber
\\
&\leq& C r^{\frac{\alpha}{n+1}} \,
\Big(1+ \fint_{B(x,3r_{0}/2)} |\nabla u|^2 +  \log^2(r_{0}/r) 
+ \log^2(1/r) \Big) .
\end{eqnarray}
Here again $c_{n} = \big[n(n-2)\omega_n\big]^{-1}$
as in (\ref{eqn6.18}) 
and $C$ depends only on $\kappa$, the $||q_{\pm}||_{\infty}$,
$\alpha$, and $n$.
\end{lemma}

\begin{proof}
We continue with the case of $A_{+}$ and $u^+$,
and prove the upper bound. We still apply 
Lemma \ref{lem6.1}, but this time with 
$\lambda' = - \lambda = - c_{n}^{-1} r^{n-2+\beta}$ in 
(\ref{eqn6.39}). The requirement that $|\lambda' \varphi| < 1$
is still satisfied, so we get (\ref{eqn6.23}) with $-\lambda$.
That is,
\begin{eqnarray}
\label{eqn6.45}
0 \leq \kappa r^\alpha J_{x,r}(u)
&-& 2\lambda \left[
\int_{B(x,r)} \varphi \, |\nabla u^+|^2
+  \int_{B(x,r)} 
u^+ \langle \nabla u^+,\nabla \varphi \rangle
\right]
 \nonumber
\\
&+& 2\lambda^2 \left[
\int_{B(x,r)} \varphi^2 \, |\nabla u^+|^2
+  (u^+)^2 |\nabla \varphi |^2 \right].
\end{eqnarray}
Because of (\ref{eqn6.25}), this is the same as
\begin{eqnarray}
\label{eqn6.46}
0 \leq \kappa r^\alpha J_{x,r}(u)
&-& \lambda \left[
2\int_{B(x,r)} \varphi \, |\nabla u^+|^2
+ \fint_{\partial B(x,s)} (u^+)^2 - \fint_{\partial B(x,r)} (u^+)^2
\right]
\nonumber
\\
&+& 2\lambda^2 \left[
\int_{B(x,r)} \varphi^2 \, |\nabla u^+|^2
+  (u^+)^2 |\nabla \varphi |^2 \right].
\end{eqnarray}
or equivalently (since $\lambda > 0$)
\begin{eqnarray}
\label{eqn6.47}
2\int_{B(x,r)} \varphi \, |\nabla u^+|^2
- \fint_{\partial B(x,r)} (u^+)^2
&\leq&
-\fint_{\partial B(x,s)} (u^+)^2
+ \frac{\kappa r^\alpha J_{x,r}(u)}{\lambda}
\nonumber
\\
&+& 2\lambda \left[
\int_{B(x,r)} \varphi^2 \, |\nabla u^+|^2
+  (u^+)^2 |\nabla \varphi |^2 \right].
\end{eqnarray}
This time we can even drop $\fint_{\partial B(x,s)} (u^+)^2$
(which was small anyway), and get that
\begin{eqnarray}
\label{eqn6.48}
2\int_{B(x,r)} \varphi \, |\nabla u^+|^2
- \fint_{\partial B(x,r)} (u^+)^2
&\leq&
\frac{\kappa r^\alpha J_{x,r}(u)}{\lambda}
\nonumber
\\
&\,& \hskip1cm
+ 2\lambda \left[
\int_{B(x,r)} \varphi^2 \, |\nabla u^+|^2
+  (u^+)^2 |\nabla \varphi |^2 \right].
\end{eqnarray}
By the first part of (\ref{eqn6.24}),
\begin{eqnarray}
\label{eqn6.49}
\int_{B(x,r)} \varphi \, |\nabla u^+|^2
&=&
c_{n} \int_{B(x,s)} 
\left({\frac{1}{s^{n-2}}} -  {\frac{1}{r^{n-2}}}\right)
\, |\nabla u^+|^2
\nonumber
\\
&\,& \hskip 1.5cm
+c_{n} \int_{B(x,r) \sm B(x,s)} 
\left({\frac{1}{|x-y|^{n-2}}} -  {\frac{1}{r^{n-2}}}\right)
\, |\nabla u^+|^2.
\end{eqnarray}
Set
\begin{equation}
\label{eqn6.50}
\Delta = c_{n} \int_{B(x,s)} 
\left({\frac{1}{|x-y|^{n-2}}} -  {\frac{1}{s^{n-2}}}\right)
\, |\nabla u^+|^2;
\end{equation}
then
\begin{eqnarray}
\label{eqn6.51}
\int_{B(x,r)} \varphi \, |\nabla u^+|^2
&=&
c_{n} \int_{B(x,r)} 
\left({\frac{1}{|x-y|^{n-2}}} -  {\frac{1}{r^{n-2}}}\right)
\, |\nabla u^+|^2 - \Delta
\\
&=& c_{n} A_{+}(r) - {\frac{c_{n}}{r^{n-2}}}
\int_{B(x,r)}  |\nabla u^+|^2 - \Delta
\nonumber
\end{eqnarray}
and (\ref{eqn6.48}) yields
\begin{eqnarray}
\label{eqn6.52}
2c_{n} A_{+}(r) - {\frac{2c_{n}}{r^{n-2}}}
\int_{B(x,r)}  |\nabla u^+|^2
- \fint_{\partial B(x,r)} (u^+)^2
&\,&
\nonumber
\\
&\,& \hskip-5cm
\leq
\frac{\kappa r^\alpha J_{x,r}(u)}{\lambda} + \Delta
+ 2\lambda \left[
\int_{B(x,r)} \varphi^2 \, |\nabla u^+|^2
+  (u^+)^2 |\nabla \varphi |^2 \right].
\end{eqnarray}
The left-hand side of (\ref{eqn6.52}) is $2r^2 A$
(see (\ref{eqn6.32})), so
\begin{equation}
\label{eqn6.53}
A \leq
\frac{\kappa r^\alpha J_{x,r}(u)}{2 r^2 \lambda} 
+ \frac{\Delta}{2r^2}
+ \lambda r^{-2} \left[
\int_{B(x,r)} \varphi^2 \, |\nabla u^+|^2
+  (u^+)^2 |\nabla \varphi |^2 \right].
\end{equation}
We just need to estimate $\Delta$, since the other
terms are the same as before, and were already estimated 
for Lemma \ref{lem6.2}. But, with the notation of
(\ref{eqn2.6}) and by (\ref{eqn2.11}) (used twice),
\begin{eqnarray}
\label{eqn6.54}
\, \hskip1cm 
\Delta &\leq&
c_{n} \int_{B(x,s)} 
{\frac{1}{|x-y|^{n-2}}} \, |\nabla u^+|^2
\leq c_{n} \sum_{j \geq 0} \int_{B(x,2^{-j}s) \sm B(x,2^{-j-1}s)} 
{\frac{1}{|x-y|^{n-2}}} \, |\nabla u|^2
\nonumber
\\
&\leq&
c_{n} \sum_{j \geq 0} 2^{(j+1)(n-2)} s^{2-n} \int_{B(x,2^{-j}s)} 
 |\nabla u|^2
\leq C \sum_{j \geq 0} 2^{(j+1)(n-2)} s^{2-n} 2^{-nj}s^n
\fint_{B(x,2^{-j}s)}  |\nabla u|^2
\nonumber
\\
&=& C s^2 \sum_{j \geq 0} 2^{-2j} \o(x,2^{-j}s)^2
\leq C s^2 \sum_{j \geq 0} 2^{-2j} 
\big(\o(x,s) + j \big)^2
\\
&\leq& C s^2 \big(\o(x,s) + 1 \big)^2
\leq C s^2 \big(\o(x,3r_{0}/2) + 1 + \log\frac{3r_{0}}{2s} \big)^2.
\nonumber
\end{eqnarray}
Thus
\begin{eqnarray}
\label{eqn6.55}
\frac{\Delta}{2r^2} 
&\leq& C r^{-2} s^2 
\big(\o(x,3r_{0}/2) + 1 + \log\frac{3r_{0}}{2s} \big)^2
\nonumber
\\
&\leq& C  r^{\frac{\beta}{n}}
\big(\o(x,3r_{0}/2) + 1 + \log\frac{r_{0}}{r} 
+ \log\frac{r}{s} \big)^2
\\
&=&
C  r^{\frac{\alpha}{n+1}}
\big(\o(x,3r_{0}/2) + 1 + \log\frac{r_{0}}{r} 
+ \frac{\beta}{2n}\log\frac{1}{r} \big)^2
\nonumber
\end{eqnarray}
because $s =  r^{1+\frac{\beta}{2n}}$.
and $\beta= \frac{n\alpha}{n+1}$. Again this is
dominated by the right-hand side of (\ref{eqn6.44}),
and this completes our proof of (\ref{eqn6.44}) for
$u^+$. The proof for $u^-$ is the same: just replace
$+$ by $-$ in all the proof, or apply the result
for $u^+$ to $-u$, with $q_{+}$ and $q_{-}$ exchanged.
This proves Lemma \ref{lem6.3}.
\qed
\end{proof}

In the next lemma we return to the general case of $n \geq 2$. 
We would have preferred to obtain an estimate 
on the difference of integrals (as in the case of minimizers), 
but unfortunately we only get an estimate on the average of the integrals on a small interval near
$r$. Notice that in (\ref{eqn6.56}) we shall not even 
get the absolute values to be inside the $s$-integral.

\begin{lemma} \label{lem6.4}
Let $u$ be an almost minimizer for $J$ in $\O$,
and assume that $B(x,2r_{0}) \i \O$ and that $u(x) = 0$. 
For $0 < r \leq \frac{1}{2} \, \min(1,r_{0})$, set 
$t = t(r) = (1-\frac{r^{\alpha/4}}{10})\, r$.
Then for $0 < r < \min(1/2,r_{0})$ and each choice 
of sign $\pm$,
\begin{eqnarray}
\label{eqn6.56}
\left| 
\fint_{t(r)}^{r}
\Big(
\int_{B(x,s)} |\nabla u^\pm(y)|^2 dy \Big) \, ds
-\fint_{t(r)}^{r}
\Big(\int_{\partial B(x,s)} 
u^\pm \, \frac{\partial u^\pm}{\partial n}
\Big) \, ds
\right| &\,&
\nonumber
\\
&\,& \hskip-6cm
\leq C r^{n+\frac{\alpha}{4}}
\Big(1 + \fint_{B(x,3r_{0}/2)} |\nabla u|^2 
+ \log^2\frac{r_{0}}{r}
\Big) ,
\end{eqnarray}
were $\frac{\partial u^\pm}{\partial n}$ denotes the 
radial dervative of $u^\pm$,
and $C$ depends only on $\kappa$, the $||q_{\pm}||_{\infty}$,
$\alpha$, and $n$.
\end{lemma}

\begin{proof}
We want to apply Lemma \ref{lem6.1} with
yet another choice of function $\varphi$. 
We choose a radial cut-off function such that
\begin{equation}
\label{eqn6.57}
\varphi(y)=0 \hbox{ for } y\in \O \sm B(x,r)
\ \hbox{ and } \ 
\varphi(y)=1 \hbox{ for } y\in  B(x,t)
\end{equation}
In the remaining annulus, we interpolate linearly, 
i.e. set
\begin{equation}
\label{eqn6.58}
\varphi(y)= \frac{r-|y-x|}{r-t} 
\ \hbox{ for } y\in B(x,r)\sm B(x,t). 
\end{equation}
The assumptions of Lemma \ref{lem6.1} are satisfied as long as 
$|\lambda| < 1$. In this case (\ref{eqn6.7})  
yields
\begin{eqnarray}
\label{eqn6.59}
-2\lambda \left[
\int_{B(x,r)} \varphi \, |\nabla u^\pm|^2
+  \int_{B(x,r)} u^\pm \langle \nabla u^\pm,\nabla \varphi \rangle
\right]
&\leq& \kappa r^\alpha J_{x,r}(u)
 \nonumber
\\
&\,& \hskip -4cm
+ \lambda^2 \left[
\int_{B(x,r)} \varphi^2 \, |\nabla u^\pm|^2
+  (u^\pm)^2 |\nabla \varphi|^2 
+ 2\varphi u^\pm \langle \nabla u^\pm,\nabla \varphi \rangle
\right],
\end{eqnarray}
First observe that by our choice of $\varphi$,
\begin{eqnarray}
\label{eqn6.60}
\fint_{t(r)}^{r}
\Big(
\int_{B(x,s)} |\nabla u^\pm(y)|^2 dy \Big) \, ds
&=&  \int_{B(x,r)} |\nabla u^\pm(y)|^2 
\Big( \frac{1}{r-t} \int_{ t}^r \chi_{|y-x| \leq s}\, ds \Big) dy
 \nonumber
 \\
&=& \int_{B(x,r)} \varphi(y) \, |\nabla u^\pm(y)|^2 dy.
\end{eqnarray}
Next notice that $\nabla \varphi = 0$, except on
$B(x,r)\sm B(x,t)$ where its only component
is $\frac{\partial \varphi}{\partial n}
= - \frac{1}{r-t}$. Then
\begin{eqnarray}
\label{eqn6.61}
\fint_{t(r)}^{r}
\Big(\int_{\partial B(x,s)} 
u^\pm \, \frac{\partial u^\pm}{\partial n} \Big) \, ds
&=& \frac{1}{r-t} \int_{t}^r \Big(\int_{\partial B(x,s)} 
u^\pm \, \frac{\partial u^\pm}{\partial n}\Big) ds
 \nonumber
 \\
&=&
\frac{1}{r-t} \int_{B(x,r) \sm B(x,t)} 
u^\pm \, \frac{\partial u^\pm}{\partial n}
\\
&=& - \int_{B(x,r) \sm B(x,t)} 
u^\pm \, \langle \nabla u^\pm,\nabla \varphi \rangle.
 \nonumber
\end{eqnarray}
Let $A$ denote the quantity that we need to estimate 
for (\ref{eqn6.56}); that is, set 
\begin{equation}
\label{eqn6.62}
A = \fint_{t(r)}^{r}
\Big(
\int_{B(x,s)} |\nabla u^\pm(y)|^2 dy \Big) \, ds
-\fint_{t(r)}^{r}
\Big(\int_{\partial B(x,s)} 
u^\pm \, \frac{\partial u^\pm}{\partial n}
\Big) \, ds
\end{equation}
Then by (\ref{eqn6.60}) and (\ref{eqn6.61}), 
$A$ is equal to the content of the first brackets
in (\ref{eqn6.59}), and so (\ref{eqn6.59}) says that
\begin{eqnarray}
\label{eqn6.63}
-2\lambda A
&\leq& \kappa r^\alpha J_{x,r}(u)
+ \lambda^2 \left[
\int_{B(x,r)} \varphi^2 \, |\nabla u^\pm|^2
+  (u^\pm)^2 |\nabla \varphi|^2 
+ 2\varphi u^\pm \langle \nabla u^\pm,\nabla \varphi \rangle
\right]
 \nonumber
 \\
 &\leq& \kappa r^\alpha J_{x,r}(u)
+ 2  \lambda^2 \left[
\int_{B(x,r)} \varphi^2 \, |\nabla u|^2
+  u^2 |\nabla \varphi|^2  \right]
\\
 &\leq& \kappa r^\alpha J_{x,r}(u)
+ 2  \lambda^2 \left[
\int_{B(x,r)}  |\nabla u|^2
+  (r-t)^{-2} \int_{B(x,r)\sm B(x,t)} u^2  \right]
\nonumber
\end{eqnarray}
by Cauchy-Schwarz, straightforward estimates on $\varphi$,
and because $u^+\le|u|$ and $|\nabla u^\pm| \leq |\nabla u|$.
We now take $\lambda = r^{\alpha/2}$, and then
$\lambda = - r^{\alpha/2}$ (both authorized because
$r<1$), and (\ref{eqn6.63}) yields
\begin{eqnarray}
\label{eqn6.64}
 |A|
&\leq& \frac{\kappa r^\alpha J_{x,r}(u)}{2r^{\alpha/2}}
+ r^{\alpha/2} \left[
\int_{B(x,r)}  |\nabla u|^2
+  (r-t)^{-2} \int_{B(x,r)\sm B(x,t)} u^2  \right]
\nonumber
\\
&\leq&
C r^{\alpha/2} \left[
r^n + \int_{B(x,r)}  |\nabla u|^2
+  (r-t)^{-2} \int_{B(x,r)\sm B(x,t)} u^2  \right]
\end{eqnarray}
because $J_{x,r}(u) \leq \int_{B(x,r)}  |\nabla u|^2
+ C r^n$ by definition.

We can still use (\ref{eqn6.35}) and (\ref{eqn6.38})
(our assumptions are the same as in the previous lemmas);
the second one yields
\begin{equation}
\label{eqn6.65}
\int_{B(x,r)}  |\nabla u^+|^2
\leq C r^n \big(\o(x,r_{0}) +  \log(r_{0}/r)\big)^2
\end{equation}
and by the first one,
\begin{equation}
\label{eqn6.66}
|u(y)| \leq C |y-x| \Big( \o(x,3r_{0}/2) + \log\frac{2r_{0}}{|y-x|}\Big)
\ \hbox{ for } y\in B(x,r).
\end{equation}
We apply this for $y\in B(x,r)\sm B(x,t)$
(and thus $|y-x| \geq r/2$) and get that 
\begin{eqnarray}
\label{eqn6.67}
(r-t)^{-2} \int_{B(x,r)\sm B(x,t)} u^2
&\leq& C (r-t)^{-2} |B(x,r)\sm B(x,t)| \, r^2
\Big( \o(x,3r_{0}/2) + \log\frac{2r_{0}}{r}\Big)^2
\nonumber
\\
&\leq& C (r-t)^{-2} r^{n-1}(r-t) r^2
\Big( \o(x,3r_{0}/2) + \log\frac{2r_{0}}{r}\Big)^2
\\
&=&
C r^n r^{-\alpha/4}
\Big( \o(x,3r_{0}/2) + \log\frac{2r_{0}}{r}\Big)^2
\nonumber
\end{eqnarray}
(recall that $r-t = r^{1+\alpha/4} /10$). 
Combining (\ref{eqn6.64}), (\ref{eqn6.65}), and (\ref{eqn6.67})
we obtain
\begin{eqnarray}
\label{eqn6.68}
 |A| \leq C r^n r^{\alpha/4}   
\left[ \Big( \o(x,3r_{0}/2) + \log\frac{2r_{0}}{r}\Big)^2 +1\right],
\end{eqnarray}
which implies (\ref{eqn6.56}) ; 
Lemma \ref{lem6.4} follows.
\qed
\end{proof}

\section{Almost monotonicity 2: we put things together}

In this section we use the inequalities proved in the
last section to complete the proof of the near monotonicity
result stated in Theorem \ref{thm6.1}. We first compute the derivatives of $A_\pm$.

It follows from definition (\ref{eqn6.1}) that
$A_{\pm}$ are absolutely continuous and differentiable almost everywhere, with
\begin{equation}
\label{eqn7.1}
A'_{\pm}(r) = r^{2-n} \int_{\partial B(x,r)}  |\nabla u^\pm|^2.
\end{equation}
Then 
$\Phi(r) = r^{-4} A_{+}(r) A_{-}(r)$ is also differentiable
almost everywhere, with
\begin{equation}
\label{eqn7.2}
\Phi'(r) = -4 r^{-5} A_{+}(r) A_{-}(r) +  r^{-4} A'_{+}(r) A_{-}(r) 
+ r^{-4} A_{+}(r) A'_{-}(r).
\end{equation}
Moreover, a fairly simple manipulation of multiple
integrals shows that $\Phi$ is the integral of
its derivative.

We shall denote by ${\mathbb{S}}$ or ${\mathbb{S}}^{n-1}$ the unit sphere
of $\R^n$. For $0 < s < r$, we shall use Poincar\'e-type estimates
on the domains $\Gamma^\pm(s)$ of ${\mathbb{S}}$
defined by
\begin{equation}
\label{eqn7.3}
\Gamma^\pm(s) = \big\{ \theta \in {\mathbb{S}}^{n-1}
\, ; \, \pm u(s\theta+x) > 0 \big\},
\end{equation}
which are open because $u$ is continuous. Define $\alpha_{\pm}(s)$ by
\begin{equation}
\label{eqn7.4}
\alpha^\pm(s) = \sup\Big\{ 
\frac{\int_{\Gamma^\pm(s)} |v|^2}
{ \int_{\Gamma^\pm(s)} |\nabla_{\theta} v|^2}
\, ; \, v\in W^{1,2}_{0}(\Gamma^\pm(s)),
v \neq 0 \Big\} \in [0,+\infty],
\end{equation}
where $W^{1,2}_{0}(\Gamma^\pm(s))$ is the
closure, in $W^{1,2}(\Gamma^\pm(s))$, of the
set of smooth functions compactly supported 
on $\Gamma^\pm(s)$, and $\nabla_{\theta} v$
is our notation for the gradient on the sphere.
Note that $1/\alpha^\pm(s)$ corresponds to the first eigenvalue of the Laplacian on $\Gamma^\pm(s)$.
Moreover $\alpha^\pm(s) > 0$ as soon as 
$\Gamma^\pm(s) \neq \emptyset$ (because we can
easily find nontrivial smooth functions $v$ with compact 
support in  $\Gamma^\pm(s)$), but it is reasonable
to take $\alpha^\pm(s) = 0$ when  
$\Gamma^\pm(s) = \emptyset$, because 
$\alpha^\pm(s)$ is a nondecreasing function
of the domain. Finally, $\alpha^\pm(s) < +\infty$ 
if the complement of $\Gamma^\pm(s)$ contains a ball,
because then there is a Poincar\'e inequality for compactly 
supported functions in the complement of that ball.

We work under the following assumptions for the next computations.
Let $u$ be an almost minimizer for $J$. Fix $x\in \O$
and radii $0 < r < r_{0}$, with $B(x,2r_{0} )\i  \O$
and $r < \min(1/2,r_{0})$. We assume that $u(x) = 0$, and to simplify the notation we take $x=0$.
Notice that these assumptions correspond to those needed in the hypothesis of 
the lemmas in the previous section.
As in Lemma \ref{lem6.4} set 
\begin{equation}
\label{eqn7.5}
t = t(r) = (1-\frac{r^{\alpha/4}}{10})\, r.
\end{equation}
Our next long term goal
is to estimate $\Phi(r)-\Phi(t)$ (see (\ref{eqn7.58})).

Choose $s_{0} \in (t,r)$, such that $u^\pm(s_{0} \cdot) \in
W^{1,2}_{0}(\Gamma^\pm(s_{0}))$. We want to avoid
problems coming from the variations of the
$\Gamma^\pm(s)$ as a function of $s$, so we shall try to reduce to the
single $\Gamma^\pm(s_{0})$. Observe that 
for $\theta \in \SS^{n-1}$ and $s\in [t,r]$,
\begin{eqnarray}
\label{eqn7.6}
|u(s\theta)-u(s_{0}\theta)|
&\leq& C |s-s_{0}| \Big( \o(x,3r_{0}/2) + \log\frac{2r_{0}}{|s-s_{0}|}\Big)
\nonumber
\\
&\leq& C |r-t| \Big( \o(x,3r_{0}/2) + \log\frac{2r_{0}}{|r-t|}\Big)
=: a(r) 
\end{eqnarray}
by (\ref{eqn6.34}) (or directly by (\ref{eqn2.20})), and using the fact that for $0<\tau<1/e$ the function 
$-\tau\log \tau$ is increasing.The
last identity gives the definition of $a(r)$. Set
\begin{equation}
\label{eqn7.7}
w_{s}(\theta) = \big(u(s\theta)-2a(r)\big)_{+}.
\end{equation}
We claim that $w_s\in W^{1,2}_{0}(\Gamma^+(s_{0}))$
for almost every $s \in [t,r]$. First, the fact that 
$w_s\in W^{1,2}(\SS)$ for almost every $s \in [t,r]$ is classical 
(locally, after a change of variables, it comes from the fact that the restriction
to almost every hyperplane lies in $W^{1,2}$). Next,
$u$ is continuous and $u(s_{0}\theta) \leq 0$ on
$\SS \sm \Gamma^+(s_{0})$; then by (\ref{eqn7.6}) $u(s\theta)\le a(r)$ and 
$w_{s}(\theta) =0 $ on that set. So $w_{s}$
is continuous and compactly supported in $\Gamma^+(s_{0})$.
It is easily 
approximated by smooth compactly supported functions,
and our claim follows. The definition (\ref{eqn7.4})
yields for almost every $s\in [t,r]$ that
\begin{equation}
\label{eqn7.8}
\int_{\SS}  |w_{s}|^2 =
\int_{\Gamma^+(s_{0})} 
|w_{s}|^2  \leq \alpha^+(s_{0}) 
\int_{\Gamma^+(s_{0})} |\nabla_{\theta} w_{s}|^2.
\end{equation}
This is still true when $\Gamma^+(s_{0})$ is empty
(and we set $\alpha^+(s_{0}) = 0$), because then
$w_{s}=0$.

Return to $u(s\theta)$, and cut
$\SS$ into
$E = \big\{ \theta \in \SS \, ; \,u(s\theta) \leq 2a(r) \big\}$ 
and $F = \big\{ \theta \in \SS \, ; \, u(s\theta) > 2a(r) \big\}$; 
then
\begin{eqnarray}
\label{eqn7.9}
\int_{\SS} |u(s\theta)|^2
&\leq& \int_{E} |u(s\theta)|^2 + \int_{F} |u(s\theta)|^2
\leq 4a(r)^2 |E| + \int_{F} |w_{s}(\theta)+2a(r)|^2 
\nonumber 
\\
&\leq&
4a(r)^2 |\SS|+ 4a(r) \int_{F} w_{s}(\theta) 
+\int_{F} |w_{s}(\theta)|^2 
\\
&\leq&
4a(r)^2 |\SS|+ 4a(r) \int_{F} w_{s}(\theta) 
+ \alpha^+(s_{0}) 
\int_{\Gamma^+(s_{0})} |\nabla_{\theta} w_{s}|^2
\nonumber
\end{eqnarray}
Observe that for $\theta \in F$, 
\begin{eqnarray}
0 \leq w_{s}(\theta) &\leq& u(s\theta) 
\leq C |s\theta-x| \big( \o(x,3r_{0}/2) 
+ \log\frac{2r_{0}}{|s\theta-x|}\big)
\nonumber
\\
&\leq& C r \big( \o(x,3r_{0}/2) 
+ \log\frac{2r_{0}}{r}\big)
\label{eqn7.10}
\end{eqnarray}
by (\ref{eqn6.35}) and the fact that $-\tau\log \tau$ is increasing for $\tau\in(0,e^{-1})$ 
(we did not need $n \geq 3$ there). 
Also, (\ref{eqn7.7}) says
that $\nabla_{\theta} w_{s} = \nabla_{\theta} u(s \,\cdot)
= \nabla_{\theta} u^{+}(s \,\cdot)$ on the open set $F$, 
while $\nabla_{\theta} w_{s} = 0$ almost everywhere
on $\SS \sm F$. (By Calder\'on's extension of Rademacher's theorem,
almost everywhere on that set, $\nabla_{\theta} w_{s}(\theta)$ 
comes from a true differential, which has to vanish if
$\theta$ is a point of density of $\SS \sm F$.)
So $\int_{\Gamma^+(s_{0})} |\nabla_{\theta} w_{s}|^2
= \int_{F} |\nabla_{\theta} u^{+}(s \,\cdot)|^2$.  
Thus (\ref{eqn7.9}) and (\ref{eqn7.10}) yield
\begin{equation}
\label{eqn7.11}
\int_{\SS} |u(s\theta)|^2  \leq {\cal E}_{1}
+ \alpha^+(s_{0})   \int_{F} |\nabla_{\theta} u^{+}(s \,\cdot)|^2, 
\end{equation}
with
\begin{equation}
\label{eqn7.12}
{\cal E}_{1} =  C a(r)^2 
+ C a(r) r \big( \o(x,3r_{0}/2) + \log\frac{2r_{0}}{r}\big).
\end{equation}
It turns out that if $\alpha^+(s_{0})$ is too small,
(\ref{eqn7.11}) is too good to be used like this
(due to the way we shall write the error terms), so
we choose $\alpha^+ > 0$, with
\begin{equation}
\label{eqn7.13}
\alpha^+ \geq \alpha^+(s_{0}),
\end{equation}
and then of course 
\begin{equation}
\label{eqn7.14}
\int_{\SS} |u(s\theta)|^2  \leq {\cal E}_{1}
+ \alpha^+   \int_{F} |\nabla_{\theta} u^{+}(s \,\cdot)|^2. 
\end{equation}
Return to the computation of $A'_{+}$ in (\ref{eqn7.1});
for almost every $s\in [t,r]$,
\begin{eqnarray}
\label{eqn7.15}
A'_{+}(s) &=& s^{2-n} \int_{\partial B(x,s)}  |\nabla u^+|^2 
= s^{2-n} \int_{\partial B(x,s)} 
\Big(\frac{\p u^+}{\p n}\Big)^2 
+ |\nabla_{\theta} u^{+}|^2,
\end{eqnarray}
where $\frac{\p u^+}{\p n}$ 
is the radial derivative. By (\ref{eqn7.14}),
\begin{eqnarray}
\label{eqn7.16}
s^{2-n} \int_{\partial B(x,s)}  |\nabla_{\theta} u^{+}|^2 
&=& s \int_{\SS}  |(\nabla_{\theta} u^{+})(s\theta)|^2 
= s^{-1} \int_{\SS}  |\nabla_{\theta} u^{+}(s\,\cdot)|^2 
\geq s^{-1} \int_{F}  |\nabla_{\theta} u^{+}(s\,\cdot)|^2 
\nonumber
\\
&\geq& \frac{1}{s  \alpha^+} \,
\Big(\int_{\SS} |u(s\theta)|^2  - {\cal E}_{1}\Big)
= \frac{s^{-n}}{ \alpha^+} \, \int_{\p B(x,s)} |u|^2  
- \frac{{\cal E}_{1}}{s \alpha^+}.
\end{eqnarray}
Now introduce $\beta_{+} \in [0,1]$, to be chosen 
later, and split
\begin{eqnarray}
\label{eqn7.17}
A'_{+}(s) &=&  s^{2-n} \int_{\partial B(x,s)} 
\Big(\frac{\p u^+}{\p n}\Big)^2 
+ |\nabla_{\theta} u^{+}|^2 
\nonumber
\\
&\geq& s^{2-n} \int_{\partial B(x,s)} 
\Big(\frac{\p u^+}{\p n}\Big)^2  
+ \frac{s^{-n}}{ \alpha^+} \, \int_{\p B(x,s)} |u|^2  
- \frac{{\cal E}_{1}}{s \alpha^+}
\nonumber
\\
&=& s^{2-n} \int_{\partial B(x,s)} 
\Big(\frac{\p u^+}{\p n}\Big)^2 
+ \frac{s^{-n} \beta_{+}^2}{ \alpha^+} \, \int_{\p B(x,s)} |u|^2
+ \frac{s^{-n}(1- \beta_{+}^2)}{ \alpha^+} \, \int_{\p B(x,s)} |u|^2
- \frac{{\cal E}_{1}}{s \alpha^+}.
\end{eqnarray}
By Cauchy-Schwarz, since $|u| \geq u^+$, the 
first two terms in (\ref{eqn7.17}) combine as
\begin{eqnarray}
\label{eqn7.18}
s^{2-n} \int_{\partial B(x,s)} 
\Big(\frac{\p u^+}{\p n}\Big)^2 
+ \frac{s^{-n} \beta_{+}^2}{ \alpha^+} \, \int_{\p B(x,s)} |u|^2
&\geq&
\frac{2 s^{1-n} \beta_{+}}{\sqrt{ \alpha^+}} 
\int_{\p B(x,s)} \big| u^+ \, \frac{\p u^+}{\p n}\big| 
\nonumber
\\
&\geq&
\frac{2 s^{1-n} \beta_{+}}{\sqrt{ \alpha^+}} 
\int_{\p B(x,s)} \big( u^+ \, \frac{\p u^+}{\p n}\big).
\end{eqnarray}
Hence
\begin{equation}
\label{eqn7.19}
A'_{+}(s) \geq   \frac{2 s^{1-n} \beta_{+}}{\sqrt{ \alpha^+}} 
\int_{\p B(x,s)} \big( u^+ \, \frac{\p u^+}{\p n}\big)
+ \frac{s^{-n} (1- \beta_{+}^2)}{ \alpha^+} \, \int_{\p B(x,s)} |u^+|^2
- \frac{{\cal E}_{1}}{s \alpha^+}.
\end{equation}
When $n \geq 3$,
pick $\beta_{+}$
so that
\begin{equation}
\label{eqn7.20}
\frac{\beta_{+}}{\sqrt{ \alpha^+}} 
= \frac{1- \beta_{+}^2}{(n-2)\alpha^+};
\end{equation}
this is possible, and $\beta_{+}$ is unique, 
because the difference between the two sides of 
(\ref{eqn7.20}) is a strictly monotone function of $\beta_{+}$
on $[0,1]$, whose endpoint values have different signs.
When $n=2$, take $\beta_{+}=1$. Then for $n\ge 3$ set 
\begin{equation}
\label{eqn7.21}
\gamma_{+} = \frac{2 \beta_{+}}{\sqrt{ \alpha^+}} 
= 2 \frac{1- \beta_{+}^2}{(n-2)\alpha^+}
\end{equation}
and $\gamma_{+} = \frac{2}{\sqrt{ \alpha^+}}$
when $n=2$. Return to (\ref{eqn7.19}), which now can be written as
\begin{equation}
\label{eqn7.22}
A'_{+}(s) \geq  \gamma_{+}  s^{1-n}
\int_{\p B(x,s)} \big( u^+ \, \frac{\p u^+}{\p n}\big)
+ \frac{(n-2)\gamma_{+}}{2} s^{-n}
\int_{\p B(x,s)} |u^+|^2
- \frac{{\cal E}_{1}}{s \alpha^+}.
\end{equation}
When $n\geq 3$, use Lemma \ref{lem6.3} to write for $s\in [t,r]$
\begin{equation}
\label{eqn7.23}
\frac{1}{2} \fint_{\partial B(x,s)} |u^+|^2
=  c_{n} \, A_{+}(s)
- \frac{s^2}{n(n-2)} \fint_{B(x,s)} |\nabla u^+|^2
+{\cal E}_{2},
\end{equation}
with
\begin{equation}
\label{eqn7.24}
|{\cal E}_{2}| \leq  C r^{2+\frac{\alpha}{n+1}} \,
\Big(1+ \fint_{B(x,3r_{0}/2)} |\nabla u|^2 +  \log^2(r_{0}/r) 
+ \log^2(1/r) \Big).
\end{equation}
We shall not need to do this when $n=2$, because the middle term
in (\ref{eqn7.22}) vanishes. Continue with $n \geq 3$ for the moment,
multiply (\ref{eqn7.23}) by $\H^{n-1}(\p B(x,s)) = n s^{n-1} \omega_n$ 
and get that 
\begin{eqnarray}
\label{eqn7.25}
\frac{1}{2} \int_{\partial B(x,s)} |u^+|^2
&=& n s^{n-1} \omega_n c_{n} \, A_{+}(s)
- \frac{s}{n-2} \int_{B(x,s)} |\nabla u^+|^2
+C s^{n-1} {\cal E}_{2}
\nonumber
\\
&=&  \frac{s^{n-1} A_{+}(s)}{n-2}
- \frac{s}{n-2} \int_{B(x,s)} |\nabla u^+|^2
+C s^{n-1} {\cal E}_{2}
\end{eqnarray}
because $c_{n} = \big(n(n-2)\omega_n\big)^{-1}$.
We multiply by $(n-2) \gamma_{+} s^{-n}$,
replace in (\ref{eqn7.22}), and get that
\begin{eqnarray}
\label{eqn7.26}
A'_{+}(s)&\geq& \gamma_{+}  s^{1-n}
\int_{\p B(x,s)} \big( u^+ \, \frac{\p u^+}{\p n}\big)
+\frac{\gamma_{+} A_{+}(s)}{s}
-  s^{1-n} \gamma_{+} \int_{B(x,s)} |\nabla u^+|^2
+ C   \frac{\gamma_{+}{\cal E}_{2}}{s}
- \frac{{\cal E}_{1}}{s \alpha^+}  
\nonumber
\\
&= &
\frac{\gamma_{+} A_{+}(s)}{s} + \gamma_{+}  s^{1-n} Z^+(s)
+ C   \frac{\gamma_{+}{\cal E}_{2}}{s}
- \frac{{\cal E}_{1}}{s \alpha^+},
\end{eqnarray}
where we set
\begin{equation}
\label{eqn7.27}
Z^+(s) = \int_{\p B(x,s)} \big( u^+ \, \frac{\p u^+}{\p n}\big)
- \int_{B(x,s)} |\nabla u^+|^2.
\end{equation}
When $n=2$, we also get this directly from (\ref{eqn7.22})
and (\ref{eqn6.1}), 
and with one less error term.

We know from Lemma \ref{lem6.4} that $Z^+(s)$
is rather small in average, i.e., that
\begin{equation}
\label{eqn7.28}
\left| 
\fint_{t(r)}^{r} Z^+(s) ds \right| \leq {\cal E}_{3}
\end{equation}
with 
\begin{equation}
\label{eqn7.29}
{\cal E}_{3}
= C r^{n+\frac{\alpha}{4}}
\Big(1 + \fint_{B(x,3r_{0}/2)} |\nabla u|^2 
+ \log^2\frac{r_{0}}{r}
\Big),
\end{equation}
so let us treat $Z^+(s)$ as another error term and continue the
computation. 

We now consider $A_{-}$. We keep the same radius
$s_{0}$, so $\alpha^-(s_{0})$ is defined; then we shall
pick some $\alpha^- \geq \alpha^-(s_{0})$ (as in 
(\ref{eqn7.13})), and also choose $\beta_{-}$ and
$\gamma_{-}$, so that
\begin{equation}
\label{eqn7.30}
\gamma_{-} = \frac{2 \beta_{-}}{\sqrt{ \alpha^-}} 
= 2 \frac{1- \beta_{-}^2}{(n-2)\alpha^-},
\end{equation}
(as in (\ref{eqn7.21}), and where we choose $\beta_-=1$ 
and forget the second part when $n=2$). 
Then we proceed as we did for $A_{+}$, and find that
\begin{equation}
\label{eqn7.31}
A'_{-}(s) \geq 
\frac{\gamma_{-} A_{-}(s)}{s}+ \gamma_{-}  s^{1-n} Z^-(s)
+ C   \frac{\gamma_{-}{\cal E}'_{2}}{s}
- \frac{{\cal E}_{1}}{s \alpha^-},
\end{equation}
(as in (\ref{eqn7.26})), where ${\cal E}'_{2}$ also
satisfies (\ref{eqn7.24}) and $Z^-$ also satisfies
(\ref{eqn7.28}).

Write (\ref{eqn7.26}) and (\ref{eqn7.31})
as $A'_{\pm} \geq s^{-1}\gamma_{\pm} A_{\pm}(s) + R_{\pm}$,
and plug this back in (\ref{eqn7.2}). This yields
\begin{eqnarray}
\label{eqn7.32}
s^5 \Phi'(s) &=& -4  A_{+}(r) A_{-}(s) +  s A'_{+}(s) A_{-}(s) 
+ s A_{+}(s) A'_{-}(s)
\nonumber
\\
&\geq& [\gamma_{+}+\gamma_{-}-4] A_{+}(r) A_{-}(s)
+ s R_{+}(s) A_{-}(s) + s A_{+}(s) R_{-}(s).
\end{eqnarray}
The whole point of the computation is that we can
now choose $\alpha^+$ and $\alpha^-$ so that
$\gamma_{+}+\gamma_{-} \geq 4$. To see this,
let us compute the numbers $\beta_{\pm}$ and 
$\gamma_{\pm}$ in terms of $\alpha^{\pm}$.

Let us start with the case when $n \geq 3$.
First, $\beta_{\pm} = \beta(\alpha^{\pm})$, where
$\beta(\alpha)$ is the unique solution in $[0,1]$ of 
$\frac{\beta}{\sqrt{ \alpha}} = \frac{1- \beta^2}{(n-2)\alpha}$
(as in (\ref{eqn7.20}) or (\ref{eqn7.30})).
That equation is the same as 
$(n-2)\sqrt{ \alpha}\,\beta = 1- \beta^2$
or as $\beta^2 + (n-2)\sqrt{ \alpha}\beta -1 = 0$;
the discriminant is 
$\Delta= (n-2)^2 \alpha +4$
and solutions are 
$\beta = -\frac{1}{2}[(n-2)\sqrt{ \alpha} \pm \sqrt\Delta]$.
We keep the only positive solution, hence 
$\beta(\alpha) = -\frac{1}{2}[(n-2)\sqrt{ \alpha} - \sqrt\Delta]$.
Then $\gamma_{\pm} = \gamma(\alpha)$, where
\begin{eqnarray}
\label{eqn7.33}
\gamma(\alpha)&=& \frac{2\beta}{\sqrt{ \alpha}}
= \frac{1}{ \sqrt{ \alpha}} \cdot
\frac{\Delta-((n-2)\sqrt{ \alpha})^2}
{\sqrt\Delta + (n-2)\sqrt{ \alpha}}
= \frac{1}{\sqrt{ \alpha}}\cdot
\frac{4}{\sqrt\Delta +(n-2)\sqrt{ \alpha}}
\nonumber
\\
&=& \frac{1}{\sqrt{ \alpha}}\cdot
\frac{4}{\sqrt{(n-2)^2 \alpha +4} +(n-2)\sqrt{\alpha}}
\end{eqnarray}
The function $\gamma$ is continuous and (strictly) decreasing on
$(0,+\infty)$. It goes from $+\infty$ to $0$. Let
$\alpha_{0}$ be the solution of $\gamma(\alpha_{0}) =4$, 
and pick
\begin{equation}
\label{eqn7.34}
\alpha^+ = \max(\alpha^+(s_{0}),\alpha_{0})
\ \hbox{ and } \  
\alpha^- = \max(\alpha^-(s_{0}),\alpha_{0}).
\end{equation}
Then (\ref{eqn7.13}) and its analogue for $\alpha^-$
hold. Also recall that 
\begin{equation}
\label{eqn7.35}
\gamma(\alpha^+(s_{0})) + \gamma(\alpha^-(s_{0}))
\geq 4.
\end{equation}
This is the same nontrivial result about first eigenvalues
on disjoint domains of the sphere that was already used
in \cite{ACF}.  See (5.7) in \cite{ACF}, which is itself derived
from results in \cite{FH} and \cite{Sp}. 
Now it is easy to see that 
\begin{equation}
\label{eqn7.36}
\gamma_{+}+\gamma_{-} 
= \gamma(\alpha^{+}) + \gamma(\alpha^{-}) \geq 4,
\end{equation}
because either $\alpha^+ = \alpha^+(s_{0})$
and $\alpha^- = \alpha^-(s_{0})$, and (\ref{eqn7.36})
follows from (\ref{eqn7.35}), or else one of the
$\alpha^\pm$ is equel to $\alpha_{0}$, and then 
$\gamma_{\pm} = \gamma(\alpha_{0}) = 4$.

Let us also check (\ref{eqn7.36}) when $n=2$. It is
well know, and easy to check, that when $I$ is an interval
and $v \in W_{0}^{1,2}(I)$, we have that
$\int_{I} |v|^2 \leq (|I|/\pi)^2 \int_{I} |v'|^2$.
Moreover, for $I=[0,l]$ the optimal functions are multiples of 
$\sin(\pi x/l)$. Thus $\alpha^\pm(s_{0}) = (l^{\pm}/\pi)^2$, where
$l^\pm$ is the length of the longest component of $\Gamma^\pm(s_{0})$
(compare with the definition (\ref{eqn7.4})). The analogue of (\ref{eqn7.35})
is then 
$$
\gamma(\alpha^+(s_{0})) + \gamma(\alpha^-(s_{0}))
= \frac{2}{\sqrt{\alpha^+(s_{0})}} + \frac{2}{\sqrt{\alpha^-(s_{0})}}
= \frac{2\pi}{l^+} + \frac{2\pi}{l^-} \geq 4.
$$
We choose $\alpha^\pm = \max(\alpha^\pm(s_{0}),\alpha_{0})$,
where $\alpha_{0} = 1/4$ is again chosen so that
$\gamma(\alpha_{0}) = \frac{2}{\sqrt{\alpha_{0}}} = 4$,
and then we get (\ref{eqn7.36}) as before.

We may now return to the general case. Notice that (\ref{eqn7.32}) yields 
\begin{equation}
\label{eqn7.37}
s^5 \Phi'(s) \geq
 s R_{+}(s) A_{-}(s) + s A_{+}(s) R_{-}(s)
\end{equation}
and (using the fact that $\Phi$ is the integral
of its derivative),
\begin{equation}
\label{eqn7.38}
\Phi(r) - \Phi(t)  \geq
\int_{t}^{r} [ R_{+}(s) A_{-}(s) +  A_{+}(s) R_{-}(s) ] \frac{ds}{s^4}.
\end{equation}

We are now left with the task of giving lower bounds
fo the various terms in the integral. We shall concentrate
on $R_{+}(s) A_{-}(s)$; the other terms would be treated
the same way, by exchanging the roles of $A_{+}$ and $A_{-}$.
Recall that 
\begin{equation}
\label{eqn7.39}
R_{+}(s) =  \gamma_{+}  s^{1-n} Z^+(s)
+ C   \frac{\gamma_{+}{\cal E}_{2}}{s}
- \frac{{\cal E}_{1}}{s \alpha^+}
\end{equation}
(see (\ref{eqn7.26})). We start with
\begin{equation}
\label{eqn7.40}
E_{1} = \int_{t}^{r} \frac{{\cal E}_{1}}{s \alpha^+} 
A_{-}(s) \frac{ds}{s^4} \, .
\end{equation}
Observe that by (\ref{eqn7.1}) and (\ref{eqn2.11})
(used twice), 
\begin{eqnarray}
\label{eqn7.41}
A_{-}(s) &=& \int_{B(x,s)} {\frac{1}{|x-y|^{n-2}}} \, |\nabla u^-|^2
\leq \int_{B(x,r)}{\frac{1}{|x-y|^{n-2}}} \, |\nabla u|^2
\nonumber
\\
&\leq& \sum_{j \geq 0} \int_{B(x,2^{-j}r) \sm B(x,2^{-j-1}r)} 
{\frac{1}{|x-y|^{n-2}}} \, |\nabla u|^2
\nonumber
\\
&\leq&
\sum_{j \geq 0} 2^{(j+1)(n-2)} r^{2-n} \int_{B(x,2^{-j}r)} 
|\nabla u|^2
\leq C \sum_{j \geq 0} 2^{j(n-2)} r^{2-n} (2^{-j}r)^n
\fint_{B(x,2^{-j}r)}  |\nabla u|^2
\\
&=& C r^2 \sum_{j \geq 0} 2^{-2j} \o(x,2^{-j}r)^2
\leq C r^2 \sum_{j \geq 0} 2^{-2j} \big(\o(x,r) + j \big)^2
\nonumber
\\
&\leq& C r^2 \big(\o(x,r) + 1 \big)^2
\leq C r^2 \big(\o(x,3r_{0}/2) + 1 + \log\frac{3r_{0}}{2r} \big)^2;
\nonumber
\end{eqnarray}
then
\begin{equation}
\label{eqn7.42}
E_{1} \leq C \frac{r-t}{r^5 \alpha^+}
\big[ C a(r)^2 
+ C a(r) r \big( \o(x,3r_{0}/2) + \log\frac{2r_{0}}{r}\big)
\big]
r^2 \big(\o(x,3r_{0}/2) + 1 + \log\frac{3r_{0}}{2r} \big)^2
\end{equation}
by (\ref{eqn7.12}). We may drop $\alpha^+$, because 
we made sure that it is never less than the constant 
$\alpha_{0}$. Also
\begin{equation}
\label{eqn7.43}
a(r) =
C |r-t| \Big( \o(x,3r_{0}/2) + \log\frac{2r_{0}}{|r-t|}\Big)
\end{equation}
by (\ref{eqn7.6}). Recall from (\ref{eqn7.5}) that 
$|r-t| = 10^{-1}r^{1+\frac{\alpha}{4}}$; hence, for $r<1$
\begin{equation}
\label{eqn7.44}
|r-t| \log\frac{2r_{0}}{|r-t|}
\leq r^{1+\frac{\alpha}{4}} \log\frac{2r_{0}}{r}
+ r^{1+\frac{\alpha}{4}} \log\frac{r}{|r-t|}
\leq r \log\frac{2r_{0}}{r} + C r.
\end{equation}
Then (\ref{eqn7.42}) yields  
\begin{eqnarray}
\label{eqn7.45}
E_{1} &\leq& C \frac{r-t}{r^2} a(r)
\big( 1 + \o(x,3r_{0}/2) + \log\frac{2r_{0}}{r}\big)^3
\nonumber
\\
&\leq& C \frac{(r-t)^2}{r^2}  
\big(1 + \o(x,3r_{0}/2) + \log\frac{2r_{0}}{r} \big)^4 
\\
&\leq& C r^{\frac{\alpha}{2}}
\big(1 + \o(x,3r_{0}/2) + \log\frac{2r_{0}}{r} \big)^4. 
\nonumber
\end{eqnarray}
Next we deal with 
\begin{equation}
\label{eqn7.46}
E_{2} = \int_{t}^{r} \Big|C \frac{\gamma_{+}{\cal E}_{2}}{s} \Big|
A_{-}(s) \frac{ds}{s^4}.
\end{equation}
Here again we may drop $\gamma_{+}$, since
$\gamma_{+} = \gamma(\alpha^+) \leq \gamma(\alpha_{0})$;
we use (\ref{eqn7.24}) and (\ref{eqn7.41}) and get that
\begin{eqnarray}
\label{eqn7.47}
E_{2} &\leq& C \frac{r-t}{r^5} 
r^2 \big(\o(x,3r_{0}/2) + 1 + \log\frac{3r_{0}}{2r} \big)^2\cdot
\nonumber
\\
&\,& \hskip2cm
r^{2+\frac{\alpha}{n+1}} \, 
\Big(1+ \fint_{B(x,3r_{0}/2)} |\nabla u|^2 +  \log^2(r_{0}/r) 
+ \log^2(1/r) \Big)
\nonumber
\\
&\leq&
C \frac{r-t}{r} \, r^{\frac{\alpha}{n+1}}
\Big(1 + \o(x,3r_{0}/2) + \log\frac{3r_{0}}{2r} 
+ \log\frac{1}{r} \Big)^4
\\
&\leq& C  r^{\frac{\alpha}{4}}\, r^{\frac{\alpha}{n+1}}
\Big(1 + \o(x,3r_{0}/2) + \log\frac{3r_{0}}{2r} 
+ \log\frac{1}{r} \Big)^4
\nonumber
\end{eqnarray}
by (\ref{eqn7.5}).  
Finally set
\begin{equation}
\label{eqn7.48}
E_{3} = \int_{t}^{r} \gamma_{+}  s^{1-n} Z^+(s)
A_{-}(s) \frac{ds}{s^4}.
\end{equation}
Since we could only estimate $Z^+$ in average,
we first estimate
\begin{equation}
\label{eqn7.49}
E_{31} = \int_{t}^{r} \gamma_{+}  r^{1-n} Z^+(s)
A_{-}(r) \frac{ds}{r^4} 
=  \gamma_{+} r^{-3-n} A_{-}(r) \int_{t}^{r} Z^+(s) ds,
\end{equation}
for which we use (\ref{eqn7.28}), (\ref{eqn7.29}),
and (\ref{eqn7.41}) and get
\begin{eqnarray}
\label{eqn7.50}
|E_{31}| &\leq& C r^{-3-n} A_{-}(r) (r-t) {\cal E}_{3}
\leq C r^{-3-n}
 r^2 \Big(\o(x,3r_{0}/2) + 1 + \log\frac{3r_{0}}{2r} \Big)^2
 \nonumber
 \\
 &\,& \hskip5.2cm
(r-t) r^{n+\frac{\alpha}{4}} \Big(1 + \fint_{B(x,3r_{0}/2)} |\nabla u|^2 
+ \log^2\frac{r_{0}}{r} \Big)
\nonumber
\\
&\leq& C \frac{r-t}{r} r^{\frac{\alpha}{4}}
 \Big(1 + \o(x,3r_{0}/2) + \log\frac{3r_{0}}{2r} \Big)^4
\leq C r^{\frac{\alpha}{2}}
 \Big(1 + \o(x,3r_{0}/2) + \log\frac{3r_{0}}{2r} \Big)^4.
 \end{eqnarray}  
The other piece is
\begin{equation}
\label{eqn7.51}
E_{32} = \gamma_{+} \int_{t}^{r}   Z^+(s)
\Big[  s^{-3-n} A_{-}(s) - r^{-3-n} A_{-}(r) \Big] \, ds.
\end{equation}
which we shall only estimate under the assumption that
\begin{equation}
\label{eqn7.52}
\Big|  s^{-3-n} A_{-}(s) - r^{-3-n} A_{-}(r) \Big| 
\leq \frac{K (r-t)}{r^{n+4}} A_{-}(r),
\end{equation}
where the numerical constant $K \geq 1$ is to be 
chosen soon.
Recall from (\ref{eqn7.27}) that
\begin{equation}
\label{eqn7.53}
|Z^+(s)| \leq 
\int_{\p B(x,s)} \big| u^+ \, \frac{\p u^+}{\p n}\big|
+ \int_{B(x,s)} |\nabla u^+|^2.
\end{equation}
By (\ref{eqn6.38}) (or directly (2.11)),
\begin{equation}
\label{eqn7.54}
\int_{B(x,s)} |\nabla u^+|^2 \leq \int_{B(x,r)} |\nabla u^+|^2
\leq C r^n\big(\o(x,r_{0})+\log\frac{r_{0}}{r}\big)^2.
\end{equation}
It is more convenient to integrate the other term:
\begin{eqnarray}
\label{eqn7.55}
\int_{t}^{r} \int_{\p B(x,s)} \big| u^+ \, \frac{\p u^+}{\p n}\big|
&=& \int_{B(x,r)\sm B(x,t)}  \big| u^+ \, \frac{\p u^+}{\p n}\big|
 \nonumber
 \\
&\leq&
C r \Big( \o(x,3r_{0}/2) + \log\frac{2r_{0}}{r}\Big)
\int_{B(x,r)\sm B(x,t)} |\nabla u^+|
 \nonumber
 \\
&\leq&
C r \Big( \o(x,3r_{0}/2) + \log\frac{2r_{0}}{r}\Big)
\nonumber
\\
&\,&\hskip2cm
|B(x,r)\sm B(x,t)|^{1/2} 
\Big(\int_{B(x,r)\sm B(x,t)} |\nabla u^+|^2\Big)^{1/2}
\\ 
&\leq&
C r \Big( \o(x,3r_{0}/2) + \log\frac{2r_{0}}{r}\Big)
\big(\frac{r-t}{r}\big)^{1/2} 
r^n \big(\o(x,r_{0})+\log\frac{r_{0}}{r}\big) 
 \nonumber
 \\
 &\leq&
C r^{n+1} \big(\frac{r-t}{r}\big)^{1/2} 
\Big( \o(x,3r_{0}/2) + \log\frac{2r_{0}}{r}\Big)^{2} 
\nonumber
\end{eqnarray}
because by (\ref{eqn6.35}) 
$|u(y)| \leq r \big( \o(x,3r_{0}/2) + \log\frac{2r_{0}}{r}\big)$
in $B(x,r)$, and then by (\ref{eqn7.54}). Notice that this term
gives a bigger contribution than the one in (\ref{eqn7.54})
(after it is integrated on $[t,r]$).
Thus, under our additional assumption (\ref{eqn7.52}),
\begin{eqnarray}
\label{eqn7.56} 
|E_{32}| &\leq& C  \frac{(r-t)}{r^{n+4}} A_{-}(r)
r^{n+1} \big(\frac{r-t}{r}\big)^{1/2} 
( \o(x,3r_{0}/2) + \log\frac{2r_{0}}{r}\Big)^{2} 
\nonumber
 \\
&\leq& Cr^{-2} \big(\frac{r-t}{r}\big)^{3/2} A_{-}(r)
( \o(x,3r_{0}/2) + \log\frac{2r_{0}}{r}\Big)^{2}
\nonumber
\\
&\leq& C  \big(\frac{r-t}{r}\big)^{3/2}
\big(\o(x,3r_{0}/2) + 1 + \log\frac{3r_{0}}{2r} \big)^{4}
\\
&\leq& C r^{\frac{3\alpha}{8}}
\big(1 + \o(x,3r_{0}/2) +  \log\frac{3r_{0}}{2r} \big)^{4}.
\nonumber
\end{eqnarray}
by (\ref{eqn7.41}). We now sum the pieces from
(\ref{eqn7.45}), (\ref{eqn7.47}), (\ref{eqn7.50}), and (\ref{eqn7.56})
and get that
\begin{equation}
\label{eqn7.57}
\int_{t}^{r} R_{+}(s) A_{-}(s)  \frac{ds}{s^4}
\geq - C r^{\frac{\alpha}{4}}\, r^{\frac{\alpha}{4(n+1)}}
\Big(1 + \o(x,3r_{0}/2) + \log\frac{3r_{0}}{2r} 
+ \log\frac{1}{r} \Big)^4.
\end{equation}
Under the same assumption (\ref{eqn7.52}), but
for $A_{+}$, we get the same estimate for
$\int_{t}^{r} R_{-}(s) A_{+}(s)  \frac{ds}{s^4}$,
and the we sum, compare with (\ref{eqn7.38}),
and get that
\begin{equation}
\label{eqn7.58}
\Phi(r) - \Phi(t)  \geq
- C r^{\frac{\alpha}{4}}\, r^{\frac{\alpha}{4(n+1)}}
\Big(1 + \o(x,3r_{0}/2) + \log\frac{3r_{0}}{2r} 
+ \log\frac{1}{r} \Big)^4.
\end{equation}
This will be good enough for us, but we still need
to study the case when (\ref{eqn7.52}), or its analogue
for $A_{+}$, fails. In this case, we shall prove that
$\Phi(r) \geq \Phi(t)$ by a more direct argument, just
because the large increase in $A_{-}$ or $A^+$ is enough.

Suppose for instance that (\ref{eqn7.52}) fails,
i.e., that
\begin{equation}
\label{eqn7.59}
\Big|  s^{-3-n} A_{-}(s) - r^{-3-n} A_{-}(r) \Big| 
> \frac{K (r-t)}{r^{n+4}} A_{-}(r)
\end{equation}
for some $s\in [t,r]$.
Observe that since $A_{-}$ is nondecreasing
by (\ref{eqn6.1}))  and 
\begin{eqnarray}
\label{eqn7.60}
s^{-3-n} A_{-}(s) - r^{-3-n} A_{-}(r)  
&\leq& [s^{-3-n}  - r^{-3-n}] A_{-}(r) 
\leq (3+n) s^{-4-n} (r-s) A_{-}(r)
\nonumber
\\
&\leq& (3+n) t^{-4-n} (r-t) A_{-}(r)
\leq \frac{K (r-t)}{r^{n+4}} A_{-}(r)
\end{eqnarray}
by the fundamental theorem of calculus and 
if we choose $r < 1$ and $K \geq (3+n)(10/9)^{n+4}$ (see the definition
(\ref{eqn7.5})). So (\ref{eqn7.59}) actually says that
\begin{equation}
\label{eqn7.61}
r^{-3-n} A_{-}(r) -  s^{-3-n} A_{-}(s) 
>  \frac{K (r-t)}{r^{n+4}} A_{-}(r)
\end{equation}
because the other sign is impossible. Then
\begin{equation}
\label{eqn7.62}
A_{-}(s) \leq  \frac{s^{3+n}}{r^{3+n}} \, 
\Big[1 - \frac{K (r-t)}{r} \Big] A_{-}(r)
\leq \Big[1 - \frac{K (r-t)}{r} \Big] A_{-}(r) 
\end{equation}
and
\begin{eqnarray}
\label{eqn7.63}
\Phi(t) &=& t^{-4} A_{+}(t)A_{-}(t)
\leq  t^{-4} A_{+}(r)A_{-}(s)
\leq  t^{-4} A_{+}(r)A_{-}(r) \Big[1 - \frac{K (r-t)}{r} \Big]
\nonumber
\\
&=& \Phi(r) \frac{r^4}{t^4} \Big[1 - \frac{K (r-t)}{r} \Big]
=  \Phi(r)  \Big[1-\frac{r-t}{r} \Big]^{-4} \Big[1 - \frac{K (r-t)}{r} \Big]
\leq \Phi(r)
\end{eqnarray}
if $K$ is large enough. The case when the analogue
of  (\ref{eqn7.52}) for $A_{+}$ fails is treated the
same way. Thus (\ref{eqn7.58}) is established  
in full generality.

Next we want to use (\ref{eqn7.58}) to estimate
$\Phi(s)$ for all $s < r$, and not just $s=t(r)$.
We start with estimates along the slowly decreasing
sequence $\{ r_{j} \}$, where $r_{0}=r$ and 
$r_{j+1} = t(r_{j}) = (1-\frac{r_{j}^{\alpha/4}}{10})\, r_{j}$
(see (\ref{eqn7.5})). 

First observe that since $\{ r_{j} \}$ is decreasing
and nonnegative, it has a limit $\ell \geq 0$. In addition,
since $t(\ell) = \ell$, we get that $l=0$. For the moment,
fix $r$ and set
\begin{equation}
\label{eqn7.64}
\Psi(s) = \Phi(s) + C(r) s^{\delta },
\end{equation}
where the exponent $\delta$ is chosen so that 
$0 < \delta < \frac{\alpha}{4(n+1)}$, and
$C(r)$ is to be chosen soon.
We want to show that for $j \geq 0$,
\begin{equation}
\label{eqn7.65}
\Psi(r_{j+1}) \leq \Psi(r_{j})  
\end{equation}
By (\ref{eqn7.58})
\begin{equation}
\label{eqn7.66}
\Phi(r_{j}) - \Phi(r_{j+1})  \geq
- C_{0} r_{j}^{\frac{\alpha}{4}}\, r_{j}^{\frac{\alpha}{4(n+1)}}
\Big(1 + \o(x,3r_{0}/2) + \log\frac{3r_{0}}{2r_{j}} 
+ \log\frac{1}{r_{j}} \Big)^4,
\end{equation}
where we call the constant $C_{0}$ to avoid confusion.
We just need to check that 
\begin{equation}
\label{eqn7.67}
C_{0} r_{j}^{\frac{\alpha}{4}}\, r_{j}^{\frac{\alpha}{4(n+1)}}
\Big(1 + \o(x,3r_{0}/2) + \log\frac{3r_{0}}{2r_{j}} 
+ \log\frac{1}{r_{j}} \Big)^4
\leq C(r) [r_{j}^{\delta } - r_{j+1}^{\delta }].
\end{equation}
Since $r_{j+1} = (1-\frac{r_{j}^{\alpha/4}}{10})\, r_{j}$,
\begin{equation}
\label{eqn7.68}
r_{j}^{\delta } - r_{j+1}^{\delta }
= r_{j}^{\delta } 
- (1-\frac{r_{j}^{\alpha/4}}{10})^{\delta } r_{j}^{\delta }
= r_{j}^{\delta } \big[1 - (1-\frac{r_{j}^{\alpha/4}}{10})^{\delta }\big].
\end{equation}
Set $u = \frac{r_{j}^{\alpha/4}}{10}$; then by Taylor's
formula (or just concavity) $(1-u)^\delta \leq 1 - \delta u$ 
(the second derivative is negative), so
$1- (1-u)^\delta \geq \delta u$ and 
\begin{equation}
\label{eqn7.69}
r_{j}^{\delta } - r_{j+1}^{\delta }
\geq \delta r_{j}^{\delta} \,\frac{r_{j}^{\alpha/4}}{10}.
\end{equation}
Thus, returning to (\ref{eqn7.67}), it is enough to prove that
\begin{equation}
\label{eqn7.70}
C_{0} r_{j}^{\frac{\alpha}{4}}\, r_{j}^{\frac{\alpha}{4(n+1)}}
\Big(1 + \o(x,3r_{0}/2) + \log\frac{3r_{0}}{2r_{j}} 
+ \log\frac{1}{r_{j}} \Big)^4
\leq C(r) \delta r_{j}^{\delta} \,\frac{r_{j}^{\alpha/4}}{10}.
\end{equation}

Let $R = 1 + \o(x,3r_{0}/2) 
+ \log\frac{3r_{0}}{2r} + \log\frac{1}{r}$; then
\begin{eqnarray}
\label{eqn7.71}
\Big(1 + \o(x,3r_{0}/2) + \log\frac{3r_{0}}{2r_{j}} 
+ \log\frac{1}{r_{j}} \Big)^4
&=& \Big(R + 2\log\frac{r}{r_{j}} \Big)^4
\leq R^4 \Big(1+2\log\frac{r}{r_{j}} \Big)^4
\nonumber
\\
&\leq& R^4 \Big(1+2\log\frac{1}{r_{j}} \Big)^4
\leq C(\varepsilon) R^4 r_{j}^{-\varepsilon}
\end{eqnarray}
because $R \geq 1$, $r\le 1$, and where 
we choose $\varepsilon = \frac{\alpha}{4(n+1)}-\delta >0$
so that the powers will match. Then we choose $C(r)$ such that
\begin{equation}
\label{eqn7.72}
\delta C(r) = 10 C_{0} C(\varepsilon) R^4
= 10 C_{0} C(\varepsilon) \big( 1 + \o(x,3r_{0}/2) 
+ \log\frac{3r_{0}}{2r} + \log\frac{1}{r}\big)^4,
\end{equation}
and we get (\ref{eqn7.70}), (\ref{eqn7.67}), and
(\ref{eqn7.65}). In particular, 
\begin{equation}
\label{eqn7.73}
\Phi(r_{j}) \leq \Psi(r_{j}) \leq \Psi(r) = \Phi(r) + C(r) r_{j}^\delta 
\ \hbox{ for } j \geq 0.
\end{equation}
We now turn to any $s\in (0,r)$. Let $j$ be such that
$r_{j+1} \leq s \leq r_{j}$; then
\begin{equation}
\label{eqn7.74}
\Phi(s) = s^{-4} A_{+}(s)A_{-}(s)
\leq s^{-4} A_{+}(r_{j})A_{-}(r_{j})
= \frac{r_{j}^4}{s^4}\Phi(r_{j}).
\end{equation}
Recall that $r_{j+1} = (1-\frac{r_{j}^{\alpha/4}}{10})\, r_{j}$,
and set $u = \frac{r_{j}-s}{r_{j}}$. Obviously
$u \leq \frac{r_{j}-r_{j+1}}{r_{j}} = \frac{r_{j}^{\alpha/4}}{10}$.
Also, $1-u = \frac{s}{r_{j}}$ and hence
\begin{equation}
\label{eqn7.75}
\frac{r_{j}^4}{s^4} = (1-u)^{-4} \leq 1+10u
\leq 1 + r_{j}^{\alpha/4}.
\end{equation}
Next, we know from the proof of (\ref{eqn7.41})
that $A_{\pm}(r_{j}) \leq C r_{j}^2 (\o(x,r_{j}) + 1)^2$,
so
\begin{eqnarray}
\label{eqn7.76}
 \Phi(r_{j}) &=& r_{j}^{-4} A_{+}(r_{j}) A_{-}(r_{j})
 \leq C (\o(x,r_{j}) + 1)^4
 \leq C \big(1 + \o(x,3r_{0}/2) + \log\frac{3r_{0}}{2r_{j}}\big)^4
 \nonumber
 \\
 &\leq& C \big(1 + \o(x,3r_{0}/2) + \log\frac{3r_{0}}{2r}\big)^4
 \big(1+ \log\frac{r}{r_{j}}\big)^4
\end{eqnarray}
by (\ref{eqn2.11}). Thus 
\begin{eqnarray}
\label{eqn7.77}
r_{j}^{\alpha/4} \Phi(r_{j})
&\leq& C \big(1 + \o(x,3r_{0}/2) + \log\frac{3r_{0}}{2r}\big)^4
r^{\alpha/4} \Big[\big(\frac{r_{j}}{r}\big)^{\alpha/4}
\big(1+ \log\frac{r}{r_{j}}\big)^4 \Big]
 \nonumber
 \\
&\leq& C \big(1 + \o(x,3r_{0}/2) + \log\frac{3r_{0}}{2r}\big)^4
r^{\alpha/4}.
\end{eqnarray}
Finally by (\ref{eqn7.74}), (\ref{eqn7.75}), (\ref{eqn7.77}),
and (\ref{eqn7.73}),
\begin{eqnarray}
\label{eqn7.78}
\Phi(s) &\leq&  \frac{r_{j}^4}{s^4}\Phi(r_{j})
\leq \Phi(r_{j}) + r_{j}^{\alpha/4} \Phi(r_{j})
\nonumber
 \\
&\leq& \Phi(r_{j}) + C \big(1 + \o(x,3r_{0}/2) + \log\frac{3r_{0}}{2r}\big)^4
r^{\alpha/4}
\\
&\leq& \Phi(r) + C'(r) r^\delta 
\nonumber
\end{eqnarray}
with a formula for $C'(r)$ that is
just like (\ref{eqn7.72}). 
Hence
\begin{eqnarray}
\label{eqn7.79}
C'(r) &=& C \big( 1 + \o(x,3r_{0}/2) 
+ \log\frac{3r_{0}}{2r} + \log\frac{1}{r}\big)^4
\nonumber
\\
&\leq& C \big( 1 + \o(x,3r_{0}/2)
+ \big(\log r_{0}\big)_{+} \big)^4
\big( 1 +\log\frac{1}{r} \big)^4.
\end{eqnarray}
Now this works with any $\delta < \alpha/4(n+1)$.
Thus for $0<\delta<\delta'< \alpha/4(n+1))$ we have
\begin{eqnarray}
\label{eqn7.79A}
\Phi(s)\le \Phi(r) + \wt C(x, r_0)\left(1+\log\frac{1}{r}\right)^4 r^{\delta'}
\end{eqnarray}
with 
\begin{eqnarray}
\label{eqn7.79B}
\wt C(x, r_0)= C \big( 1 + \o(x,3r_{0}/2)
+ \big(\log r_{0}\big)_{+} \big)^4.
\end{eqnarray}
Hence modulo making $\tilde C(x, r_0)$ a little larger depending on $\delta$ we obtain
\begin{eqnarray}
\label{eqn7.79C}
\Phi(s)\le \Phi(r) + C(x, r_0) r^{\delta}
\end{eqnarray}
with $C(x,r_0)$ as in (\ref{eqn6.3}). This proves (\ref{eqn6.4}) and completes the 
proof of Theorem \ref{thm6.1}.

\section{Almost minimizers for $J$ are locally Lipschitz}
We are now ready to prove the following.

\begin{theorem}\label{thm8.1}
Let $u$ be an almost minimizer for $J$ in $\O$. Then $u$ is locally 
Lipschitz in $\O$.
\end{theorem}

We want to follow the same general scheme as
for Theorem \ref{thm5.1} (in the case of $J^+$), 
and here is the analogue of Lemma \ref{lem5.1}.
Recall that we set $b(x,r) = \fint_{\p B(x,r)} u$
in (\ref{eqn4.1}).

\begin{lemma}\label{lem8.1}
Let $u$ be an almost minimizer for $J$ in $\O$,
and let $B_{0} = B(x_{0},2r_{0}) \i \O$ be given.
Then there exist $\gamma>0$, $K_1>1$, 
and $r_1>0$ such that if $x\in B(x_{0},r_{0})$ 
and $0 < r \leq r_{1}$ are such 
\begin{equation}
\label{eqn8.1}
u(y) = 0 \hbox{ for some } y\in B(x, 2r/3),
\end{equation}
\begin{equation}
\label{eqn8.2}
|b(x, r)| \leq \gamma r \left(1+\o(x,r)\right),
\end{equation}
and
\begin{equation}
\label{eqn8.3}
\o(x, r) \geq K_1,
\end{equation}
then
\begin{equation}
\label{eqn8.4}
\o(x,r/3) \leq \o(x,r)/2. 
\end{equation}
\end{lemma}

So the main difference with Lemma \ref{lem5.1}
is that we now add the constraint (\ref{eqn8.1}), but we 
shall see later that things are easier if (\ref{eqn8.1}) does not hold.
We also  used $B_{0}$ to  localize a little more and get
a uniform control  on the function $\Phi$ as defined in (\ref{eqn6.2}).

\begin{proof}
Let $x$ and $r$ be as in the statement, and
let $\Phi$ be as in (\ref{eqn6.2}). Apply 
Theorem \ref{thm6.1}, but with everything centered
at some $\wt x \in B(x,r)$ such that $u(\wt x)  =  0$,
with $\wt s = 2r$, $\wt r = r_{0}/4$, and 
$\wt r_{0} = r_{0}/2$. Thus $B(\wt x,2\wt r_{0})
\i B_{0}$ and $\wt s < \wt r$ if $r_{1}$ is small enough.
We get that
\begin{equation}
\label{eqn8.5}
\Phi(2r) = \Phi(\wt s)
\leq \Phi(\wt r) + C(\wt x,\wt r_{0}) \wt r^\delta
\leq \Phi(\wt r) + C(\wt x,\wt r_{0}) r_{0}^\delta
\end{equation}
where
\begin{eqnarray}
\label{eqn8.6}
C(\wt x,\wt r_{0}) 
&=& C + C \Big( \fint_{B(\wt x,3 \wt r_{0}/2)} |\nabla u|^2 \Big)^2
+ C(\log(\wt r_{0})_+)^4 
\nonumber
\\
&\leq& C + C \Big( \fint_{B_{0}} |\nabla u|^2 \Big)^2
+ C(\log(r_{0})_+)^4.
\end{eqnarray}
Since 
\begin{equation}
\label{eqn8.7}
\Phi(\wt r) = \wt r^{-4} A_{+}(\wt r) A_{-}(\wt r)
\leq C (\o(\wt x,\wt r) + 1)^4
\leq C + C\Big( \fint_{B(x,3r_{0}/2)} |\nabla u|^2 \Big)^2 
\end{equation}
as in (\ref{eqn7.76}) and (\ref{eqn7.41}), 
we see that
\begin{equation}
\label{eqn8.8}
\Phi(2r) \leq C(B_{0}),
\ \hbox{ with }
 C(B_{0}) = C + C \Big( \fint_{B(x,3r_{0}/2)} |\nabla u|^2 \Big)^2.
\end{equation}
Set
\begin{equation}
\label{eqn8.9}
\o_{\pm}(x,r) = \Big(\fint_{B(x,r)} |\nabla u^\pm|^2\Big)^{1/2};
\end{equation}
then
\begin{equation}
\label{eqn8.10}
\o_{+}(x,r)^2 + \o_{-}(x,r)^2 = \o(x,r)^2 \geq K_{1}^2
\end{equation}
by (\ref{eqn2.6}) and (\ref{eqn8.3}). At the same time,
$B(x,r) \i B(\wt x,2r)$, so by (\ref{eqn6.1}), (\ref{eqn6.2}), and (\ref{eqn8.8})
\begin{eqnarray}
\label{eqn8.11}
\o_{+}(x,r)^2 \o_{-}(x,r)^2 
&\leq& C \o_{+}(\wt x,2r)^2 \o_{-}(\wt x,2r)^2 
\nonumber
\\
&\leq& C r^{-4} A_{+}(2r) A_{+}(2r) 
= C \Phi(2r) \leq C C(B_{0}).
\end{eqnarray}

Let us assume that $\o_{+}(x,r) \leq \o_{-}(x,r)$; the
other case would be treated the same way. Then by
(\ref{eqn8.11})
\begin{equation}
\label{eqn8.12}
\o_{+}(x,r)^2 \leq \sqrt{C C(B_{0})}
\end{equation}
and by (\ref{eqn8.10})
\begin{equation}
\label{eqn8.13}
\o_{-}(x,r)^2 \geq  K_{1}^2 - \o_{+}(x,r)^2 
\geq K_{1}^2 - \sqrt{C C(B_{0})} \geq K_{1}^2/2
\end{equation}
if we choose $K_{1}^2 \geq 2\sqrt{C C(B_{0})}$.
By (\ref{eqn8.11}) and (\ref{eqn8.13}) we have
\begin{equation}
\label{eqn8.14}
\o_{+}(x,r)^2 \leq 2 K_{1}^{-2} C C(B_{0}),
\end{equation}
which is still very small for $K_1$ large.

Our next task is to estimate  
$\int_{\p B(x,r)} u^+$. By (\ref{eqn8.1}), we can find 
$z\in B(x,2r/3)$ such that $u(z) = 0$. 
Let $\eta < 1/6$ be small, to be chosen soon,
and let us apply (\ref{eqn2.20}) with $B(z,\eta r)$ 
replacing $B(x_{0},r_{0})$. We get that for $y\in B(z,\eta r/8)$,
\begin{eqnarray}
\label{eqn8.15} 
u^+(y) &\leq& |u(y)| = |u(y)-u(z)| 
\leq C |y-z| \left(\o(z,2\eta r) + \log\frac{\eta r}{|y-z|}\right)
\nonumber
\\
&\leq& C \eta r \left(1+ \o(z,2\eta r) \right)
\leq C \eta r \left(1+ \o(z,r/3) + \log(1/\eta) \right)
\\
&\leq& C \eta r \left(1+ \o(x, r) + \log(1/\eta) \right)
\nonumber
\end{eqnarray}
by (\ref{eqn2.11}) and because $B(z,r/3) \i  B(x,r)$.
Next applying the fundamental theorem of calculus along 
rays from $z$ and between $\partial B(z,\eta r)$ and $\partial B(x,r)$ then averaging we have
\begin{equation}
\label{eqn8.16}
\fint_{\p B(x,r)} u^+ - \fint_{\p B(z,\eta r)} u^+ 
\leq C(\eta)  r \fint_{B(x,r)} |\nabla u^+|,
\end{equation}
where of course $C(\eta)$ depends on $\eta$. 
In turn
\begin{equation}
\label{eqn8.17}
\fint_{B(x,r)} |\nabla u^+|
\leq C \Big(\fint_{B(x,r)} |\nabla u^+|^2\Big)^{1/2}
= C \omega_{+}(x,r) \leq C \big[K_{1}^{-2} C(B_{0})\big]^{1/2}
\end{equation}
by (\ref{eqn8.14}). So by (\ref{eqn8.15}), (\ref{eqn8.16}), and (\ref{eqn8.17})
\begin{eqnarray}
\label{eqn8.18}
\fint_{\p B(x,r)} u^+
&\leq& \sup_{\p B(z,\eta r)} u^+ 
+\Big| \fint_{\p B(x,r)} u^+ - \fint_{\p B(z,\eta r)} u^+ \Big|
\nonumber
\\
&\leq& C \eta r \left(1+ \o(x, r) + \log(1/\eta) \right)
+ C(\eta)  r \big[K_{1}^{-2} C(B_{0})\big]^{1/2}.
\end{eqnarray}
Since $u = u^+ - u^-$,
\begin{equation}
\label{eqn8.19}
\fint_{\p B(x,r)} u^- = \fint_{\p B(x,r)} u^+ - \fint_{\p B(x,r)} u
\end{equation}
and hence by (\ref{eqn8.2})
\begin{equation}
\label{eqn8.20}
\fint_{\p B(x,r)} |u| = \fint_{\p B(x,r)} u^+ + \fint_{\p B(x,r)} u^-
\leq 2\fint_{\p B(x,r)} u^+ - \fint_{\p B(x,r)} u
\leq \xi r,
\end{equation}
where
\begin{equation}
\label{eqn8.21}
\xi = \gamma (1+\omega(x,r))
+ C \eta \left(1+ \o(x, r) + \log(1/\eta) \right)
+ CC(\eta) \big[K_{1}^{-2} C(B_{0})\big]^{1/2}
\end{equation}
is still small compared to $\omega(x,r)$.

We again want to compare $u$ to $u^\ast_{r}$,
the harmonic energy minimizing extension defined 
near (\ref{eqn2.2}). Recall from Remark 3.1 
that $u^\ast_{r}$ can also be computed from the values
of $u$ on $\p B(x,r)$ by convolving with the Poisson kernel.
Then using the Poisson kernel we have
\begin{equation}
\label{eqn8.22}
u^\ast_{r}(y) \leq C \fint_{\p B(x,r)} |u| \leq C \xi r \ \hbox{ for } y\in  B(x,3r/4)\ \hbox{ and }
|\nabla u^\ast_r(y)| \leq C \xi
\ \hbox{ for } y\in  B(x,r/2).
\end{equation}
Recall from (\ref{eqn2.5}) that
\begin{equation}
\label{eqn8.23}
\int_{B(x,r)}|\nabla u-\nabla u^\ast_r|^2 
\leq \kappa r^\alpha \int_{B(x,r)}|\nabla u|^2+Cr^n;
\end{equation}
then
\begin{eqnarray}
\label{eqn8.24}
\omega(x,r/3)^2 &=& \fint_{B(x,r/3)}|\nabla u|^2
\leq 2 \fint_{B(x,r/3)}|\nabla u^\ast_r|^2 
+ 2 \fint_{B(x,r/3)}|\nabla u-\nabla u^\ast_r|^2
\nonumber
\\
&\leq & \wt C \xi^2 + C \kappa r^\alpha \fint_{B(x,r)}|\nabla u|^2+C
= \wt C \xi^2 + C \kappa r^\alpha \omega(x,r)^2+C. 
\end{eqnarray} 
If $K_{1}$ is large enough, the last term is 
$C \leq C K_{1}^{-4} \omega(x,r)^2 \leq  \omega(x,r)^2/20$,
by (\ref{eqn8.3}). 
If $r_{1}$ is small enough, 
then $C \kappa r^\alpha \omega(x,r)^2 \leq C \kappa r_{1}^\alpha
\omega(x,r)^2 \leq \omega(x,r)^2/20$. Concerning $\xi$,
if $\gamma$ is small enough, then 
$\wt C \gamma^2 (1+\omega(x,r))^2 \leq \omega^2(x,r)/100$
in (\ref{eqn8.21}); if $\eta$ if small enough, 
$\wt C C^2\eta^2 \left(1+ \o(x, r) + \log(1/\eta) \right)^2 \leq \omega^2(x,r)/100$.
Finally, if $K_{1}$ is large enough, depending also
on $\eta$, then $\wt C(CC(\eta) \big[K_{1}^{-2} C(B_{0}) )\big]^{1/2})^2 \leq 
\omega^2(x,r)/100$. Then $\wt C\xi^2 \leq \omega^2(x,r)/10$
and altogether $\omega(x,r/3)^2 \leq \omega(x,r)^2/5$,
which implies (\ref{eqn8.4}). 
Lemma \ref{lem8.1} follows.
\qed
\end{proof}

\medskip
We are now ready to prove Theorem \ref{thm8.1}.
We shall proceed as for Theorem \ref{thm5.1},
with a slight modification to take care of the extra
assumption (\ref{eqn8.1}). Notice that except for 
(\ref{eqn8.1}) and the minor fact that we now demand that
$B(x,2r) \i \Omega$ instead of $B(x,r) \i \Omega$,
Lemma~\ref{lem8.1} is the same as Lemma \ref{lem5.1},
with $\theta = 1/3$ and $\beta = 1/2$.

So we start again with a pair $(x,r)$ such that $B(x,2r) \i \O$,
and make a construction to find a small ball $B(x,\rho) \i B(x,r)$, 
on which $u$  is Lipschitz with estimates that depend only
on $B(x_{0},r_{0})$.

We choose the constants the same way as in the proof of 
Theorem \ref{thm5.1} and now split Case 2 into two subcases,
depending on whether we can apply Lemma~\ref{lem8.1} or
not. That is, 

\noindent{\bf Case 2a}:
\begin{equation}
\label{eqn8.25}
\left\{\begin{array}{rcl}
\o(x,r) & \ge & K_2 \\
|b(x,r)| & < & \gamma r \left(1+\o(x,r)\right) \\
(\ref{eqn8.1}) &&\hbox{ holds}
\end{array}\right.
\end{equation}
and \par
\noindent{\bf Case 2b}:
\begin{equation}
\label{eqn8.26}
\left\{\begin{array}{rcl}
\o(x,r) & \ge & K_2 \\
|b(x,r)| & < & \gamma r \left(1+\o(x,r)\right) \\
(\ref{eqn8.1}) &&\hbox{ fails.}
\end{array}\right.
\end{equation}
The other cases stay the same. We treat Case 1 and Case 3
just as we did before. We treat Case 2a as Case 2 before,
except that we apply Lemma \ref{lem8.1} instead of
Lemma \ref{lem5.1}. 

In Case 2b, and since (\ref{eqn8.1}) fails, we know that
$u$ does not vanish anywhere on $B(x,2r/3)$, so there is a
sign $\pm$ such that $\pm u > 0$ on $B(x,2r/3)$. 
We may then apply (\ref{eqn3.18}) to $B(x,2r/3)$
(and to $-u$ if $\pm = -$), and get that $u$ is
Lipschitz on $B(x,2r/9)$, with
\begin{equation}
\label{eqn8.27}
|\nabla u(y)| \leq C (\omega(x,2r/3) + r^{\alpha/2})
\ \hbox{ for almost every } y \in B(x,2r/9).
\end{equation}
Then we just stop, with an even better estimate as
in Case 1. The rest of the argument is the same as
for Theorem \ref{thm5.1}. This completes our proof of 
Theorem \ref{thm8.1}.
\qed

\section{Limits of almost minimizers}

The main result of this section says under suitable
uniformity assumptions, limits of sequences of 
almost mimimzers for $J$, or for $J^+$, are also
almost minimizers.

Let $\O \i \R^n$ be a given open set; there will be
no need here to let $\O$ vary along the sequence.
For the sake of the discussion, let us generalize 
slightly our notion of almost minimizers, and replace 
the function $\kappa r^\alpha$ by more general functions
$h$. We shall only consider continuous nondecreasing functions
$h : (0,+\infty) \to [0,+\infty]$, with
$\lim_{r \to 0} h(r) = 0$; we shall call such function
a gauge function, but as before our main example is
$h(r) = \kappa r^\alpha$.

We say that $u\in K_{\loc}(\O)$ (see the definition
(\ref{eqn1.7})) is an almost minimizer for $J$ in $\Omega$ 
and with the gauge function $h$ if
\begin{equation}
\label{eqn9.1}
J_{x,r}(u) \leq (1+h(r)) J_{x,r}(v)
\end{equation}
for each ball $B(x,r) \i \O$ such that $\overline B(x,r) \i \O$ 
and every  $v\in L^1(B(x,r))$ such that  
$\nabla v\in L^2(B(x,r))$ and $v=u$ on $\p B(x,r)$.
Here $J$ is still as in (\ref{eqn1.12}) and our definition
is a very mild generalization of (\ref{eqn1.11}).

Similarly, we say that $u$ is an almost minimizer for $J^+$ 
in $\Omega$ and with the gauge function $h$ if
$u \in K_{\loc}^+(\O)$ (see (\ref{eqn1.8})) and
\begin{equation}
\label{eqn9.2}
J_{x,r}^+(u) \leq (1+h(r)) J_{x,r}^+(v)
\end{equation}
for each ball $B(x,r) \i \O$ such that $\overline B(x,r) \i \O$ 
and every  $v\in L^1(B(x,r))$ such that  
$\nabla v\in L^2(B(x,r))$ and $v=u$ on $\p B(x,r)$.
See (\ref{eqn1.9}) and compare with (\ref{eqn1.10}). 

For our main statement, we consider a sequence $\{ u_{k} \}$
of almost minimizers in $\O$, and we even allow the fuctions
$q_{\pm}$ that define the functional $J$ or $J^+$ to depend
on $k$. That is, for each $k$, we are given functions 
$q_{k,+}$ and $q_{k,-}$ (here and below, just forget about $q_{k,-}$ 
if we deal with $J^+$).

We nonetheless assume that for each ball $B_0$ with
$\overline B_{0} \i \O$, there is a constant $M(B_{0}) \geq 0$ such that
\begin{equation}
\label{eqn9.3}
|q_{k,+}(x)| + |q_{k,-}(x)| \leq M(B_{0})
 \ \hbox{ for all $x\in B_{0}$ and } k \geq 0.
\end{equation}
We also assume  that the functions $q_{k,\pm}$
converge, in $L^1_{\loc}(\O)$, to a limit $q_{\infty,\pm}$.
That is, for each ball $B_0$ with $\overline B_{0} \i \O$ and each sign
$\pm$,
\begin{equation}
\label{eqn9.4}
\lim_{k \to \infty} \int_{B_{0}} |q_{\infty,\pm}-q_{k,\pm} | = 0.
\end{equation}
We denote by $J^{k}$ (or $J^{k,+}$) the functional defined 
by the $q_{k,\pm}$, and similarly for $J^{\infty}$ 
(or $J^{\infty,+}$).

We also give ourselves functions $u_{k}$ on $\O$, and assume
that for some fixed gauge function $h$ and every $k \geq 0$, 
\begin{equation}
\label{eqn9.5}
u_{k} \text{ is an almost minimizer for $J^k$ in $\O$, with
gauge function $h$},
\end{equation}
or, if we work with $J^+$,
\begin{equation}
\label{eqn9.6}
u_{k} \text{ is an almost minimizer for $J^{k,+}$ in $\O$, with
gauge function $h$.}
\end{equation}

Let us assume that we can find $r_{0} > 0$,
$\alpha > 0$, and $\kappa \geq 0$ such that
\begin{equation}
\label{eqn9.7}
h(r) \leq \kappa r^\alpha \ \hbox{ for } 0 < r \leq r_{0}.
\end{equation}
Or even, a little more generally, that
we can cover $\Omega$ with
open balls $B_{j}$, such that $2B_{j}\i \O$, so that for each $j$
we can find $\alpha =\alpha_j> 0$ and $\kappa =\kappa_j\geq 0$ such that
\begin{equation}
\label{eqn9.8}
u_{k} \text{ is an almost minimizer for $J^k$ in 
$2B_{j}$, with the function $\kappa r^\alpha$},
\end{equation}
or, if we work with $J^+$,
\begin{equation}
\label{eqn9.9}
u_{k} \text{ is an almost minimizer for $J^{k,+}$ in the interior
of $2B_{j}$, with the function $\kappa r^\alpha$}.
\end{equation}
We add this assumption in order to be able to apply the results of the previous sections. 
The fact that we can localize here is not really important.

Our last uniformity assumption is  that for each ball $B_0$ with
$\overline B_{0} \i \O$, there is a constant $C(B_{0}) \geq 0$ such that
\begin{equation}
\label{eqn9.10}
\int_{B_{0}} |\nabla u_{k}|^2 \leq C(B_{0})
\ \text{ for $k$ large,}
\end{equation}
where of course it is important that $C(B_{0})$ does not
depend on $k$.

We claim that under these assumptions, for each
ball $B$ with $\overline B \i \O$ there is a constant $L(B)$
such that for $k$ large,
\begin{equation}
\label{eqn9.11}
\text{each $u_{k}$ is Lipschitz in $B$, with }
|\nabla u_{k}| \leq L(B)
\text { almost everywhere in } B.
\end{equation}
Indeed, cover $\overline B$ with the $B_{j}$ above; by compactness
we only need a finite collection of $B_{j}$. By (\ref{eqn9.10}),
we get a uniform bound for $\int_{\frac{3}{2}B_{j}} |\nabla u_{k}|^2$.
Then we can apply Theorem \ref{thm5.1} or Theorem \ref{thm8.1},
and we get that for $k$  large, $u_{k}$ is $L_{j}$-Lipschitz
on $B_{j}$. This implies that for $k$ large, $u_{k}$ is
locally $L$-Lipschitz in $B$, with $L = \max_{j} L_{j}$;
(\ref{eqn9.11}) follows.

Our final assumption is that there is a function $u_{\infty}$ 
defined on $\O$ such that
\begin{equation}
\label{eqn9.12}
\lim_{k \to \infty} u_{k}(x) = u_{\infty}(x)
\ \text{ for } x\in \O. 
\end{equation}

\noindent{\bf Remark 9.1.} 
We only assume pointwise convergence,
but (\ref{eqn9.11}), we know that it implies
uniform convergence on compact subsets of $\O$. 
It also implies that $u_{\infty}$ is locally Lipschitz, with the same
bounds as in (\ref{eqn9.11}). Indeed, each compact subset
of $\O$ can be covered by a finite collection of balls
$B$ such that $2B \i \Omega$, so it is enough to prove
the uniform convergence on each such ball $B$, which easily
follows from (\ref{eqn9.11}), (\ref{eqn9.12}) and Arzela-Ascoli. 
Note also that $u_k\rightharpoonup u_\infty$ in $W^{1,2}_{\loc}(\Omega)$.

\medskip
\noindent{\bf Remark 9.2.} 
Similarly, if we have a sequence $\{ u_{k} \}$
that satisfies the assumptions above, except (\ref{eqn9.12}),
and if we also know that for each connected component of $\O$
there is a point $x$ such that the family $\{ u_{k}(x) \}$
is bounded, then there is a subsequence of $\{ u_{k} \}$
that converges pointwise on $\O$. Indeed each 
ball $B$ with $\overline B \i \O$ can be connected to one of these points $x$ by a
finite chain of compact balls in $\O$, so (\ref{eqn9.11})
shows that the $u_{k}$, $k$ large, are uniformly bounded
on $B$, and once again Arzela-Ascoli guarantees the claim.

\begin{theorem}\label{thm9.1}
Let $\O$ and the functions $q_{k,\pm}$ and $u_{k}$ satisfy the
conditions above. Then $u_{\infty}$ is an almost minimizer
for $J^\infty$ (for $J^{\infty,+}$ if we assumed (\ref{eqn9.6})
and (\ref{eqn9.9})) in $\O$, with the same gauge function $h$ 
as the $u_{k}$'s.  
In addition, for each ball $B(x,r)$ such that $\overline B(x,r) \i \O$
and for each choice of sign $\pm$,
we have that
\begin{equation}
\label{eqn9.13}
\lim_{k \to \infty} \nabla u_{k}^\pm = \nabla u_{\infty}^\pm
\ \text{ in } L^2(B(x,r)),
\end{equation}
\begin{equation}
\label{eqn9.14}
\int_{B(x,r)} \chi_{\{ \pm u_{\infty} > 0\}} 
\, q_{\infty,\pm} 
= \lim_{k \to \infty}  \int_{B(x,r)} \chi_{\{ \pm u_{k} > 0\}} 
\ q_{k,\pm} \,.
\end{equation}
Hence
\begin{equation}
\label{eqn9.15}
\int_{B(x,r) } |\nabla u_{\infty}^\pm |^2 
= \lim_{k \to \infty} \int_{B(x,r) } |\nabla u_{k}^\pm |^2,
\end{equation}
\begin{equation}
\label{eqn9.16}
\int_{B(x,r) } |\nabla u_{\infty} |^2 
= \lim_{k \to \infty} \int_{B(x,r) } |\nabla u_{k} |^2,
\end{equation}
\begin{equation}
\label{eqn9.17}
J_{x,r}^{\infty}(u_{\infty})
= \lim_{k \to \infty} J^k_{x,r} (u_{k}),
\end{equation}
and similarly for $J^+$.
\end{theorem}

\begin{proof}
We prove all this in the special case of almost minimizers
for $J$; the reader will easily see that the proof carries through
to minimizers for $J^+$ with very minor modifications.

We start with lower semicontinuity estimates, that is the upper bounds
in (\ref{eqn9.14}) and (\ref{eqn9.15}). Fix $B(x,r)$ such that
$\overline B(x,r) \i \O$ and set, for $\varepsilon > 0$,
\begin{equation}
\label{eqn9.18}
W^\pm_{\varepsilon} = \big\{ y\in B(x,r) \, ; \, 
\pm u_{\infty}(y) > \varepsilon \big\}.
\end{equation}
By Remark 9.1, 
the $u_{k}'s$ converge to $u_{\infty}$ uniformly on 
$\overline B(x,r)$, so $\pm u_{k}(y) \geq \varepsilon/2$
on $W^\pm_{\varepsilon}$ for $k$ large enough. For such $k$,
\begin{eqnarray}
\label{eqn9.19}
\int_{W^\pm_{\varepsilon}}   
\chi_{\{ \pm u_{\infty} > 0\}}  \, q_{\infty,\pm}
&=& \int_{W^\pm_{\varepsilon}}   \, q_{\infty,\pm}
= \int_{W^\pm_{\varepsilon}}  \chi_{\{ \pm u_{k} > 0\}}
 \, q_{\infty,\pm}
 \nonumber
 \\
 &\leq& \int_{W^\pm_{\varepsilon}}   
\chi_{\{ \pm u_{k} > 0\}} \, q_{k,\pm}
+\int_{W^\pm_{\varepsilon}} |q_{\infty,\pm}-q_{k,\pm}|
  \nonumber
 \\
 &\leq& \int_{B(x,r)}   
\chi_{\{ \pm u_{k} > 0\}} \, q_{k,\pm}
+\int_{B(x,r)} |q_{\infty,\pm}-q_{k,\pm}|.
\end{eqnarray}
By (\ref{eqn9.4}) the second term tends to $0$ when $k$ tends
to $\infty$, so
\begin{equation}
\label{eqn9.20}
\int_{W^\pm_{\varepsilon}}  \chi_{\{ \pm u_{\infty} > 0\}}  \, q_{\infty,\pm}
\leq \liminf_{k \to \infty}
\int_{B(x,r)}   \chi_{\{ \pm u_{k} > 0\}} \, q_{k,\pm}
\end{equation}
and, since  this holds for all $\varepsilon > 0$ and $\{\chi_{W^\pm_{\varepsilon}}\}$ 
is nondecreasing in $\varepsilon$ we have 
\begin{equation}
\label{eqn9.21}
\int_{B(x,r)}  \chi_{\{ \pm u_{\infty} > 0\}}  \, q_{\infty,\pm}
\leq \liminf_{k \to \infty}
\int_{B(x,r)}   \chi_{\{ \pm u_{k} > 0\}} \, q_{k,\pm}
\, .
\end{equation}
Let us also check that
\begin{equation}
\label{eqn9.22}
\int_{B(x,r)} |\nabla u_{\infty}^\pm |^2 
\leq \liminf_{k \to \infty}
\int_{B(x,r)}  |\nabla u_{k}^\pm |^2 .
\end{equation}
Observe that the $u^\pm_{k}$ are uniformly Lipschitz
on $\overline B(x,r)$, just because the $u_{k}'s$ are
(by (\ref{eqn9.11})). Morevoer they converge uniformly to
$u_{\infty}^\pm$ (multiply by $\pm 1$ and compose with $\max(0,\cdot)$). 
Then $\{u_{k}^\pm\}$ also converges weakly in $W^{1,2}_{\loc}(\Omega)$ 
to $u^\pm_\infty$ and (\ref{eqn9.22})  
follows from the lower semicontinuity of the norm in 
$W^{1,2}(B(x,r))$. Notice that by (\ref{eqn9.21}), (\ref{eqn9.22}) and the definition of $J_{x,r}$ we have
\begin{equation}
\label{eqn9.23}
J^{\infty}_{x,r} (u_{\infty})
\leq \liminf_{k \to \infty} J^{k}_{x,r} (u_{k}).
\end{equation}

\medskip
Next we show that $u_{\infty}$ is an almost minimizer
for $J^\infty$. Let $\overline B(x,r) \i \O$, 
and let $v \in L^1(B(x,r))$ be such that  
$\nabla v\in L^2(B(x,r))$ and $v=u_{\infty}$ on $\p B(x,r)$.
We want to use $v$ to construct a good competitor for $u_{k}$,
$k$ large, in a slightly larger ball; then we will use the fact that
$u_{k}$ is an almost minimizer and get valuable information.
Let $\varepsilon > 0$ be small, and define $v_{k,\varepsilon}$
by
\begin{equation}
\label{eqn9.24}
\left\{\begin{array}{rll} 
v_{k,\varepsilon}(y) &=  v(y) &\text{ for } y\in B(x,r)
\\
v_{k,\varepsilon}(y) &=  u_{k}(y) 
&\text{ for } y\in \O \sm B(x,(1+\varepsilon)r)
\\
v_{k,\varepsilon}(y) &=   (1-a(y)) u_{\infty}(y) + a(y) u_{k}(y) 
&\text{ for } y\in B(x,(1+\varepsilon)r) \sm B(x,r),
\end{array}\right.
\end{equation}
where we set $a(y) = \frac{|y-x| - r}{\varepsilon r}$. 

We want to use $v_{k,\varepsilon}$ as a competitor,
so let us check that 
\begin{equation}
\label{eqn9.25}
v_{k,\varepsilon} \in W^{1,2}_{\loc}(\Omega).
\end{equation}
It is enough to show that
\begin{equation}
\label{eqn9.26}
v_{k,\varepsilon} \in W^{1,2}(B(x,(1+\varepsilon/2)r)),
\end{equation}
because it is clear that it is locally Lipschitz in $\O \sm B(x,r)$.

Let us start the proof of (\ref{eqn9.26}).
Consider the function $w$ such that $w(y) = v(y)-u_{\infty}(y)$
for $y\in B(x,r)$, and $w(y) = 0$ elsewhere. We know that
$w\in W^{1,2}(B(x,r))$, and that the trace
of $w$ on $\p B(x,r)$ is zero. Then $w\in W^{1,2}(\Omega)$,
i.e., the gluing along $\p B(x,r)$ does not create any additional
part of the distribution derivative of $w$. 
See for instance Lemma 14.4 in \cite{D}. 
Let $w_{1}$ be such that
$w_{1}(y) = 0$ for $y\in B(x,r)$, and 
$w_{1}(y) = a(y) (u_{k}(y)-u_{\infty}(y))$ 
on $B(x,(1+\varepsilon)r) \sm B(x,r)$. The function $w_1$ is Lipschitz, and
$w_{1} \in W^{1,2}(B(x,(1+\varepsilon/2)r))$ as well.
Since $v_{k,\varepsilon} = u_{\infty} + w + w_{1}$ in 
$B(x,(1+\varepsilon)r)$, we get (\ref{eqn9.26}) and (\ref{eqn9.25}).
Since $v_{k,\varepsilon} = u_{k}$ on $\p B(x,(1+\varepsilon)r)$,
we use it as a competitor for $u_{k}$, in the ball 
of radius $\widetilde r = (1+\varepsilon)r$ and center $x$.
By (\ref{eqn9.5}) and (\ref{eqn9.1}),
\begin{equation}
\label{eqn9.27}
J_{x,\widetilde r}^k(u_{k}) \leq (1+h(\widetilde r)) 
J_{x,\widetilde r}^k(v_{k,\varepsilon}).
\end{equation}
Next
\begin{eqnarray}
\label{eqn9.28}
J_{x,\widetilde r}^k(v_{k,\varepsilon}) &=&
J_{x,r}^k(v_{k,\varepsilon}) 
+ \int_{B(x,\widetilde r) \sm B(x,r)} |\nabla v_{k,\varepsilon}|^2
+ \int_{B(x,\widetilde r) \sm B(x,r)} 
\big[\chi_{\{ v_{k,\varepsilon} > 0\}}  \, q_{\infty,+}^k
+ \chi_{\{ v_{k,\varepsilon} < 0\}}  \, q_{\infty,-}^k \big]
\nonumber
\\
&\leq& J_{x,r}^k(v_{k,\varepsilon}) 
+ \int_{B(x,\widetilde r) \sm B(x,r)} |\nabla v_{k,\varepsilon}|^2
+ C |B(x,\widetilde r) \sm B(x,r)|
\nonumber
\\
&\leq& J_{x,r}^k(v_{k,\varepsilon}) 
+ \int_{B(x,\widetilde r) \sm B(x,r)} |\nabla v_{k,\varepsilon}|^2
+ C \varepsilon r^n,
\end{eqnarray}
where $C$ is a constant
that may depend on $B$, our sequence, and even $v$,
but not on $\varepsilon$ or $k$. 
Notice that $J_{x,r}^k(v_{k,\varepsilon}) = J_{x,r}^k(v)$
(by (\ref{eqn9.24})), and that on $B(x,\widetilde r) \sm B(x,r)$,
\begin{equation}
\label{eqn9.29}
|\nabla v_{k,\varepsilon}| \leq |\nabla u_{\infty}| + |\nabla u_{k}|
+ |\nabla a|\,|u_{\infty} - u_{k}| 
\leq C + C\varepsilon^{-1} 
||u_{\infty} - u_{k}||_{L^\infty(B(x,\widetilde r))},
\end{equation}
where $C$ depends on $n$ and $L(B)$ in (\ref{eqn9.11}).
Thus by (\ref{eqn9.28})
\begin{equation}
\label{eqn9.30}
J_{x,\widetilde r}^k(v_{k,\varepsilon})
\leq J_{x,r}^k(v) + C\varepsilon^{-2} r^n
||u_{\infty} - u_{k}||_{L^\infty(B(x,\widetilde r))}^2
+ C \varepsilon r^n
\leq J_{x,r}^k(v) + C\varepsilon r^n
\end{equation}
if $k$ is large enough (recall that $\{ u_{k}\}$
converges to $u_{\infty}$ uniformly on compact subsets,
and we restrict to $\varepsilon$ so small that 
$\overline B(x,\widetilde r) \i \O$). 

By (\ref{eqn9.23}),
$J^{\infty}_{x,r} (u_{\infty}) \leq J^{k}_{x,r} (u_{k})+ \varepsilon$
if $k$ is large enough, and so
\begin{eqnarray}
\label{eqn9.31}
J^{\infty}_{x,r} (u_{\infty}) 
&\leq& J^{k}_{x,r} (u_{k})+ \varepsilon
\leq J^{k}_{x,\widetilde r} (u_{k})+ \varepsilon
\leq (1+h(\widetilde r)) J_{x,\widetilde r}^k(v_{k,\varepsilon})
+\varepsilon
\nonumber
\\
&\leq& (1+h(\widetilde r)) J_{x,r}^k(v) + C(1+h(\widetilde r))\varepsilon r^n +C\varepsilon
\end{eqnarray}
by (\ref{eqn9.27}) and (\ref{eqn9.30}). Since in addition
\begin{eqnarray}
\label{eqn9.32}
|J_{x,r}^k(v)-J_{x,r}^\infty(v)| 
&=& \Big|\int_{B(x,r)} \chi_{\{ v > 0\}}  \,[ q_{k,+} - q_{\infty,+} ]
+  \int_{B(x,r)} \chi_{\{ v < 0\}}  \,[ q_{k,-} - q_{\infty,-} ] \Big|
\nonumber
\\
&\leq& \int_{B(x,r)} |q_{k,+} - q_{\infty,+}| + |q_{k,-} - q_{\infty,-}|
\leq \varepsilon
\end{eqnarray}
for $k$ large (by (\ref{eqn9.4})), we get that
\begin{equation}
\label{eqn9.33}
J^{\infty}_{x,r} (u_{\infty}) \leq (1+h(\widetilde r)) J_{x,r}^\infty(v) 
+ C(1+h(\widetilde r))(1+ r^n)\varepsilon.
\end{equation}
Letting $\varepsilon$ tend to $0$, and using the continuity of
$h$, we get that
\begin{equation}
\label{eqn9.34}
J^{\infty}_{x,r} (u_{\infty}) \leq (1+h(r)) J_{x,r}^\infty(v).
\end{equation}
So $u_{\infty}$ is an almost minimizer.

\medskip
Next we want to take care of the lower bounds in
(\ref{eqn9.14}) and (\ref{eqn9.15}). For this the main
point is to control what happens when $u_{\infty} = 0$.
Again fix $\overline B(x,r) \i \Omega$.
Set $Z = \big\{ y\in  \overline B(x,r) \, ; \, u_{\infty}(y) = 0\big\}$.
Then let $Z_{0}$ be the set of Lebesgue density point of $Z$,
i.e., points $y\in Z$ such that $\lim_{t \to 0} t^{-d} |B(y,t) \sm Z| = 0$,
and recall that $|Z \sm Z_{0}| = 0$.

Let $\varepsilon_{0} > 0$ be so small that 
\begin{equation}
\label{eqn9.35}
\overline B(x,r+10\varepsilon_{0}) \i \Omega
\end{equation}
Then let $\varepsilon \in (0,\varepsilon_{0})$ be small. 
For each $y\in Z_{0}$, pick a ball $B_{y} = B(y,r_{y})$ such that
\begin{equation}
\label{eqn9.36}
r_{y} < \varepsilon \ \text{ and } \ 
|B(y,10r_{y}) \sm Z| < \varepsilon^{n} |B(y,r_{y})|.
\end{equation}
Then use Vitali's covering lemma (see the first pages of \cite{S}) 
to find a covering of $Z_0$ by a countable collection
of balls $B(y,5r_{y})$, $y\in Y \i Z_0$,  
such that the $B_{y}$'s are disjoint.  
To complete the covering of $Z$, we cover $Z\sm Z_{0}$ with a collection 
of balls $D_{j} = B(z_{j},t_{j})$, so that $t_{j} \leq \varepsilon$ for all $j$, and
$\sum_{j} t_{j}^n \leq \varepsilon$. Finally, we use the fact
that $Z$ is compact to cover it with only a finite subcollection
of the $B_{y}$ and the $D_{j}$.

Notice that all the $B(y,10r_{y})$ and the $D_{l}$ are contained
in $\overline B(x,r+10\varepsilon_{0})$. By (\ref{eqn9.11}) there is a constant $L \geq 0$, 
independent of $\varepsilon$, such that each of the $u_{k}$'s 
and $u_\infty$ are 
$L$-Lipschitz on each $B(y,10r_{y})$ and each $D_{l}$.
This will be 
used to estimate the contribution of these balls to $J^k(u_{k})$.

For the $D_{j}'$ a rough estimate yields 
\begin{equation}
\label{eqn9.37}
J^k_{z_{j},t_{j}}(u_{k}) \leq \int_{B(z_{j},t_{j})}
|\nabla u_{k}|^2 + q_{k}^+ + q_{k}^- \leq C t_{j}^n.
\end{equation}
Note that (\ref{eqn9.36}) guarantees that for $z\in B(y,5r_{y})$,
$B(z,\varepsilon r_{y})$ meets $Z$. Since $u_{\infty}$ vanishes 
on $Z$ and is $L$-Lipschitz we have that 
\begin{equation}
\label{eqn9.38}
|u_{\infty}(z)| \leq L \varepsilon r_{y}
\ \text{ for } z\in B(y,5r_{y}).
\end{equation}
From
(\ref{eqn9.38}) and the uniform convergence of the $u_{k}$'s 
on $B(y,5r_{y}) \i B(x,r+10\varepsilon_{0})$ we have that for $k$ large,
\begin{equation}
\label{eqn9.39}
|u_{k}(z)| \leq 2L \varepsilon r_{y}
\ \text{ for } z\in B(y,5r_{y}).
\end{equation}
Since we only have a finite collections of balls $B(y,5r_{y})$,
$k$ large enough works for all of them at once.

We now need to compare $u_{k}$ with an appropriate competitor. Fix $y$ and set
\begin{equation}
\label{eqn9.40}
\left\{\begin{array}{rll} 
v(z) &=  u_{k}(z) &\text{ for } z\in \O \sm B(y,5r_{y})
\\
v(z) &=  0 &\text{ for } z\in B(y,(5-\varepsilon)r_{y})
\\
v(z) &=  u_{k}(z) - (\varepsilon r_{y})^{-1} [5r_{y}-|z-y|] u_{k}(z) 
&\text{ for } z\in  B(y,r_y) \sm B(x,(5-\varepsilon)r_{y}).
\end{array}\right.
\end{equation}
Notice that $v$ is piecewise Lipschitz and continuous
near $B(y,5r_{y})$, so it is an acceptable competitor
for $u_{k}$. Thus
\begin{equation}
\label{eqn9.41}
J_{y,5r_{y}}^k(u_{k}) \leq (1+h(5r_{y})) J_{y,5r_{y}}^k(v)
\leq (1+h(5\varepsilon)) J_{y,5r_{y}}^k(v).
\end{equation}
Since on $B(y,5r_{y})$, by (\ref{eqn9.40}) and (\ref{eqn9.39}) we have
\begin{equation}
\label{eqn9.42}
|\nabla v| \leq |\nabla u_{k}| + (\varepsilon r_{y})^{-1}
||u_{k}||_{L^\infty(B(y,5r_{y})} \leq 3L,
\end{equation}
then
\begin{equation}
\label{eqn9.43}
J_{y,5r_{y}}^k(v) \leq \int_{B(x,5r_{y}) \sm B(x,(5-\varepsilon)r_{y})}
|\nabla v|^2 + q_{k}^+ + q_{k}^-
\leq C (1+L)^2 \varepsilon r_{y}^n
\end{equation}
and hence by (\ref{eqn9.41})
\begin{equation}
\label{eqn9.44}
J_{y,5r_{y}}^k(u_{k}) \leq (1+h(5\varepsilon)) J_{y,5r_{y}}^k(v)
\leq C (1+h(5\varepsilon))(1+L)^2 \varepsilon r_{y}^n
\leq 2C (1+L)^2 \varepsilon r_{y}^n
\end{equation}
if $\varepsilon_{0}$ was chosen so small that 
$h(5\varepsilon_{0}) \leq 1$.
Using (\ref{eqn9.37}), (\ref{eqn9.44}), the 
definition of the $D_{j}$, and 
the fact that the $B(y_{r},r_{y})$'s 
are disjoint and contained in $B(x,r+10\varepsilon_{0})$, we get
\begin{eqnarray}
\label{eqn9.45}
\sum_{j} J^k_{z_{j},t_{j}}(u_{k})
+ \sum_{y} J_{y,5r_{y}}^k(u_{k})
&\leq& C \sum_{j} t_{j}^n  
+ C (1+L)^2 \varepsilon \sum_{y} r_{y}^n
\nonumber
\\
&\leq&  C \varepsilon + C (1+L)^2 \varepsilon \sum_{y} |B(y_{r},r_{y})|
\leq C' \varepsilon.
\end{eqnarray}
 Here $C'$ depends on the Lipschitz constant, $r$, $\varepsilon_0$ and $n$, 
 but not on $\varepsilon$ or $k$.

Set
\begin{equation}
\label{eqn9.46}
V = \left(\cup_{j} D_{j}\right) \cup \left(\cup_{y} B(y,5r_{y})\right)
\ \text{ and } \ 
A = \overline B(x,r) \sm V.
\end{equation}
Notice that $V$ is an open set that contains the compact
set $Z$, so $A$ is compact and $|u_{\infty}| > 0$ on $A$.
Let $\eta > 0$ be such that $|u_{\infty}| \geq \eta$ on $A$.
We can see $A$ 
as the union of  two compact  subsets $A_{+}$, where 
$u_{\infty} \geq \eta$, and $A_{-}$, where $u_{\infty} \leq -\eta$.

Choose an open neighborhood $U_{\pm}$ of $A_{\pm}$,
such that $\overline U_{\pm}$ is still contained in
$\big\{ y\in \O \, ; \, \pm u_{\infty} \geq \eta/2 \}$.
By Theorem \ref{thm3.2} and the more precise (\ref{eqn3.28}) 
we get that the $\nabla u_{k}$'s 
are H\"older continuous on 
$U_{\pm}$, with a fixed exponent $\beta$ and uniform 
bounds on the H\"older constant. 
Recall that the $u_{k}^\pm$  themselves 
converge to $u_{\infty}^\pm$, uniformly on $U_{\pm}$. 

We claim that the $\nabla u_k$'s 
also converge uniformly to $\nabla u_\infty$ on $U_\pm$.
Otherwise there exist $\tau > 0$ and a subsequence $\{\nabla u_{k'}\} $ of 
$\{\nabla u_{k} \}$ such that 
$\|\nabla u_{k'}-\nabla u_\infty\|_\infty\ge \tau$ for all $k'$. 
On the other hand by the uniform $C^{1,\beta}$ bounds on the $u_k$'s, 
we can extract a subsequence that converges uniformly
on $U_{\pm}$. But this subsequence also converges weakly and its limit is $ u_\infty$ (see Remark 9.1), which contradicts our assumption.
So  $\{\nabla u_{k} \}$ converges to $\nabla u_{\infty}$,
uniformly on $A_{\pm}$, and
\begin{equation}
\label{eqn9.47}
\int_{A_{\pm}} |\nabla u_{\infty}^\pm|^2 
= \lim_{k \to \infty} \int_{A_{\pm}} |\nabla u_{k}^\pm|^2.
\end{equation}
Furthermore, by (\ref{eqn9.46}) and (\ref{eqn9.45})
we have that for $k$ large enough, 
\begin{equation}
\label{eqn9.48}
\int_{B(x,r) \sm A_{\pm}} |\nabla u_{k}^\pm|^2
\leq \int_{ V} |\nabla u_{k}^\pm|^2 \leq C' \varepsilon.
\end{equation}
Thus combining (\ref{eqn9.47}) and (\ref{eqn9.48}) we obtain
\begin{equation}
\label{eqn9.49}
\limsup_{k \to \infty} \int_{B(x,r)} |\nabla u_{k}^\pm|^2
\leq \int_{B(x,r)} |\nabla u_{\infty}^\pm|^2 + C' \varepsilon.
\end{equation}
Since (\ref{eqn9.49}) holds for  $\varepsilon > 0$ arbitrarily small,
(\ref{eqn9.49}) and (\ref{eqn9.22}) prove (\ref{eqn9.15}).

Next fix a sign $\pm$ and observe that for $k$ large,
both $\pm u_{\infty}$ and $\pm u_{k}$ are positive
on $A_{\pm}$. Then for $k$ large enough,  by (\ref{eqn9.4}) we have
\begin{eqnarray}
\label{eqn9.50}
\int_{A_{\pm}} \chi_{\{ \pm u_{k} > 0\}} \, q_{k,\pm} 
&=& \int_{A_{\pm}} \chi_{\{ \pm u_{\infty} > 0\}} \, q_{k,\pm}
\leq \int_{A_{\pm}} \chi_{\{ \pm u_{\infty} > 0\}} \, q_{\infty,\pm}
+\int_{B(x,r)} |q_{k,\pm}-q_{\infty,\pm}|
\nonumber
\\
&\leq& \int_{A_{\pm}} \chi_{\{ \pm u_{\infty} > 0\}} \, q_{\infty,\pm}
+ \varepsilon.
\end{eqnarray}
 Since 
$\chi_{\{ \pm u_{k} > 0\}} = 0$ on $A_{\mp}$, we also have that
\begin{equation}
\label{eqn9.51}
\int_{B(x,r) \sm A_{\pm}} \chi_{\{ \pm u_{k} > 0\}} \, q_{k,\pm} 
= \int_{V} \chi_{\{ \pm u_{k} > 0\}} \, q_{k,\pm}
\leq \sum_{j} J^k_{z_{j},t_{j}}(u_{k})
+ \sum_{y} J_{y,5r_{y}}^k(u_{k}) \leq C' \varepsilon
\end{equation}
by (\ref{eqn9.45}). We add (\ref{eqn9.50}) and (\ref{eqn9.51}) to obtain
\begin{equation}
\label{eqn9.52}
\limsup_{k \to \infty} 
\int_{B(x,r)} \chi_{\{ \pm u_{k} > 0\}} \, q_{k,\pm}
\leq \int_{B(x,r)} \chi_{\{ \pm u_{\infty} > 0\}} \, q_{\infty,\pm}
+ C' \varepsilon + \varepsilon.
\end{equation}
Since $\varepsilon$ is arbitrarly small, (\ref{eqn9.21}) and 
(\ref{eqn9.52}) yield (\ref{eqn9.14}). 

Now we prove (\ref{eqn9.16}) and (\ref{eqn9.17}). 
Recall that for all $k$,
$\nabla u_{k}^+ = \nabla u_{k}$ and $\nabla u_{k}^- = 0$
on the open set $\big\{ x\in \O \, ; \, u_{k}(x) > 0\big\}$.
We have a similar description on 
$\big\{ x\in \O \, ; \, u_{k}(x) < 0\big\}$, and
on the remaining set $\big\{ x\in \O \, ; \, u_{k}(x) = 0\big\}$,
we know that $\nabla u_{k}^+$ and $\nabla u_{k}^-$ exist
almost everywhere (because these functions are locally
Lipschitz). Since these derivatives can be computed to be
zero on any Lebesgue density point of 
$\big\{ x\in \O \, ; \, u_{k}(x) = 0\big\}$, we get that
$\nabla u_{k}^+ = \nabla u_{k}^- =  \nabla u_{k} = 0$
almost everywhere on that set. Then
$\int_{B(x,r)} |\nabla u_{k}|^2 = \int_{B(x,r)} |\nabla u_{k}^+|^2
+ \int_{B(x,r)} |\nabla u_{k}^-|^2$. A similar description
holds for $\nabla u_{\infty}$, and now (\ref{eqn9.16})
follows from (\ref{eqn9.15}). Also, (\ref{eqn9.17}) follows
from (\ref{eqn9.14}) and (\ref{eqn9.16}).

To prove (\ref{eqn9.13}) recall than
in $B(x,r)$, $u_{\infty}$ is the uniform limit of the $u_{k}$.
By composing with the Lipschitz function $\max(0,\cdot)$,
we get that $u_{\infty}^\pm$ is the uniform limit of the ${u_{k}^\pm}'s$.
Recall that these functions are Lipschitz in
$B(x,r)$, with uniform estimates. Then (by elementary  distribution
theory) $\nabla u_{\infty}^+$ is the weak limit of the $\nabla u_{k}^+$,
in $L^2(B(x,r))$. Moreover (9.15) ensures that $\{\nabla u_k^\pm\}$ converges strongly in $L^2(B(x,r))$ to $\nabla u_\infty^\pm$. 
This proves (\ref{eqn9.13}) for $u_{\infty}^\pm$. This completes 
our proof of Theorem~\ref{thm9.1}.
\qed
\end{proof}

\medskip
\noindent{\bf Remark 9.3.} 
Suppose that instead of (\ref{eqn9.5}) or (\ref{eqn9.6}),
we have that either
\begin{equation}
\label{eqn9.53}
u_{k} \text{ is an almost minimizer for $J^k$ in $\O$, with
gauge function $h_{k}$},
\end{equation}
or, if we work with $J^+$, that
\begin{equation}
\label{eqn9.54}
u_{k} \text{ is an almost minimizer for $J^{k,+}$ in $\O$, with
gauge function $h_{k}$,}
\end{equation}
where the $h_{k}$ are continuous gauge functions
such that 
\begin{equation}
\label{eqn9.55}
\lim_{k \to \infty} h_{k}(r) = 0
\ \text{ for every } r > 0.
\end{equation}
Then the function $u_\infty$ which appears in (\ref{eqn9.12}) satisfies 
\begin{equation}
\label{eqn9.56}
u_{\infty} \text{ is a minimizer for $J^\infty$ in $\O$}
\end{equation}
or, if we work with $J^+$, 
\begin{equation}
\label{eqn9.57}
u_{\infty} \text{ is a minimizer for $J^{\infty,+}$ in $\O$,}
\end{equation}
where minimizer means almost minimizer with the gauge
function $h=0$.

Indeed, set $H_{l}(r) = \sup_{k \geq l} h_{k}(r)$.
It is clear that each $H_{l}$ is a gauge function,
and it is easy to see that it is  also  continuous. Most often we consider the case when the
sequence $\{ h_{k} \}$ is nondecreasing and
$H_{l}=h_{l}$. Apply Theorem~\ref{thm9.1} to the 
sequence $\{ u_{k} \}$, $k \geq l$. We obtain that
$u_\infty$ is an almost minimizer, with the gauge
function $H_{l}$. This means that for every ball
$B(x,r)$, (\ref{eqn9.1}) or (\ref{eqn9.2}) holds with
$H_{l}(r)$. But for each $r$, $\lim_{l \to  \infty} H_{l}(r) = 0$,
so $u_{\infty}$  is in fact a minimizer.

\medskip
Next we apply Theorem~\ref{thm9.1} 
and Remark 9.3 
to the special case of blow-up limits. Let
$\O$ be given, and let $u$ be an almost minimizer
in $\Omega$ (for $J$ or $J^+$). Here we just
work with one gauge function $h$, typically
$h(r) = \kappa r^\alpha$,
and one pair of bounded functions $q_{\pm}$. 

Before we discuss blow-up sequences, let us say
a few words about dilations. For every $x\in \Omega$
and $r > 0$, set
\begin{equation}
\label{eqn9.58}
\Omega_{x,r} = \big\{ y\in \R^n ; x+ry \in \O \big\}
= \frac{1}{r} [\Omega - x],
\end{equation}
\begin{equation}
\label{eqn9.59}
q^{(x,r)}_{\pm}(y) = q_{\pm}(x+ry) \ \text{ for }
y \in \Omega_{x,r} \,,
\end{equation}
and
\begin{equation}
\label{eqn9.60}
u^{(x,r)}(y) = \frac{1}{r} u(x+ry) \ \text{ for }
y \in \Omega_{x,r}.
\end{equation}
We use the functions $q^{(x,r)}_{\pm}$ to define
a functional $J^{x,r}$, or $J^{x,r,+}$ if we work with $J^+$.
We claim that
\begin{eqnarray}
\label{eqn9.61}
&\text{$u^{(x,r)}$ is an almost minimizer in $\Omega_{x,r}$,
for $J^{x,r}$ or $J^{x,r,+}$,}&
\nonumber
\\
&\text{with the gauge function
$h(r \,\cdot)$.}&  
\end{eqnarray}
This claim is a straightforward exercise on the chain rule.
We do the computations for $J$; those for $J^+$ are analogous.
Let $v$ be a competitor for $u^{(x,r)}$ in
the ball $B = B(y,t)$, which means that
$\overline B(y,t) \i \Omega_{x,r}$, $v \in W^{1,2}(B)$,
and its trace on $\p B$ is the same as the trace of $u^{x,r}$.
Keep $v(y) = u^{x,r}(y)$ on $\Omega_{x,r}\sm B$.
Set $w(z) = r v(r^{-1}(z-x))$ for $z\in \Omega$.
It is easy to see that $w$ is a competitor for $u$
in $B' = x + r B$. Hence
$J_{B'}(u) \leq (1+h(tr)) J_{B'}(w)$,
where for short we also set $J_{B}(u) = J_{y,t}(u)$
when $B = B(y,t)$. Now
\begin{eqnarray}
\label{eqn9.62}
J_{B}^{x,r}(u^{x,r}) &=& \int_{B} |\nabla u^{x,r}|^2
+ \chi_{\{u^{x,r} > 0\}} \, q_{+}^{x,r}
+ \chi_{\{u^{x,r} < 0\}} \, q_{-}^{x,r}
\nonumber
\\
&=& r^{-n} \int_{B'} |\nabla u|^2
+ \chi_{\{u > 0\}} \, q_{+}
+ \chi_{\{u < 0\}} \, q_{-}
= r^{-n} J_{B'}(u).
\end{eqnarray}
The same computation yields
$J_{B}^{x,r}(v) = r^{-n} J_{B'}(w)$, and now
\begin{equation}
\label{eqn9.63}
J_{B}^{x,r}(u^{x,r}) = r^{-n} J_{B'}(u) \leq r^{-n} (1+h(tr)) J_{B'}(w)
= (1+h(tr)) J_{B}^{x,r}(v),
\end{equation}
as needed for (\ref{eqn9.61}).

\medskip
We now focus on blow-up sequences for almost minimizers.
That is let   $u$ be an almost minimizer,
fix a point $x$, and take a sequence $r_{k}$
of radii, with 
\begin{equation}
\label{eqn9.64}
\lim_{k \to \infty}  r_{k} = 0.
\end{equation}
Set $u_{k} = u^{x,r_{k}}$. That is,
define $u_{k}$ on $\Omega_{k} = \frac{1}{r_{k}} [\O - x]$
by
\begin{equation}
\label{eqn9.65}
u_{k}(y) = \frac{1}{r_{k}} u(x+r_{k} y).
\end{equation}
We say that \underbar{the sequence}
$\{ u_{k} \}$  \underbar{converges} if there is a function
$u_{\infty}$, defined on $\R^n$, such that
\begin{equation}
\label{eqn9.66}
u_{\infty}(y) = \lim_{k \to \infty} u_{k}(y)
\ \text{ for every } y\in \R^n.
\end{equation}
Notice that for each $y\in \R^n$, $y\in \O_{k}$
for $k$ large (because if $B(x,a) \i \O$, then
$\Omega_{k}$ contains $B(0,a/r_{k})$), so
(\ref{eqn9.66}) makes sense. We apparently take a weak
definition of convergence, but we shall see soon that
when $h(t) \leq \kappa r^\alpha$, it implies uniform 
convergence on compact sets, and even (if the $q_{\pm}$
are continuous, say) the convergence
of the gradients in $L_{\loc}^2(B)$ as in Theorem \ref{thm9.1}.

Notice also that if $\{ u_{k} \}$  converges, then
\begin{equation}
\label{eqn9.67}
u(x) = 0,
\end{equation}
because otherwise $\{ u_{k}(0) \}$ diverges. This does not
disturb us as we are not interested in blow ups at points where $|u|>0$.

A  \underbar{blow-up limit of $u$ at $x$} is a function
$u_{\infty} : \R^n \to \R$ such that, for some choice of $\{ r_{k} \}$
with $\lim_{k \to \infty}  r_{k} = 0$, the sequence $\{u_{k}\}$ converges
to $u_{\infty}$ (as in (\ref{eqn9.66})). Of course different
sequences may give different blow-up limits.

For the following discussion, we shall assume that
for some choice of $\kappa$, $\alpha > 0$, $r_{0} > 0$,
\begin{equation}
\label{eqn9.68}
h(r) \leq \kappa r^\alpha
\ \text{ for } 0 < r < r_{0}
\end{equation}
so that we can use the results of the previous sections.
Moreover we assume that $x$ is a Lebesgue point of $q_{+}$ and $q_{-}$, 
in the precise sense that
\begin{equation}
\label{eqn9.69}
\lim_{r \to 0} \, \fint_{B(x,r)} |q_{+}(y)-q_{+}(x)|
+ |q_{-}(y)-q_{-}(x)| = 0.
\end{equation}
This is the case if $q_{+}$ and $q_{-}$
are continuous at $x$. The following will be an easy
consequence of Theorem \ref{thm9.1}.

\begin{theorem}\label{thm9.2}
Let $u$ be an almost minimizer for $J$ or  $J^+$
(associated to bounded functions $q_{\pm}$)
in $\Omega$, with a gauge function $h$ such that
(\ref{eqn9.68}) holds, and let $x\in \O$ be such that
$u(x) = 0$ and (\ref{eqn9.69}) holds.
Then for each sequence $\{ r_{k} \}$ in $(0,\infty)$ that
tends to $0$, we can extract a subsequence $\{ r_{k_{j}} \}$
such that $\{ u_{k_{j}} \}$ converges. 

Also, if $\{ r_{k} \}$ is a sequence  in $(0,\infty)$ 
that tends to $0$ and for which $\{ u_{k} \}$ converges
to a limit function $u_{\infty}$, then $u_{\infty}$ is a minimizer in
$\R^n$, for the functional  $J^{\infty}$ or $J^{{\infty},+}$
associated to the constant functions 
$q^\infty_{\pm} = q_{\pm}(x)$. In addition, for
each $R > 0$,
\begin{equation}
\label{eqn9.70}
\text{$\{ u_{k} \}$ converges to $u_{\infty}$ uniformly
in $B(0,R)$}
\end{equation}
and
\begin{equation}
\label{eqn9.71}
\text{$\{ \nabla u_{k} \}$ converges to $\nabla u_{\infty}$ 
in $L^2(B(0,R))$}.
\end{equation}
\end{theorem}

\begin{proof}
Let $u$ and $x$ be as in the statement. Also let 
$\{ r_{k} \}$ be any sequence of positive numbers
that tends to $0$, and $R > 0$ be given. Set $B = B(0,R)$ and
$B_{k} = B(x,r_{k} R)$, and observe that $B_{k} \i \O$
for $k$ large. For such $k$, by (\ref{eqn9.61})
$u_{k}$ is an amost minimizer in $B$, with the gauge
function $h_{k} = h(r_{k} \, \cdot)$, and with the functional
associated to $q_{\pm}^k(y) = q_{\pm}(x+r_{k}y)$.
Notice that for $k$ large, by (\ref{eqn9.68})
\begin{equation}
\label{eqn9.72}
h_{k}(t) = 
h(r_{k} t) \leq \kappa (r_{k}t)^\alpha = r_{k}^\alpha \kappa t^\alpha
\ \text{ for } 0 < t < 2R.
\end{equation}
To apply Theorem \ref{thm9.1} to the $u_{k}$'s,  
in the domain $B$, we need to check the assumptions.
We have (\ref{eqn9.3}) because $q_{+}$ and $q_{-}$
are always assumed to be bounded, and (\ref{eqn9.4})
(even for the full $B$) follows from (\ref{eqn9.69})
(compute the average on $B(x,r_{k}R)$). The limiting
functions are $q_{+}(x)$ and $q_{-}(x)$, as expected.

We just proved (\ref{eqn9.5}) or (\ref{eqn9.6}), 
with the gauge given by (\ref{eqn9.72}), so that 
(\ref{eqn9.7}) (and its consequence (\ref{eqn9.8}) or (\ref{eqn9.9})) 
is also immediate). Finally (9.10) holds because 
we know from Theorem \ref{thm5.1} or \ref{thm8.1}
that $u$ is Lipschitz near $x$ (notice that (\ref{eqn9.65})
preserves the Lipschitz bounds).

Thus the remarks below (\ref{eqn9.9}) apply,
and we have uniform bounds on the $\nabla u_{k}$
in $L^\infty(B(0,R/2))$. We also know that 
$u_{k}(0) = r_{k}^{-1} u(x) = 0$  for all $k$.

By Arzela-Ascoli, 
from the sequence $\{ u_{k} \}$ 
we can extract a subsequence that converges uniformly in $B(0,R/2)$.
Since this is true for every $R > 0$, by a diagonal
argument, we see that there is a subsequence of $\{ u_{k} \}$
that converges uniformly on every compact subset of $\R^n$.
This takes care of the existence of blow-up limits.

Now suppose that we started with a sequence such that
$\{ u_{k} \}$ converges to some limit $u_{\infty}$.
Because of our uniform  bounds on the 
$||\nabla u_{k}||_{L^\infty(B(0,R/2))}$, we see that the 
convergence is uniform on compact sets, i.e., (\ref{eqn9.70})
holds. Also applying Theorem \ref{thm9.1}, 
we get that $u_{\infty}$ is an almost minimizer
on $B(0,R)$. In fact, since our gauge functions $h_{k}$
tend to $0$ uniformly, Remark 9.3 
ensures that $u_{\infty}$ is a minimizer in $B(0,R)$
( with constant functions $q_{\pm}(x)$).
Since this holds for every $R$, $u$ is a minimizer
in the whole space $\R^n$.

Finally, (9.71) for $B(0,R/2)$ is a consequence 
of (\ref{eqn9.13}).
This proves Theorem \ref{thm9.2}.
\qed
\end{proof}

\medskip
\noindent{\bf Remark 9.4.} 
Under the assumptions of 
Theorem \ref{thm9.2}, if the blow-up sequence $\{ u_{k} \}$ converges to a limit  $u_{\infty}$, 
we have the following estimates, which follow from the proof above, the 
estimates (\ref{eqn9.13}), (\ref{eqn9.14}), and (\ref{eqn9.17}),
and the change of variables in (\ref{eqn9.62}):
\begin{equation}
\label{eqn9.73}
\lim_{k \to \infty} \nabla u_{k}^\pm = \nabla u_{\infty}^\pm
\ \text{ in } L^2(B(0,R)) \text{ for every } R>0,
\end{equation}
and, for every ball $B(z,t) \i \R^n$,
\begin{eqnarray}
\label{eqn9.74}
 q_{\pm}(x) \int_{B(z,t)} \chi_{\{ \pm u_{\infty} > 0\}} 
&=& \lim_{k \to \infty}  \int_{B(z,t)} \chi_{\{ \pm u_{k} > 0\}} 
\ q_{\pm}^k 
\nonumber
\\
&=&  \lim_{k \to \infty}  r_{k}^{-n}
\int_{B(x+r_{k}z,r_{k}t)} \chi_{\{ \pm u > 0\}} \ q_{\pm} \,,
\end{eqnarray}
\begin{equation}
\label{eqn9.75}
\int_{B(z,t) } |\nabla u_{\infty}^\pm |^2 
= \lim_{k \to \infty} \int_{B(z,t)} |\nabla u_{k}^\pm |^2
=  \lim_{k \to \infty} \int_{B(x+r_{k}z,r_{k}t)} |\nabla u^{\pm} |^2,
\end{equation}
\begin{equation}
\label{eqn9.76}
J_{z,t}^{\infty}(u_{\infty})
= \lim_{k \to \infty} J^k_{z,t}(u_{k})
= \lim_{k \to \infty} r_{k}^{-n} J_{x+r_{k}z,r_{k}t}(u),
\end{equation}
and similarly for the $J^+$.

\medskip
We conclude this section with a simple consequence
of Theorems \ref{thm9.1} and \ref{thm9.2}, relative to
the functional $\Phi$ from \cite{ACF}. 

\begin{corollary}\label{cor9.1}
Let $u$ be an almost minimizer for $J$ (associated to bounded functions $q_{\pm}$)
in $\Omega$, with a gauge function $h$ such that
(\ref{eqn9.68}) holds. Let $x\in \O$ be such that
$u(x) = 0$ and (\ref{eqn9.69}) holds. Set
\begin{equation}
\label{eqn9.77}
\Phi_{x}(t) = 
{\frac {1}{t^4}} \left(\int_{B(x,t)} 
{\frac {|\nabla u^+(y)|^2}{|x-y|^{n-2}}} dy\right)
\left(\int_{B(x,t)} 
{\frac {|\nabla u^{-}(y)|^2}{|x-y|^{n-2}}} dy\right)
\end{equation}
for $0 < t < \dist(x,\R^n\sm \O)$. 
Then let $u_{\infty}$ be any blow-up
limit of $u$ at $x$, and set
\begin{equation}
\label{eqn9.78}
\Phi(s) = 
{\frac {1}{s^4}} \left(\int_{B(0,s)} 
{\frac {|\nabla u_{\infty}^+(y)|^2}{|y|^{n-2}}} dy\right)
\left(\int_{B(0,s)} 
{\frac {|\nabla u^{-}_{\infty}(y)|^2}{|y|^{n-2}}} dy\right)
\end{equation}
for $s > 0$; 
this is the corresponding function to $u_{\infty}$
 at the origin. Then $\Phi$  is constant, and 
\begin{equation}
\label{eqn9.79}
\Phi(s) = \lim_{t \to  0} \Phi_{x}(t)
\ \text{ for } s > 0.
\end{equation}
\end{corollary}

The fact that we get a constant function $\Phi$
in the statement above is often quite useful, 
especially when this limit is nonzero, but we shall not
worry about this in this section.

\begin{proof}
Let $\{ r_{k} \}$ be a sequence in $(0,+\infty)$ that
tends to $0$, and for which $u_{\infty}$ is the limit
of the $u_{k}$ defined by (\ref{eqn9.65}). We want to show 
that for each $s > 0$,
\begin{equation}
\label{eqn9.80}
\Phi(s) = \lim_{k \to  0} \Phi_{x}(r_{k} s);
\end{equation}
(\ref{eqn9.79}) will follow because we already know
from (\ref{eqn6.5}) that the limit exists.

Fix $s > 0$, and set
\begin{equation}
\label{eqn9.81}
A_{\pm} = \int_{B(0,s)} 
{\frac {|\nabla u_{\infty}^\pm(y)|^2}{|y|^{n-2}}} dy
\ \text{ and } \ 
A_{\pm}^k = \int_{B(0,s)} 
{\frac {|\nabla u_{k}^\pm(y)|^2}{|y|^{n-2}}} dy
\end{equation}
Notice that $\Phi(s) = s^{-4} A_{+}A_{-}$. But also,
\begin{equation}
\label{eqn9.82}
A_{\pm}^k = \int_{B(0,s)} 
{\frac {|\nabla u^\pm(x+r_{k} y)|^2}{|y|^{n-2}}} dy
=
r_{k}^{-n} \int_{B(x,r_{k} s)} 
{\frac {|\nabla u^+(z)|^2}{r_{k}^{2-n} |z-x|^{n-2}}} dy
\end{equation}
by (\ref{eqn9.65}) and where we set $z=x+r_{k}y$.
Hence 
\begin{equation}
\label{eqn9.83}
\Phi_{x}(r_{k} s) = (r_{k} s)^{-4}
\Big( \int_{B(x,r_{k} s)} 
{\frac {|\nabla u^+(z)|^2}{ |z-x|^{n-2}}} dy \Big) \,
\Big( \int_{ B(x,r_{k} s)} 
{\frac {|\nabla u^-(z)|^2}{ |z-x|^{n-2}}} dy \Big)
= s^{-4} A_{+}^k A_{-}^k.
\end{equation}
It is enough to show that 
$\lim_{k \to \infty}  A_{\pm}^k = A_{\pm}$ and that
all these numbers are finite. Let $\eta>0$ be small, then 
\begin{eqnarray}
\label{eqn9.84}
|A_{\pm}-A_{\pm}^k| &\leq& 
\int_{B(0,\eta)} \Big|{\frac {|\nabla u_{\infty}^\pm(y)|^2}{|y|^{n-2}}} 
- {\frac {|\nabla u_{k}^\pm(y)|^2}{|y|^{n-2}}}\Big| dy
+ \int_{B(0,s) \sm B(0,\eta)} \Big|{\frac {|\nabla u_{\infty}^\pm(y)|^2}{|y|^{n-2}}} 
- {\frac {|\nabla u_{k}^\pm(y)|^2}{|y|^{n-2}}}\Big| dy
\nonumber 
\\
&\leq& C \int_{B(0,\eta)} {\frac {1}{|y|^{n-2}}} dy
+ \eta^{2-n} \int_{B(0,s) \sm B(0,\eta)}
\big| |\nabla u_{\infty}^\pm(y)|^2 - |\nabla u_{k}^\pm(y)|^2 \big|
dy
\end{eqnarray}
because $\nabla u_{\infty}$ and the $\nabla u_{k}$
are bounded on $B(0,s)$. The first term can be made
as small as we want by taking $\eta$ small, and for
$\eta > 0$ fixed, the second term tends to $0$, because
(\ref{eqn9.73}) ensures that $\nabla u_k^\pm$ converges in 
$L^2(B(0,s))$ to $\nabla u^\pm_{\infty}$. Thus 
$\lim_{k \to \infty}  A_{\pm}^k = A_{\pm}$. 
Note that a similar computation yields that $A_{\pm}^k $ and $A_{\pm}$ are finite.
Corollary \ref{cor9.1} follows.
\qed
\end{proof}

\section{Nondegeneracy near the free boundary}
One of the main results of this section states that if
$u$ is an almost minimizer for $J$ or $J^+$,
and if $q_{+}$ is bounded below away from zero, if $u^+$ is very small
on a small ball $B$, then it actually vanishes on $\frac{1}{4}B$. 
See Theorem \ref{p10.1} below. 
This statement implies various other nondegeneracy properties of $u$ 
near the free boundary $\p \{ u > 0 \}$, that we also prove in this section.
It is reminiscent of some of the estimates that appear in 
\cite{AC} and \cite{ACF}. 
Our proofs use a different approach though. 

We start with a variant of  
the result of result in \cite{AC} 
that states that $u^+$ is subharmonic when $u$ is a minimizer
for $J$ or $J^+$. Note that we do not require more than
$L^\infty$ bounds on $q_{+}$ and $q_{-}$.

\begin{lemma}\label{lem10.1} 
Let $u$ be an almost minimizer for $J$ or $J^+$ in $\O$. 
Let $B(x,r)$ be such that $\overline B(x,r) \i \O$, and let
$u^\ast$ denote the harmonic extension to $B(x,r)$ 
of the restriction of $u^+$ to $\p B(x,r)$. 
Then
\begin{equation}
\label{eqn10.1}
\fint_{B(x,r)} \big[(u-u^\ast)_{+}\big]^2
\leq C \kappa r^{2+\alpha} \big(1+ \fint_{B(x,r)} |\nabla u|^2\big),
\end{equation} 
where $C$ depends only on $n$.
\end{lemma}

\begin{proof}
Let us define a function $v$ by
\begin{equation}
\label{eqn10.2}
\left\{\begin{array}{rll} 
v(y) &= u(y)  &\text{ for } y\in \O \sm B(x,r)
\\
v(z) &=  \min(u(y), u^\ast(y)) &\text{ for } y\in B(x,r).
\end{array}\right.
\end{equation}
We first check that $v \in W^{1,2}(B(x,r))$.
First recall that $u$ and $u^+$ are Lipschitz in $B(x,r)$,
and that $W^{1,2}(B(x,r))$ is stable under minima and maxima.
So it is enough to check that $u^\ast \in W^{1,2}(B(x,r))$.
But 
the argument given in Remark 3.1 guarantees that the energy minimizing extension of 
$u^+$ in $B(x,r)$, which lies in $W^{1,2}(B(x,r))$ by definition, coincides with $u^\ast$.

Note that $v$ is locally Lipschitz in $\O \sm B(x,r)$ because $u$ is. 
Finally, $v$ is continuous across $\p B(x,r)$; 
then $v \in W^{1,2}_{\loc}(\O)$,
for instance by the welding Lemma 14.4 in \cite{D}. 
Since $v=u$ outside of $B(x,r)$, the definition of
almost minimizers yields 
\begin{equation}
\label{eqn10.3}
J_{x,r}(u) \leq (1+\kappa r^\alpha) J_{x,r}(v)
\end{equation}
(or a similar estimate on $J^+$, that would be treated the
same way). Notice that $v(y) = u(y)$ when $u(y) \leq 0$,
because $u^\ast(y) \geq 0$ everywhere. Then
\begin{equation}
\label{eqn10.4}
\int_{B(x,r)} \chi_{\{ v > 0\}} q_{+} +  \chi_{\{ v < 0\}} q_{-}
\leq  \int_{B(x,r)} \chi_{\{ u > 0\}} q_{+} +  \chi_{\{ u < 0\}} q_{-}.
\end{equation}
Also set $V = \big\{ y \in B(x,r) \, ; \, v(y) \neq u(y) \big\}
= \big\{ y \in B(x,r) \, ; \, u^\ast(y)< u(y) \big\}$. Since $\nabla v = \nabla u$ almost everywhere
on $B(x,r) \sm V$ and $\nabla v = \nabla u^\ast$ everywhere
on $V$, by (\ref{eqn10.4})
\begin{equation}
\label{eqn10.5}
J_{x,r}(v)-J_{x,r}(u) \leq \int_{B(x,r)} |\nabla v|^2-|\nabla u|^2
= \int_{V} |\nabla u^\ast|^2-|\nabla u|^2.
\end{equation}
Next we want to show that
\begin{equation}
\label{eqn10.6}
 \int_{V} |\nabla u|^2 = 
\int_{V} |\nabla u^\ast|^2 + \int_{V} |\nabla u-\nabla u^\ast|^2.
\end{equation}
We observed earlier that, by the proof of Remark 3.1, 
$u^\ast$ is also the function in $W^{1,2}(B(x,r))$ which minimizes 
$\int_{B(x,r)} |\nabla v|^2$ under the condition that the trace of 
$v$ on $\p B(x,r)$ be equal to $u^+$. 

Now set $w(y) = \max(u(y),u^\ast(y))$ for $y\in B(x,r)$.
Then $w\in W^{1,2}(B(x,r))$ because  $u$ and $u^\ast$ both 
lie in $W^{1,2}(B(x,r))$, and the trace of $w$ coincides with
$u^\ast$ and $u^+$ on $\p B(x,r)$ because $u$ and $u^+$
are continuous. So $w$ is a competitor in the minimizing
definition of $u^\ast$, and so would be 
$u^\ast + \lambda (w-u^\ast)$ for any $\lambda \in \R$.
By the usual computation, the scalar product
$\int_{B(x,r)} \langle \nabla u^\ast, \nabla(w-u^\ast) \rangle=0$. But $w = u^\ast$ on $B(x,r) \sm V$, so 
$\nabla (w-u^\ast) = 0$ almost everywhere on 
$B(x,r) \sm V$, and we are left with
$0 = \int_{V} \langle \nabla u^\ast, \nabla(w-u^\ast) \rangle
= \int_{V} \langle \nabla u^\ast, \nabla(u-u^\ast) \rangle$;
(\ref{eqn10.6}) follows. An immediate consequence of
(\ref{eqn10.5}) and (\ref{eqn10.6}) is that
$J_{x,r}(v) \leq J_{x,r}(u)$, and now (\ref{eqn10.6}), (\ref{eqn10.5}), 
and (\ref{eqn10.3}) yield
\begin{eqnarray}
\label{eqn10.7}
\int_{V} |\nabla u-\nabla u^\ast|^2 
&=& \int_{V} |\nabla u|^2 - \int_{V} |\nabla u^\ast|^2 
\leq J_{x,r}(u) - J_{x,r}(v) 
\nonumber
\\
&\leq& \kappa r^\alpha J_{x,r}(v)
  \leq C\kappa r^\alpha 
\big(r^n+\int_{B(x,r)} |\nabla u|^2 \big).
\end{eqnarray}
Now we apply Poincar\'e's inequality to the function
$h = (u-u^\ast)_{+}$. Notice that $h\in W^{1,2}(B(x,r))$,
and $h$ is continuous on $\overline B(x,r)$, with
vanishing boundary values. Then we get that
\begin{equation}
\label{eqn10.8}
\fint_{B(x,r)} \big[(u-u^\ast)_{+}\big]^2
\leq C r^2 \fint_{B(x,r)} \big[\nabla[(u-u^\ast)_{+}]\big]^2
\leq C \kappa r^{2+\alpha} 
\big(1+ \fint_{B(x,r)} |\nabla u|^2\big),
\end{equation}
because $u \leq u^\ast$ on $B(x,r) \sm V$, hence 
$\nabla[(u-u^\ast)_{+}] = \chi_{V} (\nabla u-\nabla u^\ast)$.
Lemma \ref{lem10.1} follows.
\qed
\end{proof}

In order to simplify the statements below, let us
decide that $n$, $\alpha$, $\kappa$, and the $\|q_{\pm}\|_{\infty}$
will be referred to as the usual constants.

\begin{theorem}\label{p10.1} 
Let $u$ be an almost minimizer for $J$ or $J^+$ in $\O$. 
For each choice of $\rho_{0} > 0$ and $L \geq 1$, 
there are constants $\eta_{0} > 0$ and $r_{0} > 0$, 
that depend only on $\rho_{0}$, $L$, and the usual constants,
such that if $x\in \O$ and $0  < r \leq r_{0}$
are such that $\overline B(x,r) \i \Omega$,
\begin{equation}
\label{eqn10.9}
\fint_{\p B(x,r)} u^+  \leq r \eta_{0},
\end{equation} 
\begin{equation}
\label{eqn10.10}
\text{$u$ is $L$-Lipschitz on $B(x,r)$,}
\end{equation}
and
\begin{equation}
\label{eqn10.11}
q_{+}(y) \geq \rho_{0}  \ \text{ for } y \in B(x,r),
\end{equation}
then $u(y) \leq 0$ for $y\in B(x,r/4)$. 
\end{theorem}

The same proof will also yield that if $u$ is an almost minimizer
for $J$, $\fint_{\p B(x,r)} u^-  \leq C r \eta_{0}$, (\ref{eqn10.10})
holds, and $q_{-}(y) \geq \rho_{0}$ on $B(x,r)$,
then $u(y) \geq 0$ for $y\in B(x,r/4)$. Notice that for Theorem \ref{p10.1},
we only need the nondegeneracy assumption (\ref{eqn10.11})
on $q_{+}$, and that in the case of $J$ we do not need
to assume that $|\{ u=0 \}| = 0$. 
The Lipschitz assumption (\ref{eqn10.10}) is not an issue
because we proved earlier that $u$ is locally Lipschitz.
We specify the constant to be used 
as a normalization factor.

\begin{proof}
Let us first prove that if $B(x,r)$ is as in the statement, 
and if $r_{0}$ is small enough (depending on $L$
and $\eta_{0}$ in particular), then
\begin{equation}
\label{eqn10.12}
u(y)  \leq 4^{n+1}\eta_{0} r \ \text{ for } y\in B(x,r/2).
\end{equation}
Otherwise there is  $y\in B(x,r/2)$ such that 
$u(y)  \geq 4^{n+1}\eta_{0} r$. By (\ref{eqn10.10}),
\begin{equation}\label{eqn10.12A}
u(z) \geq u(y)-\eta_0r\ge (4^{n+1}-1)\eta_0r\ \hbox{ for }\  z\in B(y,\frac{\eta_{0}r}{L}).
\end{equation}
Choose $\eta_0\in (0,\frac{1}{4})$ ; then 
$B(y,\frac{\eta_{0}r}{L})\subset B(x,\frac{3r}{4})$, and since 
$u^\ast(x) = \fint_{\p B(x,r)} u^+ \leq r \eta_{0}$, 
for $z\in B(y,\frac{\eta_{0}r}{L})$ the Poisson formula yields 
\begin{equation}\label{eqn10.12B}
u^\ast(z)=\frac{r^2-|x-z|^2}{\sigma_{n-1}r}
\int_{\p B(x,r)}\frac{u(\zeta)}{|z-\zeta|^n}\, d\zeta
\le 4^n \fint_{\p B(x,r)} u^+\le 4^n\eta_0r
\end{equation}
Combining (\ref{eqn10.12A}) and (\ref{eqn10.12B}) we have
\begin{eqnarray}
\label{eqn10.13}
\int_{B(x,r)} [(u-u^\ast)_{+}]^2 &\ge&
\int_{B(y,\frac{\eta_0r}{L})} [(u-u^\ast)_{+}]^2\\
&\ge &(4^{n+1}-1-4^n)\eta_0^2r^2\left |B(y,\frac{\eta_{0}r}{L})\right|
\geq \eta_{0}^2 r^2  
\left |B(y,\frac{\eta_{0}r}{L})\right|,\nonumber
\end{eqnarray}
which contradicts (\ref{eqn10.8}) if $r_{0}$
(and hence $r$) is small enough.

Let $\varphi$ be a smooth function such that $0 \leq \varphi \leq 1$ 
everywhere, $\varphi(y) = 1$ for $y\in B(x,r/2)$, 
$\varphi(y) = 0$ for $y\in \O \sm B(x,3r/4)$, and  
$|\nabla \varphi| \leq 5 r^{-1}$. Then set
\begin{equation}
\label{eqn10.14}
\left\{\begin{array}{rll} 
v(y) &= [u(y)-4^{n+1}\eta_{0}r \varphi(y)]_{+}  
&\text{ when } u(y) \geq 0
\\
v(y) &=  u(y)  &\text{ when } u(y) < 0.
\end{array}\right.
\end{equation}
Notice that $v \in W^{1,2}_{\loc}(\Omega)$, for instance
because we can write $v =[u-2\eta_{0}r \varphi]_{+} - u_{-}$,
and $u\in W^{1,2}_{\loc}(\Omega)$. Also, $v=u$
outside $B(x,3r/4)$, so $u$ and $v$ have the same
trace on $\p B(x,r)$. Hence, by almost minimality of $u$,
$J_{x,r}(u) \leq (1+\kappa r^\alpha) J_{x,r}(v)$.

Observe that $\{ v < 0\} = \{ u < 0\}$ and $\{ v > 0\} \i \{ u > 0\}$, 
but in addition $v = 0$ on $B(x,r/2)$, by (\ref{eqn10.12}).
So in fact $\{ v > 0\} \i \{ u > 0\} \sm B(x,r/2)$ 
and by (\ref{eqn10.11})  
\begin{eqnarray}
\label{eqn10.15}
\int_{B(x,r)} [\chi_{\{ v > 0\}} q_{+} +  \chi_{\{ v < 0\}} q_{-}]
&\leq&  \int_{B(x,r)} [\chi_{\{ u > 0\}} q_{+} +  \chi_{\{ u < 0\}} q_{-}]
- \int_{B(x,r/2) \cap \{ u > 0\}} q_{+}
\nonumber
\\
&\leq&  \int_{B(x,r)} [\chi_{\{ u > 0\}} q_{+} +  \chi_{\{ u < 0\}} q_{-}]
- \rho_{0} |B(x,r/2) \cap \{ u > 0\}|.
\end{eqnarray}
Notice that $\nabla v = \nabla u$
almost everywhere on $\{ u < 0 \}$, and 
$|\nabla v| \leq |\nabla (u-4^{n+1}\eta_{0}r\varphi)|$
almost everywhere on $\{ u \geq 0 \}$, so that
\begin{eqnarray}
\label{eqn10.16}
\int_{B(x,r)} |\nabla v|^2 - \int_{B(x,r)} |\nabla u|^2
&\leq&  2\cdot 4^{n+1}\eta_{0}r \int_{B(x,r)} |\nabla u||\nabla \varphi|
+ 4^{2(n+1)}\eta_{0}^2 r^2 \int_{B(x,r)} |\nabla \varphi|^2
\nonumber
\\
&\leq&  10\cdot 4^{n+1}\eta_{0} \int_{B(x,r)} |\nabla u|
+  25\cdot 4^{2(n+1)} 
\eta_{0}^2 \left |B(x,r)\right |
\\
&\leq&  C\eta_{0} (1+L) r^n
\nonumber
\end{eqnarray}
by Cauchy-Schwarz, (\ref{eqn10.10}), and because $\eta_{0}$
will be chosen small. Combining (\ref{eqn10.15}) and (\ref{eqn10.16}) we get that
\begin{equation}
\label{eqn10.17}
J_{x,r}(v)-J_{x,r}(u) 
\leq - \rho_{0} \left |B(x,r/2) \cap \{ u > 0\}\right |
+ C\eta_{0} (1+L) r^n
\end{equation}
and then since $J_{x,r}(u) \leq C (1+L^2) r^n$
\begin{eqnarray}
\label{eqn10.18}
J_{x,r}(u) &\leq& (1+\kappa r^\alpha) J_{x,r}(v)
\nonumber
\\
&\leq& (1+\kappa r^\alpha) J_{x,r}(u) - (1+\kappa r^\alpha) \rho_{0} |B(x,r/2) \cap \{ u > 0\}|
+ C(1+\kappa r^\alpha) \eta_{0} (1+L) r^n
\nonumber
\\
&\leq& J_{x,r}(u) + C \kappa r^\alpha (1+L^2) r^n
- \rho_{0} |B(x,r/2) \cap \{ u > 0\}|
+ C \eta_{0} (1+L) r^n .
\end{eqnarray}
We simplify (\ref{eqn10.18}) to obtain 
\begin{equation}
\label{eqn10.19}
|B(x,r/2) \cap \{ u > 0\}|
\leq C \rho_{0}^{-1} \kappa r^\alpha (1+L^2) r^n
+ C \rho_{0}^{-1} \eta_{0} (1+L) r^n.
\end{equation}
Then by (\ref{eqn10.12})
\begin{eqnarray}
\label{eqn10.20}
\int_{B(x,r/2)} u^+ &\leq& \int_{B(x,r/2) \cap \{ u > 0\}} 4^{n+1}\eta_{0}r
\leq 4^{n+1} \eta_{0}r |B(x,r/2) \cap \{ u > 0\}|
\nonumber
\\
&\leq& C \eta_{0}\rho_{0}^{-1} \kappa r^\alpha (1+L^2) r^{n+1}
+ C \rho_{0}^{-1} \eta_{0}^2 (1+L) r^{n+1}.
\end{eqnarray}
Pick $z\in B(x,r/4)$. We want to find a decreasing sequence
of radii $r_{j} > 0$, such that the balls $B(z,r_{j})$ satisfy the hypotheses of Theorem \ref{p10.1}. First use Chebyshev to find
$r_{1} \in (r/8,r/4)$ such that
\begin{eqnarray}
\label{eqn10.21}
r_{1}^{-1} \fint_{\p B(z,r_{1})} u^+ 
&\leq& C r^{-n} \int_{\p B(z,r_{1})} u^+
\leq C r^{-n-1} \int_{\rho=r/8}^{r/4} \int_{\p B(z,r_{1})} u^+
\nonumber
\\
&\leq& C r^{-n-1}\int_{B(x,r/2)} u^+
\nonumber
\\
&\leq& C \eta_{0}\rho_{0}^{-1} \kappa r^\alpha (1+L^2) 
+ C \rho_{0}^{-1} \eta_{0}^2 (1+L) < \eta_{0}
\end{eqnarray}
if $\eta_{0}$ and $r_{0}$ are small enough (depending
on $\rho_{0}$ and $L$ in particular). 

This means that the ball $B(z,r_{1})$ satisfies the assumptions
of Theorem \ref{p10.1}. By the same argument as before, we prove that 
(see (\ref{eqn10.19}))
\begin{equation}
\label{eqn10.22}
|B(z,r_{1}/2) \cap \{ u > 0\}|
\leq C \rho_{0}^{-1} \kappa r_1^\alpha (1+L^2) r_{1}^n
+ C \rho_{0}^{-1} \eta_{0} (1+L) r_{1}^n,
\end{equation}
and
\begin{equation}
\label{eqn10.23}
\int_{B(z,r_{1}/2)} u^+ 
\leq C \eta_{0}\rho_{0}^{-1} \kappa r_{1}^\alpha (1+L^2) r_{1}^{n+1}
+ C \rho_{0}^{-1} \eta_{0}^2 (1+L) r_{1}^{n+1}
\end{equation}
as in (\ref{eqn10.20}). Thus as before we can  choose 
$r_{2} \in (r_{1}/8,r_{1}/4)$ such that the ball $B(z,r_{2})$
also satisfies the assumptions of Theorem \ref{p10.1}. Notice that
we no longer need to change the base point, thus we keep
the point $z\in B(x,r/2)$.

We iterate the argument, and find a sequence of radii
$r_{j}$, with $r_{j} \in (r_{j-1}/8,r_{j-1}/4)$, such 
that
\begin{equation}
\label{eqn10.22A}
|B(z,r_{j}/2) \cap \{ u > 0\}|
\leq C \rho_{0}^{-1} \kappa r_j^\alpha (1+L^2) r_{j}^n
+ C \rho_{0}^{-1} \eta_{0} (1+L) r_{j}^n,
\end{equation}
and
\begin{equation}
\label{eqn10.23B}
\int_{B(z,r_{j}/2)} u^+  
\leq C \eta_{0}\rho_{0}^{-1} \kappa r_{j}^\alpha (1+L^2) r_{j}^{n+1}
+ C \rho_{0}^{-1} \eta_{0}^2 (1+L) r_{j}^{n+1}.
\end{equation}
 But if $u(z) > 0$, then $u(w) > 0$ in a neighborhood of $w$
because $u$ is continuous, and this contradicts
 (\ref{eqn10.22A}) for $j$ large (again,
if $\eta_{0}$ and $r_{0}$ are small enough). Since $z$
was an arbitrary point of $B(x,r/4)$, we get that $u(z) \leq 0$
on $B(x,r/4)$, as needed for Theorem \ref{p10.1}.
\qed
\end{proof}

Let us now derive a few simple consequences
of Proposition \ref{p10.1}. We shall be interested
in the rough behaviour of the almost minimizer $u$ 
near the free boundary
\begin{equation}
\label{eqn10.24}
\Gamma = \Omega \cap \p\big(\{ x\in \O \, ; \, u(x)>0 \}\big).
\end{equation} 
Notice that Theorem \ref{p10.1} ensures that 
$\frac{1}{r} \fint_{\p B(x,r)} u^+ \geq \eta_{0}$ 
when $x\in \Gamma$ and $r \leq r_{0}$,
where $\eta_{0}$ and $r_{0}$ depend on local
bounds for the Lipschitz constant for $u$ and local lower
bounds for $q_{+}$.

\begin{lemma}\label{lem10.2} 
Let $u$ be an almost minimizer for $J$ or $J^+$ in $\O$.
For each choice of $\rho_{0} > 0$ and $L \geq 1$, 
let $\eta_{0} > 0$ and $r_{0} > 0$ be as in 
Theorem \ref{p10.1}. Then if $x\in \Gamma$ and 
$0  < r \leq r_{0}$ are such that 
$\overline B(x,r) \i \Omega$, and
(\ref{eqn10.10}) and (\ref{eqn10.11}) hold,
then there exists $y\in \p B(x,r/2)$ such that
\begin{equation}
\label{eqn10.25}
u(y) \geq \frac{\eta_{0}r}{2}
\ \text{ and hence } \ 
u(z) > \frac{\eta_{0}r}{4} \text{ for } z \in B(y, \frac{\eta_{0} r}{4L}).
\end{equation} 
\end{lemma}

\begin{proof}
Let $B(x,r)$ be as in the statement; by Proposition \ref{p10.1},
applied to $B(x,r/2)$, we get that
$\frac{2}{r} \fint_{\p B(x,r/2)} u^+ \geq \eta_{0}$.
By Chebyshev, we can find $y\in \p B(x,r/2)$ such that
$u(y) \geq \frac{\eta_{0} r}{2}$. But by (\ref{eqn10.10})
$u$ is $L$-Lipschitz on $B(x,r)$, so
$u(z) > \frac{\eta_{0}r}{4}$ on $B(y, \frac{\eta_{0} r}{4L})$, 
as needed.
\qed
\end{proof}

\begin{lemma}\label{lem10.3} 
Let $u$ be an almost minimizer for $J$ or $J^+$ in $\O$,
and let $B(x,r)$ satisfy the assumptions of Lemma \ref{lem10.2}.
Then
\begin{equation}
\label{eqn10.26}
\fint_{B(x,r)} |\nabla u^+| \geq c_{0},
\end{equation} 
where $c_{0}$ depends only on $L$, $\rho_{0}$,
and the usual constants.
\end{lemma}

\begin{proof}
Indeed, let $y \in \p B(x,r/2)$ be as in Lemma \ref{lem10.2};
thus $u(z) \geq \eta_{0} r/4$ for 
$z\in B(y, \frac{\eta_{0} r}{8L})$. On the other hand,
$u(x) = 0$, so $u^+(z) \leq |u(z)| \leq \eta_{0} r/8$
for $z\in B(x, \frac{\eta_{0} r}{8L})$. Set 
$m = \fint_{B(x,r)} u^+$, and apply Poincar\'e's inequality
to $u^+$ in $B(x,r)$; this yields
\begin{eqnarray}
\label{eqn10.27}
\eta_{0} r/8 &\leq& \fint_{B(y, \frac{\eta_{0} r}{8L})} u^+
- \fint_{B(x, \frac{\eta_{0} r}{8L})} u^+
\leq \Big|m- \fint_{B(y, \frac{\eta_{0} r}{8L})} u^+\Big|
+ \Big|m- \fint_{B(x, \frac{\eta_{0} r}{8L})} u^+\Big|
\nonumber
\\
&\leq& \fint _{B(y, \frac{\eta_{0} r}{8L})} |u^+ - m|
+ \fint _{B(x, \frac{\eta_{0} r}{8L})} |u^+ - m|
\nonumber
\\
&\leq& C(\eta_{0},L) \fint _{B(x,r)} |u^+ - m|
\leq C'(\eta_{0},L)r \fint _{B(x,r)} |\nabla u^+|;
\end{eqnarray} 
(\ref{eqn10.26}) and the lemma follow.
\qed
\end{proof}

Next we prove that locally,  if $q_{+}$ is bounded below away from zero then the function $u^+$ is equivalent 
to the distance to the zero set, which we denote by
\begin{equation}
\label{eqn10.28}
\delta(y) = \dist\big(y,\big\{ z \in \Omega \, ; \, u(z) = 0 \big\}\big).
\end{equation} 
Notice that $\delta(y) = \dist(y,\Gamma)$ when
$u(y) > 0$ and $\delta(y) < \dist(y,\p \Omega)$.
In fact if $z\in \{ u=0 \}$ minimizes
the distance to $y$, then $z\in \Gamma$ because 
$u > 0$ on $[y,z)$.

\begin{theorem}\label{p10.2} 
Let $u$ be an almost minimizer for $J$ or $J^+$ in $\O$. 
For each choice of $\rho_{0} > 0$ and $L \geq 1$, 
there are constants $\eta_{0} > 0$ and $r_{1} > 0$, 
that depend only on $\rho_{0}$, $L$, and the usual constants,
such that if $x\in \Gamma$ and $0  < r \leq r_{1}$
are such that $\overline B(x,r) \i \Omega$, 
and if (\ref{eqn10.10}) and (\ref{eqn10.11}) hold,
then
\begin{equation}
\label{eqn10.29}
u^+(y) \geq \eta_{0}\, \delta(y)/4
\ \text{ for } y\in B(x,r/2) .
\end{equation} 
\end{theorem}

\begin{proof}
We shall be able to keep the same $\eta_{0}$
as in the previous statements, but $r_{1}$
will possibly need to be smaller than $r_{0}$.
Let $B(x,r)$ be as in the statement, and let 
$y \in B(x,r/2)$ be given; since (\ref{eqn10.29})
is obvious when $\delta(y) = 0$, we can assume that
$\delta(y) > 0$. Also, $\delta(y) \leq r/2$ because $u(x)=0$.

Apply Theorem \ref{p10.1} to the ball $B_{1} = B(y,\delta(y)/2)$;
the assumptions are satisfied (if $r_{1} \leq r_{0}$)
because $B_{1} \i B(x,r)$. Since $u(y) \neq 0$, we get that
\begin{equation}
\label{eqn10.30}
\fint_{\p B(y,\delta(y)/2} u^+ \geq \eta_{0} \delta(y)/2.
\end{equation}
Then denote by $u^\ast$ the energy-minimizing
extension of $u_{\vert \p B_1}$ 
to $B_{1}$. As we have seen a few times, 
this is an acceptable competitor for $u$ in $B_{1}$, and since both 
$u$ and $u^\ast$ are positive in $B_{1}$, we get that
\begin{equation}
\label{eqn10.31}
J_{y,\delta(y)/2}(u) - J_{y,\delta(y)/2}(v)
= \int_{B_{1}} |\nabla u|^2 - |\nabla u^\ast|^2
= \int_{B_{1}} |\nabla(u -u^\ast)|^2\ge 0
\end{equation} 
where the last equation comes from the minimizing property
of $u^\ast$ ; 
see for instance (\ref{eqn2.3}). 
The almost minimizing property of $u$ yields
$J_{y,\delta(y)/2}(u) \leq (1+\kappa\delta(y)^\alpha) 
J_{y,\delta(y)/2}(v)$, hence by (\ref{eqn10.31})
\begin{eqnarray}
\label{eqn10.32}
\int_{B_{1}} |\nabla(u -u^\ast)|^2
&=& J_{y,\delta(y)/2}(u) - J_{y,\delta(y)/2}(v)
\leq \kappa \delta(y)^\alpha J_{y,\delta(y)/2}(v)
\nonumber
\\
&\leq& \kappa \delta(y)^\alpha J_{y,\delta(y)/2}(u)
\leq C \kappa \delta(y)^\alpha (1+L^2) \delta(y)^n. 
\end{eqnarray} 
Moreover since $u^\ast$ is harmonic, $u > 0$
on $\p B_{1}$, and by (\ref{eqn10.30}) we have
\begin{equation}
\label{eqn10.33}
u^\ast(y) = \fint_{\p B_{1}} u =  \fint_{\p B_{1}} u^+
\geq \eta_{0} \delta(y)/2.
\end{equation} 
We want to compare this to $u(y)$. 
First observe that $|u(z) - u(y)| \leq L \delta(y)/2$
for $z\in\p B_{1}$, by (\ref{eqn10.10}), so
$|\nabla u^\ast| \leq C L$ on $\frac{1}{2} B_{1}$,
by an easy estimate with 
the Poisson kernel. Let $\tau > 0$ be small, to be chosen soon, 
and set $B_{2} = B(y, \tau \delta(y))$; then
$|u^\ast(z)-u^\ast(y)| \leq CL\tau \delta(y)$ for
$z\in B_{2}$. Similarly, $|u(z)-u(y)| \leq L\tau \delta(y)$
on $B_{2}$, just by (\ref{eqn10.10}), so
\begin{eqnarray}
\label{eqn10.34}
|u(y)-u^\ast(y)| 
&\leq& |u(y)-u(z)|+|u(z)-u^\ast(z)|+|u^\ast(z)-u^\ast(y)|
\nonumber
\\
&\leq& |u(z)-u^\ast(z)| + CL\tau \delta(y)
\end{eqnarray} 
for $z\in B_{2}$, and now we can apply Poincar\'e's
inequality to the function $u-u^\ast$, which vanishes
at the boundary of $B_1$, 
and get that
\begin{eqnarray}
\label{eqn10.35}
|u(y)-u^\ast(y)| 
&\leq& CL\tau \delta(y) + \fint_{z\in B_{2}} |u(z)-u^\ast(z)|
\nonumber
\\
&\leq& CL\tau \delta(y) + 
\tau^{-n} \fint_{z\in B_{1}} |u(z)-u^\ast(z)|
\nonumber
\\
&\leq& CL\tau \delta(y) + 
C \tau^{-n} \delta(y) \fint_{z\in B_{1}} |\nabla (u-u^\ast)|
\\
&\leq& CL\tau \delta(y) + C \tau^{-n} \delta(y)
[\kappa \delta(y)^\alpha (1+L^2)]^{1/2}
\nonumber
\\
&\leq& CL\tau \delta(y) + C \tau^{-n} \delta(y)
[\kappa r_{1}^\alpha (1+L^2)]^{1/2}
\nonumber
\end{eqnarray} 
by (\ref{eqn10.32}). We choose $\tau$ so small
(depending on $L$ and $\eta_{0}$), 
and then $r_{1}$ so small (depending also on $\tau$)
that (\ref{eqn10.35}) yields $|u(y)-u^\ast(y)| \leq \eta_{0} \delta(y)/4$.
Then (\ref{eqn10.33}) implies that $u(y) \geq \eta_{0} \delta(y)/4$,
as needed for Theorem \ref{p10.2}.
\qed
\end{proof}

\medskip
We now prove that $\big\{ u \leq 0 \big\}$ contains non-tangential balls. 
Our initial lemma is better suited for $J^+$.

\begin{lemma}\label{lem10.4} 
Let $u$ be an almost minimizer for $J^+$ in $\O$.
For each choice of $\rho_{0} > 0$ and $L \geq 1$, 
there exist $\eta_{1} > 0$ and $r_{2} > 0$,
that depend only on $\rho_{0}$, $L$, and the usual
constants, such that if $x\in \Gamma$ and $0  < r \leq r_{2}$
are such that $\overline B(x,r) \i \Omega$ and 
(\ref{eqn10.10}) and (\ref{eqn10.11}) hold,
\begin{equation}
\label{eqn10.36}
\big|\big\{ y\in B(x,r) \, ; \, u(y) = 0 \big\}\big|
\geq \eta_{1} r^n.
\end{equation} 
\end{lemma}


It is not clear to what extent the lower bound (\ref{eqn10.11}) 
on $q_{+}$ is necessary for this lemma. It might be possible
that Lemma \ref{lem10.4} holds with (\ref{eqn10.11}) replaced
by a bound on the modulus of continuity for $q_{+}$.



\begin{proof}
Let $B(x,r)$ be as in the statement.
We shall take $r_{2} \leq r_{0}$, where $r_{0}$ is as in 
Theorem~\ref{p10.1}, thus
\begin{equation}
\label{eqn10.37}
\fint_{\p B(x,r)} u \geq \eta_{0} r.
\end{equation} 
Let $u^\ast$ denote the harmonic extension of the restriction of $u$ to $\p B(x,r)$. 
We know that $u^\ast$ is an acceptable competitor for 
$u$ in $B(x,r)$, and by the usual orthogonality argument
in (\ref{eqn2.3}),
\begin{equation}
\label{eqn10.38}
\int_{B(x,r)} |\nabla u|^2 - \int_{B(x,r)} |\nabla u^\ast|^2
= \int_{B(x,r)} |\nabla (u-u^\ast)|^2
\end{equation} 
Set $X = \big|\big\{ y\in B(x,r) \, ; \, u(y) = 0 \big\}\big|$. 
Since
\begin{equation}
\label{eqn10.39}
\int_{B(x,r)} [\chi_{\{u^\ast > 0\}} - \chi_{\{u> 0\}}] q_{+}
\leq \int_{B(x,r)} \chi_{\{u = 0\}} q_{+}
\leq ||q_{+}||_{\infty} X,
\end{equation} 
we see that
\begin{equation}
\label{eqn10.40}
J_{x,r}(u^\ast) 
\leq J_{x,r}(u) + ||q_{+}||_{\infty} X 
- \int_{B(x,r)} |\nabla (u-u^\ast)|^2
\end{equation} 
and, by the almost minimality condition,
\begin{eqnarray}
\label{eqn10.41}
J_{x,r}(u) - J_{x,r}(u^\ast) 
&\leq& (1+\kappa r^\alpha) J_{x,r}(u^\ast) - J_{x,r}(u^\ast) 
\leq \kappa r^\alpha J_{x,r}(u^\ast)  
\nonumber
\\
&\leq&
 \kappa r^\alpha J_{x,r}(u) + \kappa r^\alpha ||q_{+}||_{\infty} X
\leq C \kappa r^\alpha (1+L^2) r^n
\end{eqnarray}
by (\ref{eqn10.40}) and our Lipschitz bound (\ref{eqn10.10}).
We compare with (\ref{eqn10.40}) and get that
\begin{equation}
\label{eqn10.42}
\int_{B(x,r)} |\nabla (u-u^\ast)|^2
\leq C \kappa r^\alpha (1+L^2) r^n + ||q_{+}||_{\infty} X.
\end{equation} 
We now continue almost as in Theorem \ref{p10.2}.
Observe that 
\begin{equation}
\label{eqn10.43}
u^\ast(x) = \fint_{\p B(x,r)} u^\ast
= \fint_{\p B(x,r)} u \geq \eta_{0} r,
\end{equation}
because $u^\ast$  is harmonic in $B(x,r)$, has the
same trace as $u$ on $\p B(x,r)$, and by (\ref{eqn10.37}).
Then let $\tau > 0$ be small, to be chosen soon, and observe
that $u$ is $L$-Lipschitz on $B(x,r)$,
then $|u(y)| \leq Lr$ on $\p B(x,r)$ (because $u(x)=0$). As in the proof of Theorem \ref{p10.2}
then $u^\ast$ is $CL$-Lipschitz on $B(x,\tau r) \i B(x,r/2)$.
Hence, for $z\in B(x,\tau r)$,
\begin{eqnarray}
\label{eqn10.44}
|u(x)-u^\ast(x)| 
&\leq& |u(x)-u(z)|+|u(z)-u^\ast(z)|+|u^\ast(z)-u^\ast(x)|
\nonumber
\\
&\leq& |u(z)-u^\ast(z)| + CL\tau r
\end{eqnarray} 
and, applying Poincar\'e's inequality to the function 
$u-u^\ast$ which vanishes at the boundary, 
\begin{eqnarray}
\label{eqn10.45}
\eta_{0} r &\leq& |u(x)-u^\ast(x)| 
\leq CL\tau r+ \fint_{z\in B(x,\tau r)} |u(z)-u^\ast(z)|
\\
&\leq& CL\tau r
+ \tau^{-n} \fint_{z\in B(x,\tau r)} |u(z)-u^\ast(z)|\nonumber
\\
&\leq& CL\tau r + 
C \tau^{-n} r \fint_{z\in B(x, r)} |\nabla (u-u^\ast)|
\nonumber
\\
&\leq& CL\tau r
+C \tau^{-n} r \Big\{r^{-n}
\int_{z\in B(x, r)} |\nabla (u-u^\ast)|^2\Big\}^{1/2}\nonumber
\\
&\leq& CL\tau r + C \tau^{-n} r
\big[\kappa r^\alpha (1+L^2)  + r^{-n} ||q_{+}||_{\infty} X
\big]^{1/2}
\nonumber
\end{eqnarray}
by (\ref{eqn10.43}) and (\ref{eqn10.42}). We now choose $\tau$ so small
that $CL\tau  \leq \eta_{0}/2$ in (\ref{eqn10.45}), and 
we are left with
\begin{equation}
\label{eqn10.46}
\big[\kappa r^\alpha (1+L^2) + r^{-n} ||q_{+}||_{\infty} X
\big]^{1/2} \geq C^{-1} \tau^n \eta_{0}/2.
\end{equation} 
We now choose $r_{2}$ so small, depending
on $\tau$ and $\eta_{0}$, that for $r \leq r_{2}$,
(\ref{eqn10.46}) implies that
$r^{-n} ||q_{+}||_{\infty} X \geq \frac{1}{2}[C^{-1} \tau^n \eta_{0}/2]^2$.
Then (\ref{eqn10.36}) holds, with 
$\eta_{1} = C' (\tau^n \eta_{0})^2 ||q_{+}||_{\infty}^{-1}$.
The reader should not worry about $||q_{+}||_{\infty}^{-1}$,
which seems to give a very large bound when $q_{+}$ is small, 
because we also assumed that $q_{+} \geq \rho_{0}$ on $B(x,r)$.
Lemma~\ref{lem10.4} follows.
\qed
\end{proof}

\medskip
Observe that if $u$ us a minimizer for $J$, 
and $u \geq 0$ on $B(x,r)$, the proof of Lemma~\ref{lem10.4}
can be implemented exactly as before, because $u^\ast \geq 0$
on $B(x,r)$ and there is never a contribution of $q_{-}$.
If $u$ takes negative values, then $u^\ast$ could
also take negative values, even on some places where
$u > 0$, and and it could happen that $q_{-}$ is much larger
than $q_{+}$ in those places, and then $u^\ast$ is not such
a great competitor. This case is carefully dealt with in the
next statement.

\begin{lemma}\label{lem10.5} 
Let $u$ be an almost minimizer for $J$ in $\O$.
For each choice of $\rho_{0} > 0$ and $L \geq 1$, 
there exist $\eta_{2} > 0$ and $r_{3} > 0$,
that depend only on $\rho_{0}$, $L$, and the usual
constants, such that if $x\in \Gamma$ and $0  < r \leq r_{3}$
are such that $\overline B(x,r) \i \Omega$, 
(\ref{eqn10.10}) and (\ref{eqn10.11}) hold,
and in addition,
\begin{equation}
\label{eqn10.48}
u(y) \geq 0 \text{ for all } y\in B(x,r), 
\end{equation}
or 
\begin{equation}
\label{eqn10.49}
q_{-}(y) \leq q_{+}(y) \text{ for all } y\in B(x,r), 
\end{equation}
or
\begin{equation}
\label{eqn10.50}
q_{-}(y) \geq \rho_{0} \text{ for all } y\in B(x,r),
\end{equation}
then
\begin{equation}
\label{eqn10.51}
\big|\big\{ z\in B(x,r) \, ; \, u(z) \leq 0 \big\}\big|
\geq \eta_{2} r^n.
\end{equation} 
\end{lemma}

\begin{proof}
We already explained what happens in the first case
when $u \geq 0$ on $B(x,r)$. In the case when (\ref{eqn10.49}) holds, we still
want to use a similar proof, but 
some estimates need to be replaced. We start with (\ref{eqn10.39}). 
We want to check that
\begin{equation}
\label{eqn10.52}
\int_{B(x,r)} \chi_{\{u^\ast > 0\}} \, q_{+} + \chi_{\{u^\ast < 0\}}\, q_{-}
- \int_{B(x,r)} \chi_{\{u > 0\}} \, q_{+} + \chi_{\{u < 0\}} \, q_{-}
\leq \int_{A_{0}} q_{+}
\leq ||q_{+}||_{\infty} X,
\end{equation}
where $X=\big|\big\{z\in B(x,r): u(z)\le 0\big\}\big|$, so we cut 
$B(x,r)$ into the sets 
\begin{eqnarray*}
A_{0} &=& \big\{ z\in B(x,r) \, ; \, u(z) \leq 0 \},\\
A_{1} &= &\big\{ z\in B(x,r) \, ; \,  u(z) > 0 \text{ and } u^\ast(z) > 0\},\\
A_{2} &= &\big\{ z\in B(x,r) \, ; \, u(z) > 0 \text{ and } u^\ast(z) \leq 0\},
\end{eqnarray*}
and estimates their contributions one by one. On $A_0$, we do not know
how large $u^\ast$ is, so we just pay the maxumum $\int_{A_{0}} q_{+}
\leq ||q_{+}||_{\infty} X$. On $A_1$, we integrate both functions against
$q_+$, so the contribution of the difference is zero. On $A_2$,
the contribution of $u$ is larger or equal than the contribution of $u^\ast$, 
because we assumed in (\ref{eqn10.49}) that $q_+ \geq q_-$ on $B(x,r)$. 
This proves (\ref{eqn10.52}).

The conclusion of (\ref{eqn10.52}) is the same as in (\ref{eqn10.39})
(where there was no $q_-$ to worry about) 
and we can continue the argument
up to (\ref{eqn10.43}), which we also need to replace.

Let $z$ be any point of $\p B(x,r)$. 
If $X > \eta_{2} r^n$, then  
(\ref{eqn10.51}) holds by definition. Otherwise, 
if the next 
constant $C$ is large enough, 
$B(x,r) \cap B(z,C\eta_{2}^{1/n}r)$
is not contained in $\{ u \leq 0 \}$, and we can find 
$\xi \in B(x,r)$ such that $|\xi - z| \leq C\eta_{2}^{1/n}r$
and $u(\xi) > 0$. Then $u(z) \geq - CL \eta_{2}^{1/n}r$.
This proves that
\begin{equation}
\label{eqn10.53}
\fint_{\p B(x,r)} u_{-} \leq CL \eta_{2}^{1/n}r.
\end{equation} 
Combining (\ref{eqn10.53}) and (\ref{eqn10.37}), we get that 
\begin{equation}
\label{eqn10.54}
u^\ast(x) = \fint_{\p B(x,r)} u^\ast
= \fint_{\p B(x,r)} u = \fint_{\p B(x,r)} u^+ -  \fint_{\p B(x,r)} u^- 
\geq \eta_{0}r - CL \eta_{2}^{1/n}r
\geq \eta_{0} r/2
\end{equation}
if $\eta_{2}$ is small enough, depending on $L$
and $\eta_{0}$. This is a good enough substitute
for (\ref{eqn10.43}). We may now continue the proof
as we did for Lemma \ref{lem10.4}, and our second case
follows.

In the case when (\ref{eqn10.50}) holds, we 
distinguish between two possibilities. If $u(y) \geq 0$ on
$B(x,r/2)$, we just apply our first case to the ball
$B(x,r/2)$, and get (\ref{eqn10.51}) with a slightly
worse constant. Otherwise, pick $y\in B(x,r/2)$
such that $u(y) < 0$,  and observe that since $u(x)=0$
we can find $z \in [y,x]$ such that $u(z) = 0$ and 
$z \in \p\{ w\in B(x,r) \, ; \, u(w) < 0 \}$.
In other words, $z$ lies on the analogue of the set
$\Gamma$ of (\ref{eqn10.24}), but for the function $-u$.
Then $B(z,r/2)$ satisfies the hypothesis of
Lemma \ref{lem10.2}, which gives a ball of radius
$\eta_{0} r/8L$ which is contained in 
$B(z,r/2) \i B(x,r)$ and  where $u < 0$; (\ref{eqn10.51})
holds in this case also, and Lemma \ref{lem10.5} follows.
\qed
\end{proof}

We may now improve the statement of Lemmas \ref{lem10.4} 
and \ref{lem10.5}, to obtain a non-tangential ball in $\{u\le 0\}\cap B(x,r)$ 
for $x\in \Gamma$. This will only require 
a porosity argument.

\begin{proposition}\label{p10.3} 
Let $u$ be an almost minimizer for $J$ or $J^+$ in $\O$.
For each choice of $\rho_{0} > 0$ and $L \geq 1$, 
there exist $\eta_{3} \in (0,1/3)$ and $r_{3} > 0$,
that depend only on $\rho_{0}$, $L$, and the usual
constants, such that if $x\in \Gamma$ and $B(x,r)$ satisfies the hypotheses
of Lemma \ref{lem10.4} or Lemma \ref{lem10.5},
then there is $y\in B(x,r/2)$ such that
\begin{equation}
\label{eqn10.55}
u(z) \leq 0 \ \text{ for } z \in B(y, \eta_{3} r).
\end{equation} 
\end{proposition}

\medskip
We can see Proposition \ref{p10.3} as an analogue 
of Lemma \ref{lem10.2} for $u^-$. Even when $u$ is a minimizer for $J$,  
we cannot a priori determine whether there is a large ball in $\{u < 0\}$, or $\{u =0\}$.
The proof yields that there is a large ball in the union of these two sets.

\begin{proof}
Our proof will use two large integer parameter
$N$ and $M$, to be chosen later. Let $B(x,r)$
satisfy the hypothesis of Proposition \ref{p10.3}, and 
denote by $Q_{0}$ a cube of diameter $r$
centered at $x$. Cut $Q_{0}$ into $N^d$ 
almost disjoint cubes of diameter $N^{-1} r$
in the natural way. Call $\Delta_{1}$ the set of these
cubes. For $1 \leq k \leq M+1$ define collections $\Delta_{k}$ of cubes, 
in the following inductive way: for $k \geq 2$
if $Q \in \Delta_{k-1}$, cut $Q$ into $N^n$ 
almost disjoint cubes of diameter $N^{-k} r$, which we shall call the
children of $Q$, and denote by $\Delta_{k}$ the collection
of cubes of diameter $N^{-k} r$ obtained this way. For
completeness, set $\Delta_{0} = \{ Q_{0} \}$ and
call the cubes of $\Delta_{1}$ the children of $Q_{0}$.

Set $W = \big\{ z\in B(x,r) \, ; \, u(z) > 0 \big\}$, and
assume that we cannot find $y$ as in the statement
i.e., that
\begin{equation}
\label{eqn10.56}
B(y,\eta_{3} r) \text{ meets } W
\ \text{ for every } y \in B(x,r/2).
\end{equation} 
We show that if $\eta_{3}$ is small, this assumption
contradicts Lemma \ref{lem10.4} or Lemma~\ref{lem10.5}.
Let us  first check that for each $k \leq M$ and each
$Q \in \Delta_{k}$,
\begin{equation}
\label{eqn10.57}
\text{at least one of the children of $Q$ is contained in $W$.} 
\end{equation}
Let $k \leq M$ 
and $Q \in \Delta_{k}$ be given. Choose
a child $R$ of $Q$ that touches the center $x_{Q}$ of
$Q$ (if $N$ is odd, $R$ is unique and the picture is nicer).
Let $B_{0}$ denote the largest ball which is contained in $R$;
its radius is
\begin{equation}
\label{eqn10.58}
 \ell_{0} = \frac{N^{-k-1} r}{\sqrt n}  
\geq  \frac{N^{-M-1} r}{\sqrt n}
\geq \eta_{3} r 
\end{equation}
if $\eta_{3}$ is small enough (depending on $M$ and $N$).
Then $B_{0}$ meets $W$, by (\ref{eqn10.56}). If $R \i W$,
(\ref{eqn10.57}) holds. Otherwise, $R$
meets both $W$ and its complement, so we can find 
$z\in R \cap \Gamma$ (see (\ref{eqn10.24})).
We want to apply Lemma \ref{lem10.2} to the ball 
$B(z,\ell_{1})$, where we set 
\begin{equation}
\label{eqn10.59}
\ell_{1} = \frac{8 L N^{-k-1} r}{\eta_{0}}.
\end{equation}
Notice that $B(z,\ell_{1}) \i Q \i B(x,r)$ for $N$ large enough,
because $Q$ contains the ball of radius $N \ell_{0}$ centered at $x_{Q} \in R$. 
In fact it is enough to choose $N$ so that $8L\eta_0^{-1} +1\le N/\sqrt{n}$. 
Then Lemma \ref{lem10.2} applies, 
and ensure that there is $y\in \p B(z,\ell_{1}/2)$ such that
$B(y,\frac{\eta_{0} \ell_{1}}{4L}) \i W$.
The radius $\eta_0\ell_1/4L=2 N^{-k-1} r$ is twice the diameter
of any cube of $\Delta_{k+1}$. This means that the cube
of $\Delta_{k+1}$ that contains $y$ is contained in $W$.
Thus we conclude that (\ref{eqn10.57}) holds in all cases.

Now we evaluate the measure of $A = Q_{0} \sm W$.
For each $k$, denote by $\Delta'_{k}$  the set of
cubes $Q \in \Delta_{k}$ that meet $A$.
Notice that if $Q \in \Delta_{k}$ does not meet $A$,
then none of its children meets $A$. And if $Q$ meets $A$,
(\ref{eqn10.57}) guarantees that at least one of its children 
does not meet $A$. Thus, if $n_{k}$ denotes the cardinal
of $\Delta'_{k}$, we get that
$n_{k+1} \leq (N^n -1) n_{k}$. Equivalently,
if $S_{k} = \cup_{Q \in \Delta'_{k}} Q$, that
$|S_{k+1}| \leq \frac{N^n -1}{N^n} |S_{k}|$.
After our $M$ steps, we obtain that
\begin{equation}
\label{eqn10.60}
|A| \leq |S_{M+1}| \leq \big(1-\frac{1}{N^n}\big)^M |Q_{0}|.
\end{equation}
We already chose $N$ large (below (\ref{eqn10.59})), and
now we choose $M$ large enough so
$(1-N^{-n})^M\le \min\{\eta_1,\eta_2\}$
 where 
$\eta_{1}$ and $\eta_{2}$ come from
Lemma \ref{lem10.4} and Lemma \ref{lem10.5}.
Then (\ref{eqn10.60}) yields
$|A| \leq \min(\eta_{1},\eta_{2}) r^n$, ,
which combined with Lemma \ref{lem10.4} and Lemma \ref{lem10.5} gives the desired contradiction.
This completes our proof of Proposition \ref{p10.3}.
\qed
\end{proof}

\vskip 1cm

This paper settles the question of the regularity for almost minimizers of the functional $J$, but leaves open problems concerning the structure and the regularity of the corresponding free boundary.
Some of our current work focuses on these issues. We have already obtained partial results in this direction, but they are only first steps toward what we expect to be the optimal result.

\begin{tabular}{l}
Guy David\\
Equipe d'Analyse Harmonique\\
Universit\'e Paris-Sud\\
Batiment 425 \\
91405 Orsay Cedex, France
\\ {\small \tt Guy.David@math.u-psud.fr}
\hfill
\end{tabular}
\begin{tabular}{l}
Tatiana Toro \\ University of Washington \\ Department of Mathematics \\ Box 354350 \\
Seattle, WA 98195-4350 USA
\\ {\small \tt toro@math.washington.edu}\\
\end{tabular}

\end{document}